\documentclass[11pt,a4paper]{article}
\usepackage[T1]{fontenc}
\usepackage[utf8]{inputenc}
\usepackage{lmodern,amsmath,amssymb,amsthm,mathtools}
\usepackage{indentfirst}
\usepackage[margin=28mm]{geometry}
\usepackage{microtype,enumitem}
\usepackage[colorlinks=true,linkcolor=blue,citecolor=blue,urlcolor=blue]{hyperref}
\numberwithin{equation}{section}
\newtheorem{theorem}{Theorem}[section]
\newtheorem{proposition}[theorem]{Proposition}
\newtheorem{lemma}[theorem]{Lemma}

\theoremstyle{remark}\newtheorem{remark}[theorem]{Remark}
\newcommand{\R}{\mathbb R}
\newcommand{\Sph}{\mathbb S}
\newcommand{\dd}{\,\mathrm d}
\newcommand{\Neu}{\mathrm{Neu}}
\newcommand{\e}{\mathrm e}
\newcommand{\kap}{c_N}
\DeclareMathOperator{\supp}{supp}

\title{Directional delays and singular strong-barrier limits
for Fisher--KPP propagation across a semipermeable interface}
\author{Yantao Wang\\[0.3em]
\small\href{mailto:wangyt4455@icloud.com}{\texttt{wangyt4455@icloud.com}}}
\date{}
\begin{document}
\maketitle
\begin{abstract}
We study Fisher--KPP propagation from compactly
supported data across a planar semipermeable
interface in $N\ge2$ dimensions, with unit
diffusion and $f'(0)=1$.
For every finite resistance, the receiving front
has speed $2$ and logarithmic correction
$-(N+2)\log t/2$. Its phase is determined by a
positive directional amplitude on the closed
receiving hemisphere. This amplitude is jointly
continuous in resistance and direction and is
H\"older continuous in direction with exponent
$1/3$. It also determines a threshold-independent
signed arrival-time difference.
Under concavity, the amplitude decreases strictly
with resistance.
As the resistance $\beta$ tends to infinity,
the amplitude is asymptotic to
$2(4\pi)^{-N/2}I_N(\theta)(\log\beta)/\beta$
uniformly on each fixed receiving cone bounded
away from the equator. Here $I_N$ is a weighted
input moment of the closed Neumann donor problem.
Although this moment diverges logarithmically
near the equator, the equatorial amplitude tends
to half the amplitude of the even Neumann donor
solution. The strong-resistance and equatorial
limits therefore do not commute.
The proof combines critical derivative moments,
an integrated Robin trace estimate, Cartesian
leading-edge limits, and matching through the
wave displacement. The phase and asymptotic formulas
require only the smooth KPP assumptions, without
the additional restriction $f'\le f'(0)$.
Concavity is used for strict resistance ordering
and the strict bounds on the limiting equatorial delay.
\end{abstract}

\section{Introduction and main results}\label{sec:intro}
A semipermeable barrier can change the position of
an invading front without changing its asymptotic
speed. We study this effect for Fisher--KPP
propagation from compactly supported initial data
across a planar interface. Our aim is to determine
the resulting directional phase and arrival-time
difference, and to describe their behavior as
the interface resistance tends to infinity.

Barrier-induced delays arise naturally in
population genetics. Pi\'alek and Barton~\cite{PB}
studied the spread of an advantageous allele
across a local barrier to gene flow, including
the effects of random drift and selection against
heterozygotes. In the deterministic case of an
allele favored when rare, they found that the
delay grows logarithmically with barrier strength.
This motivates a quantitative analysis of
barrier-induced front displacements in
multidimensional Fisher--KPP models.

The associated transmission law also appears in deterministic
population-genetic interface problems. Chen~\cite{Chen2006}
and Wang and Su~\cite{WangSu2022} study fixed interfaces in
bounded habitats, including comparison, well-posedness and
long-time dynamics. For reaction--diffusion systems with
membrane conditions, Ciavolella and Perthame~\cite{CP2021}
establish global weak solutions under mass-control and growth
assumptions. These works provide the interface-model background;
their bounded-domain results do not directly yield the
unbounded-domain, resistance-uniform front estimates used here.

For homogeneous propagation, Aronson and Weinberger~\cite{AW1978}
provide the classical spreading and traveling-wave framework.
Bramson~\cite{Bramson1983} identifies the logarithmic correction
and limiting wave in one dimension. Hamel, Nolen, Roquejoffre
and Ryzhik~\cite{HNRR2013} give a PDE proof of the logarithmic
correction for smooth KPP reactions, while Nolen, Roquejoffre
and Ryzhik~\cite{NRR2017} prove convergence to a single translate
for the one-dimensional logistic equation with compact
perturbations of step data. In higher dimensions, the
front-location analysis of G\"artner~\cite{Gartner1982}
precedes the angular-phase description discussed below.
These results motivate our centering and matching arguments;
the interface estimates and parameter-uniform limits are
proved in the present paper.

For the homogeneous equation with unit diffusion
and $f'(0)=1$, the reference front position is
\begin{equation}\label{eq:center}
 m(t)=2t-k\log t,\qquad k=\frac{N+2}{2}.
\end{equation}
For the logistic reaction, Roquejoffre, Rossi and
Roussier-Michon~\cite{RRR} identify the remaining
$O(1)$ displacement through a Lipschitz continuous
angular phase. Their result provides the
homogeneous reference for our directional phase
analysis. The work of Mytnik, Roquejoffre and
Ryzhik~\cite{MRR} on the dependence of the
one-dimensional front shift on small initial
data is also relevant: in the strong-resistance
regime, the receiving population is generated
by a weak input through the interface.

A one-dimensional version of the barrier problem
is studied in an unpublished manuscript by
Wang~\cite{Wang}. Its initial density equals one
sufficiently far into the donor region and
vanishes near the barrier and throughout the
receiving region. For general KPP reactions,
that work establishes the receiving-front phase,
a threshold-independent arrival-time delay,
and the logarithmic and log-log terms in the
strong-resistance expansion. It also proves
resistance monotonicity under concavity.
We build on its reflected average--difference
comparison, expanding-interval energy method,
and logistic comparison for weak interface input,
and prove the versions needed here.

Other interface laws lead to related propagation
problems. Hamel, Lutscher and Zhang~\cite{HLZ}
study spreading and blocking in a one-dimensional
two-patch model with KPP and bistable reactions.
In their rescaled formulation, the density is
continuous and the derivatives satisfy
$u_x(0^-)=\sigma u_x(0^+)$, whereas the present
model allows a resistance-dependent density jump.
Berestycki, Rossi and Tellini~\cite{BRT} study
reaction--diffusion equations in adjacent domains
coupled through Robin exchange conditions.
In their two-half-space geometry, our interface
law corresponds to unit diffusion coefficients
and equal exchange rates $\mu=\nu=\beta^{-1}$.
Their results concern positive steady states and
spreading speeds along the interface; here we
determine the directional phase and arrival-time
difference beyond the spreading speed.

For compactly supported data in $N\ge2$ dimensions,
we prove that the directional amplitude is
positive and jointly continuous in resistance
and direction, including at the equator.
Wave convergence holds uniformly on each fixed
receiving cone and, separately, along the interface.
The strong-resistance limit has three distinctive
features: an explicit leading amplitude
coefficient on fixed receiving cones, a logarithmic
singularity of the weighted input as the direction
approaches the equator, and a positive limiting
equatorial amplitude. The resulting
strong-resistance and equatorial limits do not
commute. The phase and asymptotic formulas hold
under the smooth KPP assumptions without
$f'\le f'(0)$; concavity is used for strict
resistance ordering and the strict bounds on
the limiting equatorial delay.

\subsection{The interface problem}
Let $N\ge2$ and write $x=(x',x_N)\in\R^{N-1}\times\R$. The planar interface
\[
 \Sigma=\{x_N=0\}
\]
separates the donor half-space $\Omega_-=\{x_N<0\}$ from the receiving half-space $\Omega_+=\{x_N>0\}$. For a finite positive interface resistance $\beta$, let $u_\pm$ denote the population densities in $\Omega_\pm$. Throughout the paper, we write $u$ and $u_\pm$ for a fixed resistance, suppressing their dependence on $\beta$. We retain the superscript in $u^\beta$ and $u_\pm^\beta$ only when distinguishing solutions at different resistances, including parameter limits and extrema over resistance intervals. Statements of uniformity in $\beta$ also apply when this superscript is suppressed. With unit diffusion on both sides and a common reaction term $f$, these densities satisfy
\begin{equation}\label{eq:model}
\begin{cases}
 \partial_tu_\pm=\Delta u_\pm+f(u_\pm),
       &x\in\Omega_\pm,\ t>0,\\
 \partial_Nu_-=\partial_Nu_+,&x\in\Sigma,\ t>0,\\
 u_+-u_-=\beta\,\partial_Nu_\pm,&x\in\Sigma,\ t>0.
\end{cases}
\end{equation}
Here $\partial_N=\partial_{x_N}$, so both interface derivatives are taken in the direction $e_N$, from the donor side toward the receiving side. The first interface condition expresses continuity of the diffusive flux, while the second relates this flux to the density jump. Specifically, the flux in the direction $e_N$ is
\[
 -\partial_Nu_-
 =-\partial_Nu_+
 =\frac{u_--u_+}{\beta}.
\]
Thus $\beta$ measures the interface resistance, and $\beta^{-1}$ is the exchange rate. We also consider the two endpoint cases: $\beta=0$ corresponds to a fully permeable interface, with continuity of both the density and its normal derivative, and hence to the homogeneous equation on $\R^N$; $\beta=\infty$ corresponds to an impermeable interface, with independent homogeneous Neumann conditions on the two half-spaces.

Throughout the paper, the reaction term satisfies the normalized smooth KPP assumptions
\begin{equation}\tag{K}\label{ass:K}
 f\in C^2([0,1]),\quad f(0)=f(1)=0,\quad f'(0)=1,\quad f'(1)<0,
 \qquad 0<f(s)\le s\quad(0<s<1).
\end{equation}
For the resistance-ordering results, we additionally assume concavity:
\begin{equation}\tag{C}\label{ass:C}
 f''(s)\le0\qquad(0\le s\le1).
\end{equation}
All remaining results require only \eqref{ass:K}; in particular, they do not require $f'(s)\le f'(0)$ on $[0,1]$. The normalization $f'(0)=1$ is removed in Section~\ref{sec:scaling}.

It is convenient to introduce the reaction deficit $g(s)=s-f(s)$. By \eqref{ass:K} and Taylor's theorem, there exists a constant $C_f>0$, depending only on $f$, such that
\begin{equation}\label{eq:basic-g}
 0\le g(s)\le C_fs^2,\qquad |g'(s)|\le C_fs,
 \qquad 0\le s\le1.
\end{equation}

For every resistance parameter, we prescribe the same nontrivial initial density, compactly supported in the donor half-space. We extend $u_0$ by zero to $\R^N$ and impose
\begin{equation}\label{eq:data}
\begin{gathered}
 u_0\in C_c^\infty(\Omega_-),\qquad 0\le u_0\le1,\qquad u_0\not\equiv0,\\
 u(0,x)=u_0(x).
\end{gathered}
\end{equation}
Thus the receiving half-space is initially empty, and the support of the initial population lies at a positive distance from the interface. We use $u$ to denote the function defined by $u_\pm$ in $\Omega_\pm$, specifying the relevant trace whenever a value on $\Sigma$ is needed.

Let $U_*$ denote the decreasing traveling-wave profile of the homogeneous equation with minimal speed $2$. We fix its translation by prescribing its leading-edge asymptotics:
\begin{equation}\label{eq:wave}
\begin{gathered}
 U_*''+2U_*'+f(U_*)=0,\qquad
 U_*(-\infty)=1,\qquad U_*(+\infty)=0,\\
 U_*(z)\sim z\e^{-z}\qquad(z\to+\infty).
\end{gathered}
\end{equation}
Existence, strict monotonicity and uniqueness up to translation
of this minimal-speed wave, together with its positive
$z\e^{-z}$ leading coefficient, are classical under
Assumption~\ref{ass:K}; see \cite{AW1978} and the discussion in
\cite[Section~1]{HNRR2013}. This normalization fixes the
reference profile used to define the directional phases below.
The sharper tail expansions, including derivatives, are
proved in Lemma~\ref{lem:wave-tails}.

We denote the open hemisphere of receiving directions by
\[
 \Sph_+=\{\theta\in\Sph^{N-1}:\theta_N>0\}.
\]
For $0<\delta\le1$, let
\[
 \Gamma_\delta=\{\theta\in\Sph^{N-1}:\theta_N\ge\delta\},
\]
so that $\Gamma_\delta$ consists of receiving directions bounded away from the interface. The equator $\partial\Sph_+$ consists of the tangential directions $(\eta,0)$, where $\eta\in\Sph^{N-2}$; for $N=2$, these are the two directions corresponding to $\Sph^0$. We also set $\kap=(4\pi)^{-N/2}$.

\subsection{Directional phases and parameter dependence}
\begin{theorem}[Directional phase and arrival-time difference]\label{thm:phase}
Assume \eqref{ass:K} and \eqref{eq:data}. For every $0\le\beta<\infty$, there exists a positive continuous directional amplitude $a_\beta$ on $\overline{\Sph_+}$. Define the directional phase by
\[
 \xi_\beta(\theta)=\log a_\beta(\theta).
\]
For every $0<B<\infty$, there exists $C_B>0$ such that
\begin{equation}\label{eq:holder-main}
 |a_\beta(\theta)-a_\beta(\omega)|
 \le C_B|\theta-\omega|^{1/3},
 \qquad 0\le\beta\le B,\quad\theta,\omega\in\overline{\Sph_+}.
\end{equation}
For every $0<\delta\le1$ and $0<Z<\infty$, the receiving front satisfies
\begin{equation}\label{eq:cone-wave}
 \sup_{\theta\in\Gamma_\delta,\ |z|\le Z}
 \left|u_+(t,(m(t)+z)\theta)-U_*(z-\xi_\beta(\theta))\right|\longrightarrow0
 \qquad(t\to\infty).
\end{equation}
Along the interface, both traces converge to the same profile:
\begin{equation}\label{eq:eq-wave}
 \sup_{\eta\in\Sph^{N-2},\ |z|\le Z}
 \left|u_\pm(t,(m(t)+z)\eta,0)-U_*(z-\xi_\beta(\eta,0))\right|\longrightarrow0
 \qquad(t\to\infty).
\end{equation}
This equatorial convergence also holds in bounded physical neighborhoods transverse to each equatorial ray. Both convergence statements are uniform for $\beta\in[0,B]$, for every finite $B>0$.

For $R>0$, $\theta\in\overline{\Sph_+}$, and a threshold $\upsilon\in(0,1)$, define the first arrival time on the receiving side by
\[
 T_\upsilon^\beta(R,\theta)=\inf\{t>0:u_+(t,R\theta)\ge\upsilon\},
\]
where the receiving trace is used when $\theta_N=0$. Then
\begin{equation}\label{eq:arrival}
 T_\upsilon^\beta(R,\theta)-T_\upsilon^0(R,\theta)\longrightarrow
 \tau_\beta(\theta):=\frac12\log\frac{a_0(\theta)}{a_\beta(\theta)}
 \qquad(R\to\infty).
\end{equation}
The limiting arrival-time difference is independent of $\upsilon$. The convergence is uniform for $\beta\in[0,B]$, for $\upsilon$ in compact subsets of $(0,1)$, and for directions in each fixed $\Gamma_\delta$, or separately on the equator.
\end{theorem}

\begin{theorem}[Joint continuity and strict resistance ordering]\label{thm:parameters}
For every $0<B<\infty$, the directional amplitude $a_\beta(\theta)$, phase $\xi_\beta(\theta)$, and arrival-time difference $\tau_\beta(\theta)$ are jointly continuous in $(\beta,\theta)$ on $[0,B]\times\overline{\Sph_+}$. If \eqref{ass:C} also holds, then for every fixed pair $0\le\beta_1<\beta_2<\infty$,
\begin{equation}\label{eq:strict-gap}
 \inf_{\theta\in\overline{\Sph_+}}\bigl(a_{\beta_1}(\theta)-a_{\beta_2}(\theta)\bigr)>0,
 \qquad
 \inf_{\theta\in\overline{\Sph_+}}\bigl(\tau_{\beta_2}(\theta)-\tau_{\beta_1}(\theta)\bigr)>0.
\end{equation}
\end{theorem}
Thus, under concavity, increasing the resistance strictly decreases the directional amplitude and strictly increases the arrival-time difference, uniformly in direction for each fixed pair of resistance values. The positive gaps may depend on that pair.

The front convergence above is uniform on each fixed receiving cone and, separately, on the equator. Uniform convergence over the entire closed receiving hemisphere is not asserted.

\subsection{Strong resistance and grazing directions}
To describe the strong-resistance limit, we first introduce the donor population in the absence of exchange across the interface. After reflection into $H=\R^{N-1}\times(0,\infty)$, its density $p_\infty$ solves
\[
\begin{cases}
 \partial_t p_\infty=\Delta p_\infty+f(p_\infty),&(x',y)\in H,\ t>0,\\
 \partial_y p_\infty(t,x',0)=0,&t>0,\\
 p_\infty(0,x',y)=p_0(x',y):=u_0(x',-y),&(x',y)\in H.
\end{cases}
\]
Its even extension
\[
 P(t,x',x_N)=p_\infty(t,x',|x_N|)
\]
solves the homogeneous Fisher--KPP equation on $\R^N$, with initial density $u_0(x',x_N)+u_0(x',-x_N)$. Let $a_\Neu$ denote its directional amplitude, defined relative to the same front center $m(t)$ and normalized profile $U_*$.

For $\theta=(\theta',\theta_N)\in\Sph_+$, define the weighted interface input
\begin{equation}\label{eq:input}
 I_N(\theta)=\int_0^\infty \e^{-2t}\int_{\R^{N-1}}
       \e^{\theta'\cdot x'}p_\infty(t,x',0)\dd x'\dd t.
\end{equation}
This quantity determines the leading amplitude as the resistance tends to infinity.

\begin{theorem}[Strong-resistance asymptotics with explicit coefficient]\label{thm:strong}
The function $I_N$ is positive and continuous on $\Sph_+$. For every $0<\delta\le1$, as $\beta\to\infty$,
\begin{align}
 a_\beta(\theta)&=2\kap I_N(\theta)\frac{\log\beta}{\beta}(1+o(1)),\label{eq:strong-a}\\
 \tau_\beta(\theta)&=\frac12\log\beta-\frac12\log\log\beta+C_N(\theta)+o(1),\label{eq:strong-tau}
\end{align}
uniformly for $\theta\in\Gamma_\delta$, where
\[
 C_N(\theta)=\frac12\log\frac{(4\pi)^{N/2}a_0(\theta)}{2I_N(\theta)}.
\]
\end{theorem}

The next result describes the singular behavior of this coefficient as the direction of propagation approaches the interface.

\begin{theorem}[Grazing singularity]\label{thm:grazing}
For $\eta\in\Sph^{N-2}$ and $0<s<1$, set
\[
 \theta_s=(\sqrt{1-s^2}\,\eta,s).
\]
As $s\downarrow0$,
\begin{equation}\label{eq:grazing-I}
 I_N(\theta_s)=4(4\pi)^{(N-2)/2}a_\Neu(\eta,0)\log(1/s)(1+o(1)),
\end{equation}
and consequently
\begin{equation}\label{eq:grazing-C}
 C_N(\theta_s)=-\frac12\log\log(1/s)
       +\frac12\log\frac{\pi a_0(\eta,0)}{2a_\Neu(\eta,0)}+o(1).
\end{equation}
Both asymptotic formulas hold uniformly for $\eta\in\Sph^{N-2}$.
\end{theorem}

For exactly tangential directions, the amplitude has a positive limit as the resistance tends to infinity, and the arrival-time difference remains finite.

\begin{theorem}[Equatorial saturation and noncommuting limits]\label{thm:saturation}
As $\beta\to\infty$, uniformly for $\eta\in\Sph^{N-2}$,
\begin{equation}\label{eq:saturation}
\begin{aligned}
 a_\beta(\eta,0)&\longrightarrow\frac12a_\Neu(\eta,0),\\
 \tau_\beta(\eta,0)&\longrightarrow
 \tau_*^{\rm eq}(\eta):=\frac12\log\frac{2a_0(\eta,0)}{a_\Neu(\eta,0)}.
\end{aligned}
\end{equation}
If \eqref{ass:C} also holds, then $\tau_\beta(\eta,0)$ increases strictly to this limit as $\beta$ increases, and
\begin{equation}\label{eq:saturation-range}
 0<\tau_*^{\rm eq}(\eta)<\frac12\log2.
\end{equation}
In particular, for each fixed $\eta\in\Sph^{N-2}$, with $\theta_s$ as above,
\begin{align*}
 \lim_{s\downarrow0}\lim_{\beta\to\infty}a_\beta(\theta_s)&=0,
 &\lim_{\beta\to\infty}\lim_{s\downarrow0}a_\beta(\theta_s)&=\tfrac12a_\Neu(\eta,0),\\
 \lim_{s\downarrow0}\lim_{\beta\to\infty}\tau_\beta(\theta_s)&=+\infty,
 &\lim_{\beta\to\infty}\lim_{s\downarrow0}\tau_\beta(\theta_s)&=\tau_*^{\rm eq}(\eta).
\end{align*}
\end{theorem}
In these statements, the amplitudes and arrival-time differences are first defined through the large-time and large-distance limits at each finite resistance. The limit $\beta\to\infty$ is taken only afterward. At $\beta=\infty$, the receiving half-space is completely isolated and its initially zero density remains identically zero.

\subsection{Proof strategy}
The proof begins with a common radial Gaussian
envelope, obtained from a linear supersolution
adapted from Roquejoffre, Rossi and
Roussier-Michon~\cite{RRR}; the coefficient and
comparison estimates are verified in
Appendix~\ref{app:linear}.
Reflecting the donor solution places both
densities on the same half-space, where their
average and difference satisfy Neumann and Robin
boundary conditions. An integrated Robin
dissipation estimate controls the interface
contribution to the critical derivative moment,
yielding uniform moment convergence and angular
H\"older continuity. Cartesian rescaling then
identifies the leading-edge amplitude. At the
equator, the average retains the Neumann mode.
The limiting difference satisfies Dirichlet
conditions in both the longitudinal and normal
variables; semigroup decay and the uniform Gaussian
bound force every entire difference limit to vanish.

Relative mesoscopic estimates connect the
diffusive leading edge to the nonlinear front.
For the radial comparisons, we extend the
expanding-interval energy argument of
\cite[Lemma 4.1]{Wang} to a space-dependent
$O(t^{-1})$ drift in Appendix~\ref{app:energy}.
The resulting coarse bounds control the displacement
$h$ defined by $u=U_*(z-h)$.
Two exponential transforms of $h$ satisfy linear
subsolution inequalities, while their transformed
interface conditions retain positive exchange
coefficients. Boundary-influence and initial-data decay
estimates, together with a common-supersolution
comparison, yield wave convergence on receiving
cones and at the equator. This matching argument
requires no assumption that $f'\le f'(0)$.

Under concavity, the cooperative comparison of
\cite[Proposition 5.1]{Wang}, combined with
persistence of a positive leading edge, gives
strict resistance ordering.
For strong resistance, we use the even receiving
extension with a positive interface source,
as in \cite[Section 7.3]{Wang}.
The directional projections of this interface source
are positive measures,
and the associated critical mass balance gives
total absorption asymptotic to $2I_N/\beta$.
A multidimensional small-seed estimate, with the
logistic scaling cancellation of
\cite[Section 7.2]{Wang}, provides the lower
amplitude bound. An estimate of the centered
absorption moment gives the matching upper bound.
Finally, tangential Gaussian integration yields
the grazing coefficient, while the
resistance-independent equatorial moment tail
gives the equatorial saturation limit.

\section{Comparison, kernels, and radial bounds}\label{sec:prelim}
\subsection{Well-posedness and folding}
For $0<\beta<\infty$, the diffusion operator is associated with the nonnegative closed quadratic form
\begin{equation}\label{eq:form}
 \mathfrak a_\beta(v,v)=\int_{\Omega_-}|\nabla v_-|^2+
 \int_{\Omega_+}|\nabla v_+|^2+\frac1\beta\int_\Sigma|v_+-v_-|^2
\end{equation}
with domain $H^1(\Omega_-)\oplus H^1(\Omega_+)$ in $L^2(\Omega_-)\oplus L^2(\Omega_+)$. Integration by parts recovers the transmission conditions in \eqref{eq:model}. If $\Phi$ is a normal contraction, that is, $\Phi(0)=0$ and $|\Phi(a)-\Phi(b)|\le|a-b|$, then
\[
 \mathfrak a_\beta(\Phi(v),\Phi(v))\le\mathfrak a_\beta(v,v),
\]
where $\Phi$ acts on each component. The trace theorem and the
coercivity of $\mathfrak a_\beta+\|\cdot\|_2^2$ on the form domain
show that the form is closed. The closed-form and contraction
criteria in \cite[Chapters~1--2]{Ouhabaz2005} therefore apply.
Consequently, the associated diffusion semigroup is positivity preserving and contractive on $L^2$ and $L^\infty$.

Extend $f$ from $[0,1]$ to a globally Lipschitz function on $\R$. Choose the extension with $f(0)=f(1)=0$. Its Nemytskii map
is globally Lipschitz on $L^2$, so the semilinear semigroup
theory of \cite[Chapter~6]{Pazy1983} gives a unique mild solution.
Testing with the negative part and with $(u-1)_+$ yields
$0\le u\le1$; these truncations belong to the form domain
at positive times, so this argument does not require the
constant function $1$ to belong to $L^2$. Hence the solution is global and independent of the chosen extension. More generally, testing the difference of two solutions with its positive part and using the Lipschitz continuity of $f$ gives the comparison principle. For each fixed resistance, parabolic regularity yields classical
solutions for positive time up to the interface on each side. At $\beta=0$, the form domain is $H^1(\R^N)$ and the diffusion is homogeneous; at $\beta=\infty$, the jump term is omitted and the two Neumann problems decouple.

To place both densities on $H=\R^{N-1}\times(0,\infty)$, reflect the donor solution and set
\[
 p_\beta(t,x',y)=u_-(t,x',-y),\qquad
 q_\beta(t,x',y)=u_+(t,x',y).
\]
Define their average and difference by
\[
 S_\beta=\frac{p_\beta+q_\beta}{2},\qquad D_\beta=\frac{p_\beta-q_\beta}{2}.
\]
We suppress the subscript $\beta$ when no confusion can arise. Let $n=-e_N$ be the outward unit normal to $H$, and define
\[
 c_g(p,q)=\int_0^1g'(q+s(p-q))\dd s.
\]
Then
\[
 g(p)-g(q)=c_g(p,q)(p-q),\qquad |c_g(p,q)|\le C_f(p+q),
\]
by \eqref{eq:basic-g}. The reflected equations become
\begin{equation}\label{eq:fold}
\begin{cases}
 S_t=\Delta S+S-\dfrac{g(S+D)+g(S-D)}2,&\partial_nS=0,\\[2mm]
 D_t=\Delta D+[1-c_g(p,q)]D,&\partial_nD+\kappa D=0,
\end{cases}\qquad\kappa=2/\beta,
\end{equation}
with the evolution equations posed in $H$ and the boundary conditions on $\partial H$. Initially, $S=D=p_0/2$. At $\beta=0$, the condition on $D$ is Dirichlet; at $\beta=\infty$, it is Neumann.

The folding reduces fixed-parameter boundary regularity to
scalar Neumann and Robin problems with locally Lipschitz
reaction terms. Local parabolic regularity and bootstrapping
apply on positive-time cylinders; see \cite{Lieberman1996}.
The strong maximum principle and boundary point lemma below
are used in this same scalar setting. Estimates uniform as
$\kappa\to\infty$ are proved separately, rather than inferred
from fixed-coefficient boundary regularity.

Since $1-c_g(p,q)$ is bounded, the scalar comparison principle gives $D\ge0$, and hence $p\ge q$, without any sign assumption on $c_g$. For every finite resistance and positive time, both densities are strictly positive. The strong maximum principle gives $D>0$ in $H$; for $0<\beta<\infty$, the boundary point lemma also gives a strictly positive trace of $D$. Finally, the reflected donor density satisfies
\[
 \partial_n p_\beta=-\kappa D_\beta\le0\qquad(0<\beta<\infty).
\]
Comparison with the Neumann solution having the same initial density therefore yields $p_\beta\le p_\infty$. At $\beta=0$, the same inequality follows by comparison of the homogeneous solutions with initial densities $u_0(x)$ and $u_0(x',x_N)+u_0(x',-x_N)$, respectively.

\subsection{Robin kernels and finite-time continuity}
Let
\[
 G_t^d(z)=(4\pi t)^{-d/2}\exp\!\left(-\frac{|z|^2}{4t}\right)
\]
be the heat kernel on $\R^d$. For $0\le\kappa<\infty$, the heat kernel on $(0,\infty)$ with boundary condition $\partial_yv(t,0)=\kappa v(t,0)$ is
\begin{equation}\label{eq:robin}
 r_\kappa(t,y,z)=G_t^1(y-z)+G_t^1(y+z)
 -2\kappa\int_0^\infty\e^{-\kappa a}G_t^1(y+z+a)\dd a.
\end{equation}
Formula~\eqref{eq:robin} is the half-line heat-kernel formula
in \cite[Section~III.C, equation~(3.18)]{BF2005}, after a change
of integration variable. We derive the bounds needed here
explicitly, including their dependence on $\kappa$.
Writing
\[
 r_{\mathrm D}(t,y,z)=G_t^1(y-z)-G_t^1(y+z),\qquad
 r_{\mathrm N}(t,y,z)=G_t^1(y-z)+G_t^1(y+z),
\]
we have $0\le r_{\mathrm D}\le r_\kappa\le r_{\mathrm N}$. The case $\kappa=0$ gives the Neumann kernel, while $r_\kappa\to r_{\mathrm D}$ as $\kappa\to\infty$.

Integration by parts gives the positive representation
\begin{equation}\label{eq:robin-positive}
 r_\kappa(t,y,z)=r_{\mathrm D}(t,y,z)
 +\int_0^\infty\e^{-\kappa a}\frac{y+z+a}{t}G_t^1(y+z+a)\dd a.
\end{equation}
For $\kappa>0$, it follows that
\begin{equation}\label{eq:robin-bound}
 r_\kappa(t,y,z)\le\frac{yz+(y+z)/\kappa+\kappa^{-2}}{t}G_t^1(y-z).
\end{equation}
Indeed,
\[
 r_{\mathrm D}(t,y,z)=G_t^1(y-z)(1-\e^{-yz/t})
 \le\frac{yz}{t}G_t^1(y-z),
\]
and $G_t^1(y+z+a)\le G_t^1(y-z)$ for $y,z,a\ge0$. Integrating $\e^{-\kappa a}(y+z+a)$ then yields the remaining terms.

The corresponding heat kernel on $H$ is
\[
 G_t^{N-1}(x'-z')\,r_\kappa(t,y,z).
\]
Together with the Neumann kernel for the average, formula~\eqref{eq:robin} gives strong $L^2$ continuity of the folded diffusion semigroups with respect to $\beta$. Here $\kappa=2/\beta$, with $\beta=0$ corresponding to the Dirichlet limit and $\beta=\infty$ to the Neumann limit for the difference.

For the fixed initial density in \eqref{eq:data}, Duhamel's formula for \eqref{eq:fold} and Gronwall's inequality therefore give
\[
 u^{\beta_j}\longrightarrow u^\beta
 \quad\text{in }C([0,T];L^2(\Omega_-)\oplus L^2(\Omega_+))
\]
for every $T<\infty$ whenever $\beta_j\to\beta\in[0,\infty]$. The nonlinear terms are Lipschitz on the invariant region $0\le S\pm D\le1$. 

We next strengthen this convergence to uniform
convergence of the densities, including their
separate interface traces.

\begin{lemma}[Finite-time continuity up to the interface]
\label{lem:finite-boundary-continuity}
Let $\beta_j\to\beta\in[0,\infty]$, and let
$0<t_*<T<\infty$. Then
\[
 \sup_{t_*\le t\le T}
 \left(
 \|p_{\beta_j}(t)-p_\beta(t)\|_{L^\infty(\overline H)}
 +\|q_{\beta_j}(t)-q_\beta(t)\|_{L^\infty(\overline H)}
 \right)\longrightarrow0.
\]
At $\beta=\infty$, the limiting pair is
$(p_\infty,0)$. At $\beta=0$, the limiting pair
is obtained by folding the homogeneous solution,
and its two interface traces agree.
\end{lemma}

\begin{proof}
Let $\mathsf T_\kappa(t)$ denote the heat semigroup
on $H$ with Robin coefficient $\kappa$.
We use the endpoint conventions
\[
 \mathsf T_0=\mathsf T_{\mathrm N},
 \qquad
 \mathsf T_\infty=\mathsf T_{\mathrm D}.
\]
The kernel inequalities
$0\le r_\kappa\le r_{\mathrm N}$ imply, uniformly
for $\kappa\in[0,\infty]$,
\[
 \|\mathsf T_\kappa(t)\|_{L^\infty\to L^\infty}\le1,
 \qquad
 \|\mathsf T_\kappa(t)\|_{L^2\to L^\infty}
 \le C_Nt^{-N/4}.
\]
These estimates include evaluation on $y=0$,
using the continuous positive-time representatives.

We first verify operator convergence away from
time zero. For $0<\kappa<\infty$, positivity and
\eqref{eq:robin-positive} give
\begin{align*}
 &\sup_{y\ge0}\int_0^\infty
   |r_\kappa(t,y,z)-r_{\mathrm D}(t,y,z)|\dd z\\
 &\qquad=
 \sup_{y\ge0}
 2\int_0^\infty\e^{-\kappa a}G_t^1(y+a)\dd a
 \le\frac{C}{\kappa\sqrt t}.
\end{align*}
Similarly, \eqref{eq:robin} gives
\begin{align*}
 &\sup_{y\ge0}\int_0^\infty
   |r_{\mathrm N}(t,y,z)-r_\kappa(t,y,z)|\dd z\\
 &\qquad\le
 2\kappa\int_0^\infty\int_0^\infty
       G_t^1(z+a)\dd z\dd a
 \le C\kappa\sqrt t.
\end{align*}
For a finite positive coefficient $\kappa_*$,
the same formula yields
\[
 \sup_{y\ge0}\int_0^\infty
 |r_\kappa(t,y,z)-r_{\kappa_*}(t,y,z)|\dd z
 \le
 \int_0^\infty
 |\kappa\e^{-\kappa a}
       -\kappa_*\e^{-\kappa_*a}|\dd a,
\]
whose right-hand side tends to zero as
$\kappa\to\kappa_*$.
The tangential heat kernel has mass one.
Consequently, for every $0<h<T$ and every
$\kappa_j\to\kappa_*\in[0,\infty]$,
\[
 \sup_{h\le t\le T}
 \|\mathsf T_{\kappa_j}(t)-\mathsf T_{\kappa_*}(t)
       \|_{L^\infty\to L^\infty}
 \longrightarrow0.
\]
In particular, this convergence is uniform up to
the flat boundary.

Write
\[
 \mathbf U_\beta=(S_\beta,D_\beta),\qquad
 \mathsf S_\beta(t)
 =\operatorname{diag}
   \bigl(\mathsf T_{\mathrm N}(t),
         \mathsf T_{2/\beta}(t)\bigr),
\]
with $2/0=\infty$ and $2/\infty=0$.
The folded reaction vector is
\[
 \mathcal F(S,D)
 =\frac12
 \begin{pmatrix}
 f(S+D)+f(S-D)\\
 f(S+D)-f(S-D)
 \end{pmatrix}.
\]
It is bounded and Lipschitz on the common invariant
region $0\le S-D\le S+D\le1$.

Fix $0<h<t_*/2$. For $t\in[t_*,T]$, apply the
mild formula on the short interval $[t-h,t]$:
\[
 \mathbf U_\beta(t)
 =\mathsf S_\beta(h)\mathbf U_\beta(t-h)
 +\int_{t-h}^t
   \mathsf S_\beta(t-s)\mathcal F(\mathbf U_\beta(s))
   \dd s.
\]
Subtract the corresponding formula for $\beta_j$.
The two recent-time integrals have total
$L^\infty$ norm at most $Ch$, uniformly in the
parameters. For the initial term, insert and subtract
$\mathsf S_{\beta_j}(h)\mathbf U_\beta(t-h)$.
The smoothing estimate and operator convergence
then give
\begin{align*}
 &\sup_{t_*\le t\le T}
 \|\mathbf U_{\beta_j}(t)-\mathbf U_\beta(t)\|_\infty\\
 &\quad\le
 C h^{-N/4}
 \sup_{0\le s\le T}
 \|\mathbf U_{\beta_j}(s)-\mathbf U_\beta(s)\|_2\\
 &\qquad+
 \|\mathsf S_{\beta_j}(h)-\mathsf S_\beta(h)
       \|_{L^\infty\to L^\infty}
 \sup_{0\le s\le T}\|\mathbf U_\beta(s)\|_\infty
 +Ch.
\end{align*}
The already established $C([0,T];L^2)$ convergence
makes the first term vanish as $j\to\infty$.
For fixed $h$, operator convergence makes the second
term vanish as well. Thus the limit superior is
bounded by $Ch$. Letting $h\downarrow0$ proves
uniform convergence of the folded pair.
Since $p=S+D$ and $q=S-D$, the asserted convergence
of both densities follows, including at $y=0$.
\end{proof}

The common compact-data Gaussian upper bound also
controls the spatial tails uniformly for
$t\in[t_*,T]$. Combining it with
Lemma~\ref{lem:finite-boundary-continuity} permits
passage to exponentially weighted spatial integrals,
including fixed polynomial weights. The same
argument applies to integrals of either interface
trace over $\Sigma$.

The kernel estimate \eqref{eq:robin-bound} applies to the linear diffusion semigroup. In the nonlinear difference equation, the coefficient $c_g(p,q)$ may be negative, so it cannot be discarded as a nonnegative absorption term. The boundary energy estimate in Section~\ref{sec:trace} provides the nonlinear control needed below.

\subsection{Radial comparison estimates}
A solution of the homogeneous equation that is radial about the origin has equal traces and zero normal derivative on $\Sigma$. It therefore satisfies the interface conditions for every $\beta\in[0,\infty]$. Choose a smooth, compactly supported, radially nonincreasing initial density $U_0\in[0,1]$ with $U_0\ge u_0$. The corresponding homogeneous solution provides a common upper bound for all interface solutions.

For the lower comparison, fix $0<B<\infty$ and a time $t_0>0$. Positivity and Lemma~\ref{lem:finite-boundary-continuity} imply that the two densities, including their traces, have a common positive lower bound on each fixed closed half-ball, uniformly for $0\le\beta\le B$. We may therefore choose a nonzero smooth, radially nonincreasing function $w_0$, supported in a ball centered at the origin, such that
\[
 w_0(x)\le u(t_0,x)\qquad(0\le\beta\le B),
\]
with the inequality understood for both traces on $\Sigma$. If $w$ solves the homogeneous equation with reaction $f$ and initial density $w_0$, comparison gives
\[
 w(t-t_0,x)\le u(t,x)\qquad(t\ge t_0,\ 0\le\beta\le B).
\]
Thus the lower comparison uses a common initial seed and requires neither resistance ordering nor multiplication of a nonlinear solution by a constant.

\begin{lemma}[A uniform Gaussian envelope]\label{lem:envelope}
For every $0<\sigma<1/4$, there exist constants
$C_\sigma,c>0$, independent of $\beta\in[0,\infty]$,
such that
\begin{equation}\label{eq:envelope}
 u(t,x)\le
 C_\sigma(z+t^\sigma)\e^{-z}\e^{-cz^2/t},
 \qquad z=|x|-m(t)\ge0,\quad t\ge2.
\end{equation}
Both interface traces are included in this estimate.
The same bound holds for the even Neumann extension $P$,
after adjusting the constants.
\end{lemma}

\begin{proof}
Choose a compactly supported radial upper datum as above.
The corresponding homogeneous solution dominates every
interface solution, including both traces, independently
of the resistance. Lemma~\ref{lem:linear-envelope} in
Appendix~\ref{app:linear}, together with $f(u)\le u$,
gives \eqref{eq:envelope} for this radial solution and
hence for $u$. Applying the same construction to
a radial datum dominating the initial density of $P$
gives the estimate for $P$.

The construction uses only linear comparison and
the explicit supersolution of Appendix~\ref{app:linear},
independently of the subsequent front-localization
and matching arguments.
\end{proof}

Front localization is established in
Lemma~\ref{lem:localization}, after the scalar radial
matching argument.

\section{An integrated Robin trace estimate}\label{sec:trace}
We next derive an integrated boundary estimate for the folded difference $D_\beta$. This estimate controls the interface contribution in directions approaching the equator and does not require $c_g$ to be nonnegative.

Fix $\eta\in\Sph^{N-2}$, and let $E_\eta:\R^{N-2}\to\eta^\perp\subset\R^{N-1}$ be a linear isometry onto $\eta^\perp$. Introduce the Cartesian coordinates
\begin{equation}\label{eq:cartesian}
 x'=(m(t)+\rho\sqrt t)\eta+\sqrt t E_\eta v,\qquad y=\sqrt tY,
 \quad X=(\rho,v,Y),\quad \tau=\log t.
\end{equation}
When $N=2$, the variable $v$ and the corresponding transverse term are absent. Define
\[
 W(\tau,X)=t^{-1/2}\e^{\rho\sqrt t}D_\beta(t,x',y),\qquad t=\e^\tau.
\]
The transformed equation is
\begin{equation}\label{eq:W}
 W_\tau=\Delta_XW+\frac X2\cdot\nabla_XW+\frac{N+1}{2}W
 -kt^{-1/2}W_\rho-VW,\qquad V(\tau,X)=t c_g(p_\beta,q_\beta),
\end{equation}
where $p_\beta$ and $q_\beta$ are evaluated at $(t,x',y)$. The boundary condition becomes
\[
 \partial_YW=\kappa\sqrt t\,W\qquad(Y=0).
\]
Fix
\[
 \sigma=\frac1{16},\qquad \gamma=\frac14,\qquad a(t)=t^{\gamma-1/2}.
\]
For all sufficiently large $t$, the Gaussian envelope implies
\begin{equation}\label{eq:V-small}
 |V(\tau,X)|\le Ct^{1+\gamma}\e^{-t^\gamma}\qquad\text{when }\rho\ge a(t),
\end{equation}
uniformly in the resistance and the equatorial direction. Indeed,
\[
 |(x',y)|\ge\eta\cdot x'=m(t)+\rho\sqrt t,
\]
so the radial displacement is at least $t^\gamma$ in this region. The estimate follows from $|c_g(p,q)|\le C_f(p+q)$ and \eqref{eq:envelope}.

\begin{lemma}[Boundary dissipation]\label{lem:energytrace}
Fix $0<B<\infty$ and an integer $M>(N+1)/2$. There exist $T_0\ge2$ and $C_{B,M}>0$ such that, for every $T\ge T_0$, $0<\beta\le B$, and $\eta\in\Sph^{N-2}$,
\begin{equation}\label{eq:trace-energy}
 \int_{\log T}^{\log(2T)}\int_{\rho\ge2a(t),\ v\in\R^{N-2}}
 (1+\rho^2+|v|^2)^M W(\tau,\rho,v,0)^2\dd\rho\dd v\dd\tau
 \le C_{B,M}T^{-1/2},\qquad t=\e^\tau.
\end{equation}
Consequently,
\begin{equation}\label{eq:trace-L1}
 \int_{\log T}^{\log(2T)}\int_{\rho\ge2a(t),\ v\in\R^{N-2}}
 (1+\rho)W(\tau,\rho,v,0)\dd\rho\dd v\dd\tau
 \le C_{B,M}T^{-1/4}.
\end{equation}
At $\beta=0$, the boundary trace of $W$ vanishes, so both estimates hold trivially.
\end{lemma}
\begin{proof}
Choose a smooth cutoff $\chi_0$ with $0\le\chi_0\le1$, $\chi_0=0$ on $(-\infty,1]$, and $\chi_0=1$ on $[2,\infty)$. Set
\[
 \chi(\tau,\rho)=\chi_0(\rho/a(t)),\qquad \Psi(X)=(1+|X|^2)^M.
\]
Testing \eqref{eq:W} with $\chi^2\Psi W$, initially with compact outer cutoffs, produces the nonnegative boundary dissipation
\[
 \kappa\sqrt t\int_{Y=0}\chi^2\Psi W^2.
\]
The polynomial weight satisfies
\[
 |\nabla\Psi|^2/\Psi\le C_M\Psi,\qquad |X\cdot\nabla\Psi|\le C_M\Psi.
\]
The potential is controlled by \eqref{eq:V-small}. Moreover, $\partial_\tau\chi$ is uniformly bounded and supported in $a(t)\le\rho\le2a(t)$. The cutoff terms are therefore controlled by the weighted energy and
\[
 Ca(t)^{-2}\int_{a(t)\le\rho\le2a(t),\ Y>0}\Psi W^2\dd X.
\]
To estimate this term, write
\[
 l=m(t)+\rho\sqrt t,\qquad r=\sqrt{l^2+t|(v,Y)|^2},\qquad d=r-l\ge0.
\]
Since $0\le D_\beta\le p_\beta/2$, the Gaussian envelope gives
\[
 W\le C\bigl(\rho+d/\sqrt t+t^{\sigma-1/2}\bigr)\e^{-c\rho^2}\e^{-d}.
\]
On the cutoff layer, $l$ is comparable to $t$, and
\[
 d=\frac{t|(v,Y)|^2}{r+l}\ge c\frac{|(v,Y)|^2}{1+|(v,Y)|}.
\]
Hence integration against any fixed polynomial weight yields
\[
 \int_{\R^{N-2}}\int_0^\infty\Psi W^2\dd Y\dd v
 \le C_M\bigl(a(t)^2+t^{-1}+t^{2\sigma-1}\bigr).
\]
The layer has width $a(t)$, so
\begin{equation}\label{eq:cutoff-error}
 Ca(t)^{-2}\int_{a(t)\le\rho\le2a(t),\ Y>0}\Psi W^2\dd X
 \le C_M\left(a(t)+\frac{t^{-1}}{a(t)}+\frac{t^{2\sigma-1}}{a(t)}\right)\le C_M.
\end{equation}
This proves \eqref{eq:cutoff-error}.

The same envelope gives a uniform bound for
\[
 E(\tau)=\int_{Y>0}\chi^2\Psi W^2\dd X.
\]
For large $\rho$, the estimate
\[
 d\ge c\frac{|(v,Y)|^2}{1+\rho+|(v,Y)|}
\]
combined with $\e^{-c\rho^2}$ controls all polynomial moments and justifies removing the outer cutoffs. After absorbing the gradient cross terms, the energy inequality therefore gives
\[
 E'(\tau)+\kappa\sqrt t\int_{Y=0}\chi^2\Psi W^2\le C_M(E(\tau)+1).
\]
Integrating over $[\log T,\log(2T)]$, whose length is $\log2$, and using the uniform bound on $E$, we obtain
\[
 \int_{\log T}^{\log(2T)}\kappa\sqrt t\int_{Y=0}\chi^2\Psi W^2\dd\tau\le C_M.
\]
Since $\kappa\ge2/B$, $t\ge T$, and $\chi=1$ for $\rho\ge2a(t)$, this proves \eqref{eq:trace-energy}.

Finally, $W\ge0$, and the condition $M>(N+1)/2$ ensures
\[
 \int_0^\infty\int_{\R^{N-2}}\frac{(1+\rho)^2}{(1+\rho^2+|v|^2)^M}\dd v\dd\rho<\infty.
\]
Cauchy--Schwarz in space and logarithmic time now gives \eqref{eq:trace-L1}.
\end{proof}

\begin{lemma}[Local estimate independent of the Robin coefficient]\label{lem:robin-local}
Let $X_0\in\{Y=0\}$ and $0<r<R\le1$. Define
\[
 B_R^+=B_R(X_0)\cap\{Y>0\},\qquad
 Q_R^+=(s_0-R^2,s_0]\times B_R^+,
\]
and
\[
 \Gamma_R=(s_0-R^2,s_0]\times(B_R(X_0)\cap\{Y=0\}).
\]
Suppose that $w\ge0$ is a subsolution of
\[
 w_s\le\Delta w+b\cdot\nabla w+cw\quad\text{in }Q_R^+,
 \qquad \partial_nw+\lambda w\le0\quad\text{on }\Gamma_R,
\]
where $n=-e_Y$, $|b|+c_+\le M$, and $\lambda\ge0$. The boundary coefficient $\lambda$ may depend on space and time and is locally bounded for each subsolution; no uniform upper bound on $\lambda$ is assumed. Then
\begin{equation}\label{eq:robin-local-sup}
 \|w\|_{L^\infty(Q_r^+)}
 \le\frac{C(N,M)}{(R-r)^{(N+2)/2}}\|w\|_{L^2(Q_R^+)}.
\end{equation}
The constant is independent of $\lambda$. If $w$ is continuous up to the flat boundary, the same bound holds on $\Gamma_r$. The conclusion also holds when the Robin inequality is replaced by homogeneous Dirichlet data on the flat boundary.
\end{lemma}
\begin{proof}
Choose a space--time cutoff $0\le\zeta\le1$, equal to one on $Q_r^+$, vanishing at the initial time and near the curved spatial boundary of $Q_R^+$, and satisfying
\[
 |\nabla\zeta|\le\frac{C}{R-r},\qquad
 |\partial_s\zeta|\le\frac{C}{(R-r)^2}.
\]
The cutoff may be nonzero on the flat boundary.

For $p\ge2$, test the differential inequality with $\zeta^2w^{p-1}$. The boundary condition gives
\[
 \int_{\{Y=0\}}\zeta^2w^{p-1}\partial_nw
 \le-\int_{\{Y=0\}}\lambda\zeta^2w^p.
\]
Thus the Robin term contributes nonnegative dissipation after it is moved to the left-hand side. Estimate the drift directly by Young's inequality, without differentiating $b$, and use $c\le c_+\le M$. Setting $v=w^{p/2}$, integrating in time, and taking the supremum of the spatial energy yields
\begin{align}\label{eq:robin-local-energy}
 &\sup_s\int_{B_R^+}\zeta^2v^2
 +\int_{Q_R^+}|\nabla(\zeta v)|^2
 +p\int_{\Gamma_R}\lambda\zeta^2v^2\notag\\
 &\qquad\le Cp^2\int_{Q_R^+}
 \bigl(|\partial_s\zeta|+|\nabla\zeta|^2+(1+M^2)\zeta^2\bigr)v^2.
\end{align}
No upper bound or derivative bound on $\lambda$ enters this estimate. Truncated and regularized power tests justify the calculation for weak subsolutions. Under homogeneous Dirichlet conditions, the boundary contribution vanishes.

Extend $\zeta v$ evenly across $\{Y=0\}$. Since it vanishes near the curved boundary, the extension belongs to $H_0^1(B_R(X_0))$. The parabolic Sobolev inequality therefore applies with a constant depending only on $N$; only the test function is extended, not the equation. With $\chi=1+2/N$, the energy estimate \eqref{eq:robin-local-energy} gives
\[
 \|w\|_{L^{p\chi}(Q_r^+)}
 \le\left[\frac{C(N,M)p^2}{(R-r)^2}\right]^{1/p}\|w\|_{L^p(Q_R^+)}.
\]
Here $R-r\le1$ allows the zeroth-order contribution to be included in the same factor.

Iterate this inequality with
\[
 p_j=2\chi^j,\qquad R_j=r+(R-r)2^{-j},\qquad j=0,1,\ldots,
\]
using $R_j$ as the outer radius and $R_{j+1}$ as the inner radius. Since
\[
 \sum_{j=0}^\infty\frac{2}{p_j}=\frac{N+2}{2},\qquad
 \sum_{j=0}^\infty\frac{j+\log p_j}{p_j}<\infty,
\]
the iteration yields \eqref{eq:robin-local-sup}. Continuity up to the flat boundary gives the trace estimate.
\end{proof}

\subsection{Uniform directional tails of the derivative moment}
For $t>0$ and $\theta\in\Sph^{N-1}$, define the critical derivative moment
\[
 E_\beta(t,\theta)=\e^{-2t}\int_{\R^N}(\theta\cdot x-2t)\e^{\theta\cdot x}u(t,x)\dd x.
\]
The fixed-time Gaussian bound ensures that this integral is finite. Integration by parts on the two half-spaces gives
\begin{align}\label{eq:moment-balance}
 \partial_tE_\beta(t,\theta)={}&-\e^{-2t}\int_{\R^N}(\theta\cdot x-2t)\e^{\theta\cdot x}g(u)\dd x\notag\\
 &-2\theta_N\e^{-2t}\int_\Sigma(1+\theta'\cdot x'-2t)\e^{\theta'\cdot x'}D(t,x',0)\dd x'.
\end{align}
The two terms represent the contributions from nonlinear absorption and the interface, respectively.

The volume contribution is integrable in time, uniformly in direction and resistance. More precisely, \eqref{eq:basic-g} and \eqref{eq:envelope} give, for all sufficiently large $t$,
\begin{equation}\label{eq:absorption-tail}
 \e^{-2t}\int_{\R^N}(1+|\theta\cdot x-2t|)\e^{\theta\cdot x}g(u)\dd x
 \le Ct^{-3/2+2\sigma}(1+\log t),
\end{equation}
uniformly for $\theta\in\Sph^{N-1}$ and $\beta\in[0,\infty]$. To obtain this bound, write $x=r\omega$ and use
\[
 \int_{\Sph^{N-1}}(1+r(1-\theta\cdot\omega))\e^{r\theta\cdot\omega}\dd\omega
 \le C\e^r(1+r)^{-(N-1)/2}.
\]
In the region $r=m(t)+z\ge t$, the radial measure and angular integration give the characteristic factor $t^{-k+(N-1)/2}=t^{-3/2}$ near the front; the remaining polynomial dependence on $z$ is absorbed by the exponential bounds. For $z<0$, use $u\le1$; for $z\ge0$, use $\e^zg(u)\le C(z+t^\sigma)^2\e^{-z}$. The centered factor costs $C(1+|z|+\log t)$. The region $r<t$ is exponentially small.

\begin{proposition}[Uniform moment convergence on the closed receiving hemisphere]\label{prop:moment}
For every $0<B<\infty$, there exists a limiting moment $E_\beta^\infty(\theta)$ such that
\begin{equation}\label{eq:Etail}
 |E_\beta(t,\theta)-E_\beta^\infty(\theta)|\le C_Bt^{-1/4}.
\end{equation}
for all sufficiently large $t$, uniformly for $0\le\beta\le B$ and $\theta\in\overline{\Sph_+}$. The limit is jointly continuous in $(\beta,\theta)$ on $[0,B]\times\overline{\Sph_+}$.

At equatorial directions, the stronger estimate
\[
 |E_\beta(t,(\eta,0))-E_\beta^\infty(\eta,0)|\le Ct^{-\mu}(1+\log t),\qquad \mu=\frac12-2\sigma=\frac38,
\]
holds uniformly for $\eta\in\Sph^{N-2}$ and $\beta\in[0,\infty]$.
\end{proposition}
\begin{proof}
The volume contribution is controlled by \eqref{eq:absorption-tail}. It remains to estimate the interface contribution in \eqref{eq:moment-balance} uniformly as the direction approaches the equator.

On each fixed receiving cone, use $D\le p_\infty/2$ and the linear Neumann Gaussian upper bound; the tilted interface term is exponentially integrable. Near the equator write $\theta=(b\eta,s)$, $b=\sqrt{1-s^2}\ge1/2$, and $\lambda=1-b$. For $s>0$, the change of variables \eqref{eq:cartesian} bounds the absolute interface contribution by
\begin{equation}\label{eq:Itau}
 2s t^{-1+k\lambda}\e^{-2\lambda t}
 \int_{\R\times\R^{N-2}} |1-2\lambda t-bk\log t+b\rho\sqrt t|
       \e^{-\lambda\rho\sqrt t}W(\tau,\rho,v,0)\dd\rho\dd v.
\end{equation}
On $\rho\ge2a(t)$, maximizing the scalar coefficients over $s$ and using $s^2=\lambda(1+b)$ bounds \eqref{eq:Itau} by
\[
 Ct^{-1}\int_{\rho\ge2a(t)}(1+\rho)W(\tau,\rho,v,0)\dd\rho\dd v.
\]
The factor $t^{k\lambda}$ is absorbed into $\e^{-\lambda t}$ for large $t$. Integration over $t\in[T,2T]$, using $\dd t/t=\dd\tau$ and \eqref{eq:trace-L1}, gives $C_BT^{-1/4}$.

For the complementary part $l=\eta\cdot x'\le m+2t^\gamma$, first restrict to $r'=|x'|\le m+4t^\gamma$. Near the front, tangential angular integration contributes $t^{(N-2)/2}$, so the tilted radial factor is $t^{-2+k\lambda}\e^{-2\lambda t}$. For $r'=m+z$, $0\le z\le4t^\gamma$, the remaining integral is bounded by
\[
 C\int_0^{4t^\gamma}(z+t^\sigma)(1+\lambda t+\log t+z)\dd z.
\]
Multiplication by $s$ and maximization give
$Ct^{-5/2+3\gamma}(1+\log t)$. The part $z<0$ is smaller, using $u\le1$. If $r'>m+4t^\gamma$ but $l\le m+2t^\gamma$, then $r'-l\ge2t^\gamma$; the factor $\e^{-b(r'-l)}$ makes this remainder superpolynomially small. Since $\gamma=1/4$, the complementary contribution integrated over $[T,2T]$ is $O(T^{-3/4}\log T)$.

Summing the dyadic bounds and the volume tail with $\sigma=1/16$ proves \eqref{eq:Etail}. At each fixed positive time, parameter continuity and the common Gaussian bound make $E_\beta(t,\theta)$ jointly continuous. Uniform convergence therefore gives joint continuity of the limit. At $s=0$, the interface contribution vanishes identically, and integration of \eqref{eq:absorption-tail} gives the stronger equatorial estimate.
\end{proof}

We finally estimate the angular modulus of the limiting moment. For a unit tangent vector $\zeta\in T_\theta\Sph^{N-1}$, differentiation at fixed time gives
\[
 \partial_\zeta E_\beta(t,\theta)=\e^{-2t}\int_{\R^N}(\zeta\cdot x)(1+\theta\cdot x-2t)
 \e^{\theta\cdot x}u(t,x)\dd x.
\]
The radial envelope and angular Gaussian integration imply
\[
 |\partial_\zeta E_\beta(t,\theta)|\le C\sqrt t.
\]
The transverse and centered longitudinal factors each contribute a factor of order $\sqrt t$, while the tilted mass is of order at most $t^{-1/2}$.

Integrating along a shortest spherical arc in the closed receiving hemisphere and applying \eqref{eq:Etail} at its endpoints yields
\begin{equation}\label{eq:modulus}
 |E_\beta^\infty(\theta)-E_\beta^\infty(\omega)|
 \le C_B\bigl(\sqrt T\,|\theta-\omega|+T^{-1/4}\bigr)
\end{equation}
for all sufficiently large $T$, uniformly for $0\le\beta\le B$. For sufficiently small $d=|\theta-\omega|>0$, choose $T=d^{-4/3}$. Both terms in \eqref{eq:modulus} are then of order $d^{1/3}$. Larger separations are covered by the uniform bound on the limiting moments.

Thus $E_\beta^\infty$ is uniformly H\"older continuous with exponent $1/3$ on the closed receiving hemisphere for bounded resistance intervals. Together with the amplitude identity \eqref{eq:amplitude-moment} proved below, this yields \eqref{eq:holder-main}. The exponent results from balancing the finite-time angular derivative bound against the moment tail estimate; no optimality is claimed.

\section{Cartesian identification of the leading edge}\label{sec:Cartesian}
We now identify the leading-edge profile in Cartesian coordinates. We first treat receiving directions bounded away from the interface; tangential directions require the mixed-boundary analysis in the next subsection.

Fix $\theta\in\Sph_+$, and let $E_\theta:\R^{N-1}\to\theta^\perp$ be a linear isometry onto the transverse space. Set
\[
 x=(m(t)+\rho\sqrt t)\theta+\sqrt tE_\theta V,\qquad
 X=(\rho,V),\qquad\tau=\log t,
\]
and define
\[
 W_u(\tau,X)=t^{-1/2}\e^{\rho\sqrt t}u_+(t,x).
\]
For $\theta$ in a fixed $\Gamma_\delta$, every bounded set of scaled coordinates lies in the receiving half-space for all sufficiently large $t$. The transformed equation is
\begin{equation}\label{eq:Cartesian}
 (W_u)_\tau=\mathcal LW_u-kt^{-1/2}\partial_\rho W_u
       -t\frac{g(u_+)}{u_+}W_u,
 \qquad \mathcal L=\Delta_X+\frac X2\cdot\nabla_X+\frac{N+1}{2}.
\end{equation}
The nonlinear potential tends to zero locally uniformly on $\{\rho>0\}$. Parabolic compactness therefore gives subsequential limits of logarithmic-time translates, defined for all translated times and satisfying
\[
 \partial_sW_\infty=\mathcal LW_\infty\qquad(\rho>0).
\]
The radial envelope gives zero boundary trace at $\rho=0$ and the bound
\begin{equation}\label{eq:limit-bound}
 0\le W_\infty(s,\rho,V)\le C\rho\e^{-c\rho^2}\e^{-|V|^2/4}.
\end{equation}
The transverse factor follows from the expansion
\[
 |x|=m(t)+\rho\sqrt t+\frac{|V|^2}{4}+o(1)
\]
on bounded sets of scaled coordinates.

To identify these entire limits, consider the semigroup generated by $\mathcal L$. Its whole-space kernel is
\[
 K_T(X,Z)=\e^{(N+1)T/2}G_{\e^T-1}^N(\e^{T/2}X-Z).
\]
Odd reflection across $\rho=0$ gives the Dirichlet
kernel
\[
 D_T(X,Z)=K_T(X,Z)-K_T(X,\bar Z),
 \qquad \bar Z=(-Z_\rho,Z_\perp).
\]
We use its explicit formula to control the
large-time limit uniformly over the backward-time
profiles.
Writing $a_T=\e^T-1$, its explicit formula is
\[
 D_T(X,Z)
 =2\e^{(N+1)T/2}(4\pi a_T)^{-N/2}
 \exp\left[
 -\frac{\e^T|X|^2+|Z|^2}{4a_T}
 +\frac{\e^{T/2}X_\perp\cdot Z_\perp}{2a_T}
 \right]
 \sinh\left(
 \frac{\e^{T/2}\rho Z_\rho}{2a_T}
 \right).
\]
Hence, locally uniformly in $X$ and $Z$,
\[
 D_T(X,Z)\longrightarrow
 c_N\rho\e^{-|X|^2/4}Z_\rho
 \qquad(T\to\infty).
\]
Moreover, for every compact
$K\subset\{\rho\ge0\}$ and $T\ge1$,
\[
 0\le D_T(X,Z)\le C_K Z_\rho\e^{C_K|Z|},
 \qquad X\in K,\quad Z_\rho>0.
\]
Indeed, this follows from
$\sinh v\le v\e^v$ for $v\ge0$ and the explicit
Gaussian factors.

By \eqref{eq:limit-bound}, the profiles
$W_\infty(s-T,Z)$ satisfy a common Gaussian bound,
uniformly for $s\in\R$ and $T\ge1$.
Together with the kernel estimate above, this gives
a common integrable majorant for
\[
 D_T(X,Z)W_\infty(s-T,Z)
 \quad\text{and}\quad
 c_N\rho\e^{-|X|^2/4}Z_\rho W_\infty(s-T,Z),
\]
uniformly for $X\in K$, $s\in\R$, and $T\ge1$.
Splitting the $Z$-integral into a fixed ball and its
complement consequently gives
\[
 \sup_{X\in K}\left|
 W_\infty(s,X)
 -c_N\rho\e^{-|X|^2/4}
   \int_{Z_\rho>0}Z_\rho W_\infty(s-T,Z)\dd Z
 \right|\longrightarrow0
\]
as $T\to\infty$, uniformly in $s$.
Here we used the Dirichlet semigroup representation
over $[s-T,s]$, which follows by uniqueness in the
bounded solution class.

The first longitudinal moment
\[
 \mathcal M(s)=
 \int_{\rho>0}\rho W_\infty(s,X)\dd X
\]
is independent of $s$. To see this, test the equation
with $\rho$ times a smooth spatial cutoff.
The adjoint identity $\mathcal L^*\rho=0$ cancels
the interior terms. The boundary terms vanish
because both the test weight and the Dirichlet trace
vanish at $\rho=0$, and the cutoff errors tend to
zero by \eqref{eq:limit-bound}.
Thus
\[
 W_\infty(s,X)
 =A\rho\e^{-|X|^2/4},\qquad A=c_N\mathcal M,
\]
with a coefficient independent of $s$.

We now connect this local limit to the critical
moment of the original transmission solution.
In the receiving Cartesian coordinates, define the
piecewise field on the whole scaled space by
\[
 \widetilde W_{\beta,\theta}(t,X)
 =t^{-1/2}\e^{\rho\sqrt t}
 u\bigl(t,(m(t)+\rho\sqrt t)\theta
                  +\sqrt tE_\theta V\bigr).
\]
Values on the interface are immaterial for the
following volume integrals. The coordinate change
gives the exact identities
\[
 \int_{\R^N}F(X)\widetilde W_{\beta,\theta}(t,X)\dd X
 =t^{1/2}\e^{-2t}
   \int_{\R^N}F(X_t(x))
       \e^{\theta\cdot x}u(t,x)\dd x
\]
and
\[
 E_\beta(t,\theta)
 =\int_{\R^N}
   \left(\rho-\frac{k\log t}{\sqrt t}\right)
   \widetilde W_{\beta,\theta}(t,X)\dd X.
\]
Here $X_t$ denotes the inverse coordinate map.

The proof of Appendix~\ref{app:tightness} applies
to these identities with the full angular domain
and with $u$ in place of $S_\beta$.
Indeed, its estimates use only $0\le u\le1$,
the common radial envelope, and the angle between
$x$ and the chosen unit direction. They do not
use the equatorial orientation or the Neumann
condition of the average. Thus, uniformly for
$\theta\in\Gamma_\delta$ and $0\le\beta\le B$,
\[
 \sup_{t\ge T}\int_{\R^N}
       (1+|\rho|)\widetilde W_{\beta,\theta}\dd X<\infty,
\]
\[
 \lim_{M\to\infty}\limsup_{t\to\infty}
 \int_{|X|>M}(1+|\rho|)
       \widetilde W_{\beta,\theta}\dd X=0,
 \qquad
 \int_{\rho<0}|\rho|
       \widetilde W_{\beta,\theta}\dd X\longrightarrow0.
\]

The donor contribution can also be excluded directly.
For a common compactly supported radial datum $\psi$,
linear comparison gives
\[
 \e^{-2t+\theta\cdot x}u(t,x)
 \le\int_{\R^N}\e^{\theta\cdot a}
       G_t^N(x-a-2t\theta)\psi(a)\dd a.
\]
Since $\theta_N\ge\delta$, the Gaussian center has
normal coordinate at least $2\delta t-C$.
Gaussian tail estimates therefore imply
\[
 \e^{-2t}\int_{\Omega_-}
 (1+|\theta\cdot x-2t|)
 \e^{\theta\cdot x}u(t,x)\dd x
 \le C_\delta(1+t)^2\e^{-c_\delta t}.
\]
This estimate is uniform in the resistance.

On every bounded scaled set, the points lie in
$\Omega_+$ for sufficiently large $t$, and
$\widetilde W_{\beta,\theta}=W_u$ there.
The preceding tightness estimates justify passage
from local convergence to the moment integral.
More explicitly, the logarithmic correction tends
to zero by the uniform mass bound, the negative
longitudinal contribution vanishes, and
\[
 \int_{0<\rho<h}\rho\widetilde W_{\beta,\theta}\dd X
 \le Ch
\]
controls the omitted strip near $\rho=0$.
After removing the large spatial cutoff and then
letting $h\downarrow0$, we obtain
\[
 E_\beta^\infty(\theta)
 =\int_{\rho>0}\rho W_\infty(X)\dd X
 =A(4\pi)^{N/2},
\]
since
\[
 \int_{\rho>0}\rho^2\e^{-|X|^2/4}\dd X=(4\pi)^{N/2}.
\]
Consequently, the directional amplitude is
\begin{equation}\label{eq:amplitude-moment}
 a_\beta(\theta)=c_NE_\beta^\infty(\theta),\qquad\theta\in\Sph_+.
\end{equation}
The next subsection establishes the same identity at the equator.

To prove positivity, we apply
Lemma~\ref{lem:persistence} to the homogeneous
radial lower comparison solution, rather than
to the transmission solution itself.
Fix $0<B<\infty$, and let $w$ and $t_0>0$ be the
common lower comparison constructed in
Section~\ref{sec:prelim}, so that
\[
 w(t-t_0,x)\le u(t,x),
 \qquad t\ge t_0,\quad 0\le\beta\le B.
\]
In its own time variable $s$, the whole-space
solution $w$ satisfies
\[
 w_s=\Delta w+(1-Q_w)w,\qquad
 Q_w(s,x)=\frac{g(w(s,x))}{w(s,x)}.
\]
For $s>0$, positivity makes this quotient
well defined, and \eqref{eq:basic-g} gives
\[
 0\le Q_w\le C_fw\le C_f.
\]
The independently proved radial Gaussian envelope
bounds $Q_w$ in $|x|\ge m(s)$ as required by
Lemma~\ref{lem:persistence}.
At any fixed positive time, $w$ dominates a
nonzero compactly supported seed.
The lemma therefore gives a constant $c_w>0$,
independent of direction, such that its normalized
leading edge has lower bound
$c_w\rho\e^{-\rho^2/4}$.

To account for the insertion time, set $s=t-t_0$
and
\[
 \rho_s=
 \frac{m(t)-m(s)+\rho\sqrt t}{\sqrt s}.
\]
For $\rho$ in any compact subset of $(0,\infty)$,
we have $\rho_s\to\rho$ uniformly and
\[
 \begin{aligned}
 &t^{-1/2}\e^{\rho\sqrt t}
 w\bigl(t-t_0,(m(t)+\rho\sqrt t)\theta\bigr)\\
 &\quad=
 \sqrt{\frac{s}{t}}\,
 \e^{-[m(t)-m(s)]}
 \left[
 s^{-1/2}\e^{\rho_s\sqrt s}
 w\bigl(s,(m(s)+\rho_s\sqrt s)\theta\bigr)
 \right].
 \end{aligned}
\]
Since $m(t)-m(t-t_0)\to2t_0$, comparison yields
\[
 \liminf_{t\to\infty}
 \inf_{\substack{0\le\beta\le B\\
                  \theta\in\Gamma_\delta}}
 t^{-1/2}\e^{\rho\sqrt t}
 u_+^\beta\bigl(t,(m(t)+\rho\sqrt t)\theta\bigr)
 \ge c_w\e^{-2t_0}\rho\e^{-\rho^2/4}.
\]
The estimate is uniform for $\rho$ in compact
subsets of $(0,\infty)$.
Thus every subsequential coefficient satisfies
\[
 A\ge c_w\e^{-2t_0}>0.
\]
The lower-bound constant is independent of
$\delta$, since the radial comparison and the
persistence estimate are independent of direction.
Once the equatorial amplitude is identified below,
continuity extends the same positive lower bound
to the closed receiving hemisphere.

Since the limiting moment determines $A$ uniquely,
all subsequential limits agree. Thus, with
\begin{equation}\label{eq:w}
 w_\beta(\tau,\rho,\theta)=t^{-1/2}\e^{\rho\sqrt t}u_+(t,(m(t)+\rho\sqrt t)\theta),
 \qquad\phi_0(\rho)=\rho\e^{-\rho^2/4},
\end{equation}
we obtain
\begin{equation}\label{eq:cone-profile}
 w_\beta(\tau,\rho,\theta)\longrightarrow a_\beta(\theta)\phi_0(\rho)\qquad(\tau\to\infty).
\end{equation}
The convergence is uniform for $\rho$ in compact subsets of $[0,\infty)$, $\theta\in\Gamma_\delta$, and $0\le\beta\le B$, for every fixed $0<\delta\le1$ and $0<B<\infty$. Uniformity follows by allowing both the direction and resistance to vary in the compactness argument.

The same argument applies to homogeneous solutions with compactly supported initial data, including the even Neumann extension $P$.

\subsection{Equatorial mixed-boundary limits}
We now consider directions tangent to the interface. In the coordinates \eqref{eq:cartesian}, define
\[
 W_S=t^{-1/2}\e^{\rho\sqrt t}S_\beta(t,x',y),\qquad
 W_D=t^{-1/2}\e^{\rho\sqrt t}D_\beta(t,x',y),\qquad
 \Phi_0(X)=\rho\e^{-|X|^2/4}.
\]
The average satisfies a homogeneous Neumann condition at $Y=0$. Its nonlinear source is
\[
 -\frac12t^{1/2}\e^{\rho\sqrt t}\bigl[g(p_\beta)+g(q_\beta)\bigr],
\]
which tends to zero locally uniformly for $\rho>0$. Even reflection across $Y=0$ therefore gives local parabolic compactness of logarithmic-time translates of $W_S$, including at the flat boundary.

For the difference, the Robin coefficient grows with time, so we use estimates independent of that coefficient. Fix $0<B<\infty$, take $\tau_j\to\infty$ and $0<\beta_j\le B$, and allow the equatorial directions $\eta_j$ to vary. Set
\[
 w_j(s,X)=W_{D,\beta_j,\eta_j}(\tau_j+s,X),\qquad
 \lambda_j(s)=\frac{2}{\beta_j}\e^{(\tau_j+s)/2},
\]
where the subscripts indicate the resistance and coordinate direction used to define $W_D$.

Choose nested half-cylinders $Q'\subset Q$, with a positive margin at the initial time and curved spatial boundary, but with flat faces allowed to meet $Y=0$, as in Lemma~\ref{lem:robin-local}. Assume that both cylinders remain a positive distance from $\rho=0$. The radial envelope bounds $w_j$ uniformly on $Q$, and \eqref{eq:W} has coefficients
\[
 b_j=X/2-k\e^{-(\tau_j+s)/2}e_\rho,\qquad c_j=(N+1)/2-V_j,
\]
which are uniformly bounded there, with $V_j\to0$. Applying the $p=2$ case of \eqref{eq:robin-local-energy} gives
\begin{equation}\label{eq:equatorial-local-compactness}
 \int_{Q'}|\nabla w_j|^2+\int_{\Gamma'}\lambda_jw_j^2\le C_Q,
 \qquad \int_{\Gamma'}w_j^2\le C_QB\e^{-\tau_j/2},
\end{equation}
where $\Gamma'$ is the flat face of $Q'$. The constants are independent of $\beta_j$ and $\eta_j$; the bounded range of $s$ is absorbed into $C_Q$.

Interior parabolic compactness gives subsequential convergence away from $Y=0$. The uniform supremum bound also gives
\[
 \int_{Q'\cap\{0<Y<\varepsilon\}}w_j^2\le C_Q\varepsilon.
\]
Combining interior convergence with this boundary-strip estimate yields strong local $L^2$ convergence on half-cylinders. The gradient bound in \eqref{eq:equatorial-local-compactness} gives weak convergence in $L^2_sH^1_X$ on smaller cylinders. Continuity of the trace map and the vanishing boundary $L^2$ norm identify the limiting trace as zero. Testing against functions that vanish on $Y=0$ then passes to the limiting equation without requiring a bound on $\lambda_jw_j$.

Thus every entire difference limit satisfies the Ornstein--Uhlenbeck equation with Dirichlet conditions at $Y=0$ and, by the radial envelope, at $\rho=0$. The average limit satisfies the same equation with Dirichlet data at $\rho=0$ and Neumann data at $Y=0$. For the average limit, even reflection across $Y=0$
reduces the problem to the preceding Dirichlet
half-space classification. Its entire limit therefore
has the form $A\Phi_0$, with $A$ independent of time.

For the difference limit, the kernel is odd in both
the longitudinal variable $\rho$ and the normal
variable $Y$. Its explicit product formula contains
two hyperbolic sine factors. For every compact set
$K$ in the closed quadrant and every $T\ge1$, it
satisfies
\[
 0\le D_T^{DD}(X,Z)
 \le C_K\e^{-T/2}Z_\rho Z_Y\e^{C_K|Z|},
 \qquad X\in K,\quad Z_\rho,Z_Y>0.
\]
The factor $\e^{-T/2}$ follows from the two odd
reflections: each contributes a factor of order
$\e^{-T/2}$, whereas the common kernel prefactor
is of order $\e^{T/2}$.
The radial envelope gives a Gaussian bound for
the entire difference limit, uniform in time.
Applying the semigroup representation over
$[s-T,s]$ therefore yields
\[
 \sup_{X\in K}W_{D,\infty}(s,X)
 \le C_K\e^{-T/2}.
\]
Letting $T\to\infty$ proves
$W_{D,\infty}\equiv0$.

The amplitude $A$ is determined by the exact identity
\begin{equation}\label{eq:cart-moment}
 E_\beta(t,(\eta,0))=2\int_{\R\times\R^{N-2}\times(0,\infty)}
 \left(\rho-\frac{k\log t}{\sqrt t}\right)W_S\dd X.
\end{equation}
The tightness estimates in Appendix~\ref{app:tightness} justify passage to the limit. Since
\[
 2\int_{\rho>0,\ Y>0}\rho^2\e^{-|X|^2/4}\dd X=(4\pi)^{N/2},
\]
we obtain $A=c_NE_\beta^\infty(\eta,0)$. Proposition~\ref{prop:moment} identifies this value with the continuous extension of the receiving-direction amplitude, establishing \eqref{eq:amplitude-moment} also at the equator.

It follows that
\begin{equation}\label{eq:eqprofiles}
 W_S\longrightarrow a_\beta(\eta,0)\Phi_0,\qquad W_D\longrightarrow0\qquad(\tau\to\infty).
\end{equation}
To obtain convergence of the difference up to $Y=0$, first use the preceding identification and compactness to conclude that $w_j\to0$ strongly in $L^2(Q)$. Lemma~\ref{lem:robin-local} then gives, on a smaller half-cylinder,
\[
 \sup_{Q'\cup\Gamma'}w_j\le C_{Q',Q}\|w_j\|_{L^2(Q)}\longrightarrow0.
\]
The estimate is applied directly to the nonnegative difference $w_j$, whose limit has already been identified as zero. The average converges up to the boundary by even reflection. Finally, the radial envelope controls the remaining strip near $\rho=0$.

These arguments give local uniform convergence for $\rho\ge0$ and $Y\ge0$, uniformly in $\eta$ and $0<\beta\le B$. At $\beta=0$, the homogeneous whole-space identification gives the same limits for the folded average and difference. Splitting an arbitrary parameter sequence into its zero- and positive-resistance subsequences therefore extends the uniformity to $0\le\beta\le B$. Since the scaled donor and receiving densities are $W_S+W_D$ and $W_S-W_D$, respectively, both have the same limiting profile at the interface.

\subsection{Relative mesoscopic matching}\label{sec:meso}
The leading-edge limits above are locally uniform in the scaled longitudinal coordinate $\rho$. For inner matching, we need relative asymptotics at the intermediate distance $z=t^\gamma$, where $\rho=t^{\gamma-1/2}\to0$. We establish these asymptotics uniformly for transverse displacements smaller than $\sqrt t$.

Fix $0<\gamma<1/3$, a nonnegative function $H_\perp(t)=o(\sqrt t)$, and a finite resistance bound $B>0$. Choose an auxiliary exponent
\[
 0<\alpha<\min\{\gamma,1/4\},\qquad\eta_t=t^{\alpha-1/2}.
\]
On cylinders with $\eta_t\le\rho\le L$, the radial envelope \eqref{eq:envelope} with exponent $\alpha$ makes the nonlinear terms exponentially small in $t^\alpha$. Set
\[
 \mathcal R=\partial_\tau-\mathcal L+kt^{-1/2}\partial_\rho.
\]
In particular, the folded difference satisfies
\[
 \mathcal RW_D\le\omega(\tau)W_D,\qquad
 0\le\omega(\tau)\le Ct^{1+\alpha}\e^{-t^\alpha}.
\]
This estimate uses the absolute value of the nonlinear coefficient and therefore does not require $c_g\ge0$.

We first treat the equatorial directions. For a small fixed $L>0$, consider the moving cylinder
\[
 \eta_t\le\rho\le L,\qquad |v_i|\le1,\qquad 0\le Y\le1.
\]
Write $V=(v,Y)$ and define
\[
 \psi_L(\rho)=\sin\frac{\pi\rho}{2L},\qquad
 h_L(\rho)=\cos\frac{\pi\rho}{4L},\qquad
 g_L(V)=\sum_{i=1}^{N-1}\frac{\cosh(V_i/L)}{\cosh(1/L)}.
\]
For sufficiently small $L$ and all sufficiently large $t$, direct differentiation gives
\begin{align*}
 \mathcal R\psi_L&\ge\left[\frac{\pi^2}{4L^2}-\frac{N+2}{2}\right]\psi_L,\\
 \mathcal R(\psi_Lg_L)&\ge
 \left[\frac{\pi^2/4-1}{L^2}-\frac1{2L}-\frac{N+2}{2}\right]\psi_Lg_L\ge0,
 \qquad \mathcal Rh_L\ge2h_L.
\end{align*}
Set $A=a_\beta(\eta,0)$ and $Z_S=W_S-A\Phi_0$. The equation for the average and the Gaussian envelope imply
\[
 |\mathcal RZ_S|\le Ct^{-1/2}+Ct^{3/2}\e^{-t^\alpha}
\]
on this cylinder. For $\epsilon>0$, define
\begin{equation}\label{eq:eq-meso-barrier}
 \mathcal B=\epsilon\psi_L+C_0L\psi_Lg_L+
       (K\eta_t+Q\e^{-(\tau-\tau_0)})h_L.
\end{equation}
Choose $C_0$ to control the lateral boundary data, then choose $K,Q$ and a sufficiently large starting time $\tau_0$. These choices make $\mathcal B$ a common upper comparison function for $Z_S$ and $-Z_S$.

Indeed, the radial envelope gives $|Z_S|=O(\eta_t)$ at the moving inner boundary. At $\rho=L$, \eqref{eq:eqprofiles} makes the error arbitrarily small after increasing $\tau_0$. On each transverse lateral face, $g_L\ge1$ and $L\psi_L\ge c\rho$, so the second term controls the lateral error. The term containing $Q$ controls the initial data, while the term containing $K\eta_t$ absorbs the residual. At $Y=0$, both $Z_S$ and $\mathcal B$ have zero normal derivative. Thus comparison yields $|Z_S|\le\mathcal B$.

For $W_D$, use the operator $\mathcal R-\omega(\tau)$. The exponentially small correction preserves the comparison inequalities for sufficiently large $\tau_0$. Moreover,
\[
 \partial_n\mathcal B+\kappa\sqrt t\,\mathcal B=\kappa\sqrt t\,\mathcal B\ge0\qquad(Y=0).
\]
The same construction therefore gives $0\le W_D\le\mathcal B$. Homogeneous Dirichlet comparison covers $\beta=0$.

Evaluate these bounds at
\[
 \rho=t^{\gamma-1/2},\qquad |V|\le H_\perp(t)/\sqrt t.
\]
Since $\eta_t/\rho=t^{\alpha-\gamma}\to0$ and $\e^{-(\tau-\tau_0)}/\rho\to0$, division by $\rho$ gives
\[
 \limsup_{t\to\infty}\frac{\mathcal B}{\rho}\le\frac{\pi\epsilon}{2L}
 +\frac{C_0\pi(N-1)}{2\cosh(1/L)}.
\]
First let $\epsilon\downarrow0$, then $L\downarrow0$. Since $\Phi_0/\rho\to1$ at these points, we obtain
\[
 \frac{W_S}{\rho}\longrightarrow a_\beta(\eta,0),\qquad\frac{W_D}{\rho}\longrightarrow0.
\]
Both folded densities consequently have the same relative asymptotics, giving the equatorial form of \eqref{eq:meso}.

On each fixed receiving cone, the same construction applies to \eqref{eq:Cartesian}, using full transverse intervals and no interface boundary. Hence, at $z=t^\gamma$,
\begin{equation}\label{eq:meso}
 \frac{\e^z}{z}u_+\bigl(t,(m(t)+z)\theta+\zeta\bigr)\longrightarrow a_\beta(\theta),
\end{equation}
uniformly for $0\le\beta\le B$, $\theta\in\Gamma_\delta$, and physical transverse displacements $\zeta\perp\theta$ satisfying $|\zeta|\le H_\perp(t)$. At the equator, the corresponding statement holds for both folded components. The same argument gives uniformity when $z/t^\gamma$ ranges over a fixed compact subset of $(0,\infty)$. These relative estimates provide the outer boundary data for the inner matching argument in the next section.

\section{Inner matching through the displacement field}\label{sec:inner}
We now use the relative mesoscopic estimates to identify the nonlinear front. The argument is formulated in terms of the displacement from the normalized wave $U_*$, retaining both transverse diffusion and the transformed interface exchange. We first establish coarse displacement bounds using scalar radial comparisons.

\subsection{Coarse bounds that precede inner matching}
Define
\[
 \lambda_*=\sqrt{1-f'(1)}-1>0,\qquad
 r_U(s)=\frac{U_*''(s)}{U_*'(s)}.
\]
\begin{lemma}[Wave tails and derivative ratios]\label{lem:wave-tails}
There exist $B_->0$ and $B_+\in\mathbb R$ such that, as $s\to-\infty$,
\begin{align*}
 1-U_*(s)&=B_-\e^{\lambda_*s}+O(\e^{2\lambda_*s}),\\
 U_*'(s)&=-\lambda_*B_-\e^{\lambda_*s}+O(\e^{2\lambda_*s}),\\
 U_*''(s)&=-\lambda_*^2B_-\e^{\lambda_*s}+O(\e^{2\lambda_*s}).
\end{align*}
As $s\to+\infty$,
\begin{align*}
 U_*(s)&=(s+B_+)\e^{-s}+O((1+s)^2\e^{-2s}),\\
 U_*'(s)&=(1-s-B_+)\e^{-s}+O((1+s)^2\e^{-2s}),\\
 U_*''(s)&=(s+B_+-2)\e^{-s}+O((1+s)^2\e^{-2s}).
\end{align*}
In particular, $r_U$ is well defined and
\begin{equation}\label{eq:wave-ratio}
 \|r_U\|_\infty<\infty,\qquad
 r_U(s)\longrightarrow\lambda_*\quad(s\to-\infty),\qquad
 r_U(s)=-1+\frac1s+O(s^{-2})\quad(s\to+\infty).
\end{equation}
\end{lemma}
\begin{proof}
Write $w=1-U_*$, $a=-f'(1)>0$ and $\lambda=\lambda_*$.
Then $w'>0$, $\lambda^2+2\lambda=a$, and
\[
 w''+2w'=f(1-w)=aw+R(w),\qquad R(w)=O(w^2).
\]
First, $v=w'$ is bounded. Indeed, $v'+2v=f(1-w)\ge0$
implies $v(r)\ge\e^{-2(r-s)}v(s)$ for $r\ge s$, whereas
$\int_s^{s+1}v(r)\dd r\le1$; hence
$v(s)\le2/(1-\e^{-2})$. Thus $q=w'-\lambda w$ is bounded
and satisfies $q'+(\lambda+2)q=R(w)$. Variation of constants
from $-\infty$ gives, for all sufficiently negative $s$,
\[
 q(s)=\int_{-\infty}^s\e^{-(\lambda+2)(s-r)}R(w(r))\dd r
      =O(w(s)^2),
\]
where we used monotonicity of $w$. Consequently,
$(\log w)'=\lambda+O(w)\ge\lambda/2$ on a left half-line.
Integrating backwards gives $w(s)\le C\e^{\lambda s/2}$,
so $w$ is integrable at $-\infty$. Since
\[
 \frac{\dd}{\dd s}\log(w(s)\e^{-\lambda s})=O(w(s)),
\]
this logarithm has a finite real limit there. Thus
$w(s)\sim B_-\e^{\lambda s}$ with $B_->0$.
Integrating the last identity once more yields
$w=B_-\e^{\lambda s}+O(\e^{2\lambda s})$.
The identity $w'=\lambda w+q$ and the wave equation give
the asserted first and second derivative expansions.

For the right tail set $V(s)=\e^sU_*(s)$. The normalization
in \eqref{eq:wave} gives $V(s)/s\to1$, while
\[
 V''(s)=\e^s g(U_*(s))\ge0,\qquad
 V''(s)=O((1+s)^2\e^{-s}).
\]
The second estimate makes $V''$ integrable at $+\infty$;
the limiting value of $V'$ must be $1$, since $V(s)/s\to1$.
Therefore
\[
 V'(s)-1=-\int_s^\infty V''(r)\dd r
       =O((1+s)^2\e^{-s}).
\]
This error is also integrable, so $V(s)-s$ has a finite
limit $B_+$ and
$V(s)=s+B_++O((1+s)^2\e^{-s})$.
Using $U_*'=\e^{-s}(V'-V)$ and
$U_*''=\e^{-s}(V''-2V'+V)$ proves the three right-tail
expansions without differentiating an uncontrolled remainder.
Division yields \eqref{eq:wave-ratio} at both ends.
On compact intervals, continuity and $U_*'<0$ bound $r_U$.
\end{proof}

\begin{lemma}[Coarse displacement bounds]\label{lem:coarse}
Fix $0<B<\infty$, $0<\gamma<1/3$, and $\gamma<\delta<1/2$. In the moving Cartesian coordinates
\[
 x=(m(t)+z)\theta+\zeta,\qquad
 \zeta\perp\theta,\qquad |\zeta|\le t^\delta,
\]
define
\[
 h(t,z,\theta,\zeta)=z-U_*^{-1}(u(t,x)).
\]
On each fixed receiving cone $\Gamma_{\delta_0}$, $0<\delta_0\le1$, the following estimates hold for all sufficiently large $t$, uniformly for $0\le\beta\le B$, $\theta\in\Gamma_{\delta_0}$, and the indicated transverse displacements:
\begin{align}
 |h|&\le C,&&0\le z\le t^\gamma,\label{eq:h-ahead}\\
 |h|&\le C_\epsilon+\epsilon|z|,&&-Z\le z\le0,\label{eq:h-behind}
\end{align}
for every $\epsilon>0$ and $Z>0$ in the second estimate. The same bounds hold in equatorial coordinates for each folded component. In \eqref{eq:h-behind}, $C_\epsilon$ is independent of $Z$, although the starting time may depend on $Z$.
\end{lemma}
\begin{proof}
We derive the bounds from scalar comparisons, independently of the multidimensional wave convergence to be proved below.

First, the radial linear supersolution gives
\begin{equation}\label{eq:sharp-upper}
 u(t,x)\le C(1+z_r)\e^{-z_r},\qquad
 0\le z_r:=|x|-m(t)\le2t^\gamma.
\end{equation}
Indeed, apply Lemma~\ref{lem:linear-envelope} in Appendix~\ref{app:linear} with $\sigma=0$. Then $\eta_t=t^{-1/2}$, the inner boundary is $z_r=-1$, and the coefficient bounds give bounded $\zeta_{\rm lin}$ and $Q\asymp\eta_t$. Here $\zeta_{\rm lin}$ is the scalar coefficient in \eqref{eq:linear-Phi}. The comparison accommodates both the inner-boundary data and the Gaussian tail at the starting time. The normalized supersolution is therefore bounded by $C(\rho+\eta_t)$ for $0\le\rho\le1$, which yields \eqref{eq:sharp-upper} after multiplication by $\sqrt t\,\e^{-z_r}$. This construction uses only $f(u)\le u$.

For the lower bound, use the common radial seed and the inequality $f(u)\ge u-C_fu^2$. The auxiliary logistic matching estimate in Appendix~\ref{app:logistic} gives
\begin{equation}\label{eq:sharp-lower}
 u(t,x)\ge c(1+z_r)\e^{-z_r},\qquad 0\le z_r\le2t^\gamma.
\end{equation}
Its weighted matching error is $O(t^{3\gamma-1})$, which tends to zero because $\gamma<1/3$. The fixed time shift used to insert the seed changes only the constants.

Since
\[
 z_r-z=\frac{|\zeta|^2}{|x|+m(t)+z}=O(t^{2\delta-1}),
\]
the bounds \eqref{eq:sharp-upper}--\eqref{eq:sharp-lower}, together with the right-hand wave tail, imply \eqref{eq:h-ahead} for sufficiently large positive $z$. It remains to control bounded values of $z$ and the region behind the front.

For this purpose, consider the common radial comparison solutions with the original reaction $f$. Their leading-edge identification and relative mesoscopic matching use only $g\ge0$ and the radial envelope. On $|z_r|\le t^\gamma$, the radial moving-frame equation has drift
\[
 2+b(t,z_r),\qquad
 b(t,z_r)=-\frac{k}{t}+\frac{N-1}{m(t)+z_r},\qquad |b(t,z_r)|\le C/t.
\]
Let $\mathcal U(t,z)$ denote one of the radial
comparison solutions in its moving coordinate
$z=r-m(t)$, and let $a>0$ be its leading-edge
amplitude. Fix $0<\varepsilon<a$ and put
$R(t)=t^\gamma$.
The relative mesoscopic estimate allows us to
choose $T_\varepsilon$ sufficiently large that,
for every $t\ge T_\varepsilon$,
\[
 (a-\varepsilon)R(t)\e^{-R(t)}
 \le \mathcal U(t,R(t))
 \le (a+\varepsilon)R(t)\e^{-R(t)}.
\]
Increase $T_\varepsilon$ if necessary so that
the prescribed right boundary values lie in $[0,1]$
and $m(t)-R(t)>0$.

On $-R(t)<z<R(t)$, let $q_-$ and $q_+$ solve
\[
 (q_\pm)_t
 =(q_\pm)_{zz}+(2+b(t,z))(q_\pm)_z+f(q_\pm),
\]
with boundary values
\[
 \begin{aligned}
 q_-(t,-R(t))&=0,
 &q_-(t,R(t))&=(a-\varepsilon)R(t)\e^{-R(t)},\\
 q_+(t,-R(t))&=1,
 &q_+(t,R(t))&=(a+\varepsilon)R(t)\e^{-R(t)},
 \end{aligned}
\]
and initial profiles
\[
 q_-(T_\varepsilon,z)=0,\qquad
 q_+(T_\varepsilon,z)=1
 \quad(-R(T_\varepsilon)<z<R(T_\varepsilon)).
\]
The initial profiles need not agree with the
boundary data at the two initial corners;
this is permitted in Lemma~\ref{lem:energy}.
Comparison with the constants $0$ and $1$ gives
$0\le q_\pm\le1$.
The initial and lateral inequalities then yield
\[
 q_-(t,z)\le\mathcal U(t,z)\le q_+(t,z),
 \qquad t>T_\varepsilon,\quad |z|<R(t).
\]

Lemma~\ref{lem:energy}, whose argument is adapted
from \cite[Lemma 4.1]{Wang}, applies because
$|b(t,z)|\le C/t$.
In particular, the drift contribution
\[
 \int V_t\,b(V_z-V)
\]
is estimated by Young's inequality without
differentiating $b$. The resulting errors
\[
 Ct^{2\gamma-2}(\mathcal F+C)+Ct^{3\gamma-2}
\]
are integrable since $\gamma<1/3$.
Thus, for each fixed $\varepsilon$,
\[
 q_\pm(t,z)\longrightarrow
 U_*\bigl(z-\log(a\pm\varepsilon)\bigr)
\]
locally uniformly in $z$ as $t\to\infty$.
The comparison consequently gives
\[
 \begin{aligned}
 U_*\bigl(z-\log(a-\varepsilon)\bigr)
 &\le\liminf_{t\to\infty}\mathcal U(t,z)\\
 &\le\limsup_{t\to\infty}\mathcal U(t,z)
 \le U_*\bigl(z-\log(a+\varepsilon)\bigr).
 \end{aligned}
\]
First let $t\to\infty$ with $\varepsilon$ fixed,
and then let $\varepsilon\downarrow0$.
Continuity of the translated wave, uniformly on
compact $z$-intervals, proves
\[
 \mathcal U(t,z)\longrightarrow U_*(z-\log a)
\]
locally uniformly in $z$.

The radial comparison solutions are nonincreasing in radius. Their convergence bounds $u$ away from both $0$ and $1$ on each fixed interval of $z_r$, completing \eqref{eq:h-ahead}. It also gives uniform saturation behind the front.

To quantify the latter statement, fix a sufficiently small $0<\varepsilon<\lambda_*$. We claim that
\begin{equation}\label{eq:negative-tail}
 c_\varepsilon\e^{(\lambda_*+\varepsilon)z_r}
 \le1-u(t,x)\le C_\varepsilon\e^{(\lambda_*-\varepsilon)z_r}
\end{equation}
for sufficiently large $t$ in the region
\[
 t\le|x|\le m(t)-Z_0,
\]
where $Z_0$ is chosen sufficiently large. The constants do not depend on a subsequent fixed left cutoff in the moving coordinate.

Set $v=1-u$ and $\omega_*=-f'(1)>0$. Then
\[
 v_t=\Delta v+A(t,x)v,
\]
and saturation makes $|A+\omega_*|$ arbitrarily small in $|x|\le m(t)-Z_0$. Put $\nu_\pm=\lambda_*\pm\varepsilon$. Upper and lower comparison functions on this ball are, respectively,
\[
 C\exp\!\left[
 \nu_-\bigl(\sqrt{|x|^2+R_0^2}-m(t)\bigr)
 \right],\qquad
 c\,\e^{-\nu_+m(t)}\cosh(\nu_+|x|).
\]
Choose $R_0$ large, then make the coefficient error small by increasing $Z_0$, and finally take the starting time large. The differential inequalities follow from
\[
 \omega_*-2\nu_- -\nu_-^2>0,\qquad
 \omega_*-2\nu_+ -\nu_+^2<0,
\]
together with
\[
 |\nabla\sqrt{|x|^2+R_0^2}|\le1,\qquad
 \Delta\sqrt{|x|^2+R_0^2}\le N/R_0,
\]
and
\[
 \Delta\cosh(\nu_+|x|)\ge\nu_+^2\cosh(\nu_+|x|).
\]
The scalar radial comparisons provide the required bounds at the moving outer boundary. The constants $c,C$ can be chosen to fit the data on the finite starting ball, and both radial comparison functions satisfy every transmission condition. On $|x|\ge t$, these bounds give \eqref{eq:negative-tail}.

Finally, the left-hand wave asymptotics convert \eqref{eq:negative-tail} into a bound of the form
\[
 |h|\le C_\varepsilon+C\varepsilon|z|.
\]
Combining this with the bounded-interval estimates and choosing the auxiliary $\varepsilon$ sufficiently small proves \eqref{eq:h-behind}. For each fixed $Z$, the corresponding points satisfy $|x|\ge t$ once $t$ is sufficiently large, which explains the allowed dependence of the starting time on $Z$.
\end{proof}

\begin{lemma}[Uniform front localization]
\label{lem:localization}
For every $0<B<\infty$,
\begin{align*}
 \lim_{A\to\infty}\liminf_{t\to\infty}
 \inf_{\substack{0\le\beta\le B\\|x|\le m(t)-A}}
 u^\beta(t,x)&=1,\\
 \lim_{A\to\infty}\limsup_{t\to\infty}
 \sup_{\substack{0\le\beta\le B\\|x|\ge m(t)+A}}
 u^\beta(t,x)&=0.
\end{align*}
Both interface traces are included.
\end{lemma}

\begin{proof}
Let $U$ be the common radial upper solution, and let
$w(t-t_0,\cdot)$ be the common radial lower solution
inserted at time $t_0$, as constructed in
Section~\ref{sec:prelim}. Thus
\[
 w(t-t_0,x)\le u(t,x)\le U(t,x),
 \qquad t\ge t_0,\quad 0\le\beta\le B.
\]
Both comparison solutions are nonincreasing in radius.

The radial wave convergence established in the
proof of Lemma~\ref{lem:coarse} gives finite constants
$\xi_-$ and $\xi_+$ such that, locally uniformly in $z$,
\[
 w(t-t_0,m(t)+z)\longrightarrow U_*(z-\xi_-),
 \qquad
 U(t,m(t)+z)\longrightarrow U_*(z-\xi_+).
\]
Here the radial solutions are written as functions of
radius, and the fixed time shift is absorbed into $\xi_-$,
using $m(t)-m(t-t_0)\to2t_0$.

For each fixed $A>0$, radial monotonicity gives
\[
 \inf_{\substack{0\le\beta\le B\\|x|\le m(t)-A}}
 u^\beta(t,x)\ge w(t-t_0,m(t)-A),
\]
and
\[
 \sup_{\substack{0\le\beta\le B\\|x|\ge m(t)+A}}
 u^\beta(t,x)\le U(t,m(t)+A).
\]
First let $t\to\infty$, and then let $A\to\infty$.
The limits $U_*(-\infty)=1$ and $U_*(+\infty)=0$
prove both assertions.
\end{proof}

\subsection{An exact change of variable}
Fix a direction and write $u(t,z,\zeta)$ in the corresponding moving Cartesian frame, where $\zeta$ denotes the physical transverse displacement. We suppress the direction and resistance parameters in this subsection. In the interior of each component,
\[
 u_t=\Delta u+(2-k/t)u_z+f(u),
\]
where $\nabla$ and $\Delta$ act on all spatial variables $(z,\zeta)$. Since $0<u<1$ at positive times and $U_*'<0$, the displacement
\[
 h=z-U_*^{-1}(u),\qquad u=U_*(z-h),
\]
is well defined.

Set $s=z-h$. Differentiation gives
\[
 u_t=-U_*'(s)h_t,\qquad
 u_z=U_*'(s)(1-h_z),
\]
and
\[
 \Delta u
 =U_*''(s)\bigl(1-2h_z+|\nabla h|^2\bigr)
  -U_*'(s)\Delta h.
\]
Substituting these identities and using $U_*''+2U_*'+f(U_*)=0$ yields the exact equation
\begin{equation}\label{eq:h-pde}
 h_t=\Delta h+d(t,z,\zeta)h_z
      -r_U(z-h)|\nabla h|^2+\frac{k}{t},
 \qquad
 d=2\bigl(1+r_U(z-h)\bigr)-\frac{k}{t}.
\end{equation}

Choose $M>\|r_U\|_\infty$ and fix a target displacement $\xi\in\mathbb R$, independent of time and the spatial variables. For
\[
 \Phi_\pm=\exp\bigl(\pm M(h-\xi)\bigr),
\]
the chain rule gives
\begin{align}\label{eq:Phi-pde}
 (\partial_t-\Delta-d\partial_z)\Phi_\pm
 &=\pm\frac{Mk}{t}\Phi_\pm
   -\bigl(M^2\pm Mr_U(z-h)\bigr)
      \Phi_\pm|\nabla h|^2\notag\\
 &\le\frac{Mk}{t}\Phi_\pm.
\end{align}
Indeed, $M^2\pm Mr_U(z-h)\ge M(M-\|r_U\|_\infty)>0$. Thus both exponential transforms satisfy the same scalar differential inequality. The reaction term has been eliminated using the wave equation, and no reaction divided difference appears.

We next construct a common lower bound for the longitudinal drift. Fix the cone and the parameter ranges in Lemma~\ref{lem:coarse}. There exist constants $Z_1>0$, $b_0>0$, and a smooth function $b_*(z)$, all independent of the left cutoff $Z$, such that
\[
 d(t,z,\zeta)\ge b_*(z)
 \qquad
 \bigl(-Z\le z\le t^\gamma,\ |\zeta|\le t^\delta\bigr)
\]
for every fixed $Z>0$ and all sufficiently large $t$. The starting time may depend on $Z$. The same function can be used for both folded components at the equator, and it can be chosen to satisfy
\begin{equation}\label{eq:b-star}
 b_*(z)=b_0>0\quad(z\le-Z_1),\qquad
 b_*(z)=\frac{3}{2z}\quad(z\ge Z_1).
\end{equation}

To verify the left-hand bound, apply \eqref{eq:h-behind} with a fixed $0<\epsilon<1$. For $-Z\le z\le0$,
\[
 z-h\le z+|h|
      \le C_\epsilon-(1-\epsilon)|z|.
\]
Since $C_\epsilon$ is independent of $Z$, sufficiently negative $z$ places $z-h$ uniformly in the left wave tail whenever the coarse estimate applies. By \eqref{eq:wave-ratio}, $r_U(z-h)$ is then close to $\lambda_*>0$. We may therefore choose a fixed $b_0>0$ and enlarge $Z_1$ and the starting time so that $d\ge b_0$ on $[-Z,-Z_1]$.

Ahead of the front, Lemma~\ref{lem:coarse} gives bounded $h$. The right-hand expansion in \eqref{eq:wave-ratio} consequently yields
\[
 d=\frac{2}{z}+O(z^{-2})-\frac{k}{t}.
\]
For $Z_1\le z\le t^\gamma$, this implies
\[
 d-\frac{3}{2z}
 \ge\frac1z\left(\frac12-\frac{C}{Z_1}
                         -kt^{\gamma-1}\right)\ge0
\]
after increasing $Z_1$ and then the starting time, establishing the right-hand bound in \eqref{eq:b-star}. Finally, the global bound on $r_U$ provides a constant lower bound for $d$ on the intermediate interval. A smooth interpolation below these bounds completes the construction of $b_*$. This interpolation may be negative on the fixed intermediate interval.

\subsection{Boundary influence and decay with a common lower drift}
Fix $Z>0$, $\ell>1$, $R>0$, and a starting time $s_0$.
On the fixed box
\[
 \mathcal D=(-Z,\ell)\times(-R,R)^{N-1},
 \qquad
 \mathcal L_d=\partial_s-\Delta-d(s,z,\zeta)\partial_z,
\]
assume that $d$ is bounded and continuous on each finite
closed cylinder and satisfies $d\ge b_*(z)$. All subsolutions
below are $C^{1,2}$ in the interior and continuous on the closed
cylinder, with the indicated boundary inequalities. The
coefficients and solutions arising above have this regularity
on positive-time cylinders, so classical parabolic comparison
applies. We estimate separately the influence of the left
face, the initial slice, and the transverse faces by explicit
supersolutions. No coefficient outside this cylinder is used.

First define
\[
 \mu(z)=\exp\left(\int_0^z b_*(r)\dd r\right),\qquad
 P_{Z,\ell}(z)=
 \frac{\int_z^\ell\mu(r)^{-1}\dd r}
      {\int_{-Z}^\ell\mu(r)^{-1}\dd r}.
\]
Writing $P=P_{Z,\ell}$, we have
\[
 \begin{gathered}
 0\le P\le1,\quad P'<0,\quad P''+b_*P'=0,\\
 \mathcal L_dP=-(d-b_*)P'\ge0,\qquad
 P(-Z)=1,\quad P(\ell)=0.
 \end{gathered}
\]
Thus any bounded subsolution with nonpositive initial data,
boundary data at most one on the left face, and nonpositive
data on the other faces is bounded above by $P$.
The prescribed tails of $b_*$ give
\[
 \mu(z)\asymp\e^{b_0z}\quad(z\to-\infty),\qquad
 \mu(z)\asymp z^{3/2}\quad(z\to+\infty).
\]
Fix $K>0$ and take $Z,\ell>K+1$. For $|z|\le K$, the numerator
in $P$ is bounded by $\int_{-K}^\infty\mu(r)^{-1}\dd r<\infty$,
whereas its denominator is bounded below by a positive
constant times $\e^{b_0Z}$. Consequently,
\begin{equation}\label{eq:left-boundary}
 P_{Z,\ell}(z)\le C_K\e^{-b_0Z},\qquad |z|\le K,
\end{equation}
with $C_K$ independent of $Z,\ell,R$, and the particular drift.

For control of the initial data, define
\[
 F_\ell(z)=\int_z^\ell\mu(r)^{-1}
                 \left(\int_{-\infty}^r\mu(u)\dd u\right)\dd r.
\]
The inner integral converges because $b_0>0$. Differentiation
and the lower bound on $d$ give
\[
 F_\ell\ge0,\quad F_\ell'\le0,\quad
 F_\ell''+b_*F_\ell'=-1,\qquad
 \mathcal L_dF_\ell=1-(d-b_*)F_\ell'\ge1.
\]
The function
$G(r)=\mu(r)^{-1}\int_{-\infty}^r\mu(u)\dd u$
is bounded for $r\le0$ and satisfies $G(r)\le C(1+r)$
for $r\ge0$. Hence, with a fixed sufficiently large $C_*$,
\[
 0\le F_\ell(z)\le L_*:=C_*(1+Z+\ell^2),
 \qquad -Z\le z\le\ell.
\]
In particular, $F_\ell$ bounds any solution of
$\mathcal L_dw=1$ with zero initial and Dirichlet boundary
data. This is a deterministic bound on the accumulated
response to a unit source.

Now let $v_{\rm in}$ be a bounded subsolution of
$\mathcal L_dv_{\rm in}\le0$ with initial data at most one
and nonpositive data on every spatial face. Define
\[
 S_*(s,z)=\e^{-(s-s_0)/(2L_*)}
                \left(1+\frac{F_\ell(z)}{L_*}\right).
\]
This function is at least one initially and nonnegative on
all spatial faces. Since $0\le F_\ell/L_*\le1$,
\[
 \mathcal L_d S_*
 =\e^{-(s-s_0)/(2L_*)}
 \left[\frac{\mathcal L_dF_\ell}{L_*}
       -\frac{1+F_\ell/L_*}{2L_*}\right]\ge0.
\]
Comparison therefore gives, for every $H\ge0$,
\begin{equation}\label{eq:initial-decay}
 v_{\rm in}(s_0+H,z,\zeta)\le S_*(s_0+H,z)
 \le2\exp\left[-\frac{H}{2C_*(1+Z+\ell^2)}\right].
\end{equation}
This global supersolution uses only the pointwise lower bound
on $d$ and works for time-dependent coefficients without a
semigroup or a time-homogeneity assumption.

For the transverse-boundary influence, let $v_{\rm lat}$
be a bounded subsolution with nonpositive initial and
longitudinal-face data, and data at most one on all transverse
faces during $[s_0,s_0+H]$, where $H>0$. For $\lambda>0$, set
\[
 B_R(s,\zeta)=2\e^{-\lambda R+\lambda^2(s-s_0)}
                  \sum_{j=1}^{N-1}\cosh(\lambda\zeta_j).
\]
It is independent of $z$ and solves the transverse heat
equation, so $\mathcal L_dB_R=0$. It is nonnegative on the
initial slice and the longitudinal faces. If any
$|\zeta_j|=R$, then
$2\e^{-\lambda R}\cosh(\lambda R)=1+\e^{-2\lambda R}\ge1$,
so it dominates the transverse data. Comparison gives
$v_{\rm lat}\le B_R$. For $|\zeta_j|\le K<R$, optimize by
choosing $\lambda=(R-K)/(2H)$ to obtain
\begin{equation}\label{eq:lateral-boundary}
 v_{\rm lat}(s_0+H,z,\zeta)
 \le2(N-1)\exp\left[-\frac{(R-K)^2}{4H}\right]
 \le2(N-1)\exp\left[-\frac{R^2}{16H}\right]
 \quad(R\ge2K).
\end{equation}

The statements above apply to subsolutions with the indicated
parabolic boundary inequalities, understood as upper limits
at boundary points when necessary. No auxiliary solution
with discontinuous corner data is needed: for a given
continuous solution, nonnegative multiples of $P_{Z,\ell}$,
$S_*$, and $B_R$, together with a constant for the right face,
can be added to dominate all of its parabolic boundary data
simultaneously. The maximum principle then applies to this
single global supersolution. The constants depend only on
$b_*$ and the dimension, and on $K$ where indicated. In the
application, $b_*$ is fixed after the resistance bound and
geometric parameter ranges have been fixed; it is independent
of $Z,\ell,R$, the particular resistance, and the direction.
The time from which $d\ge b_*$ holds may depend on $Z$.
Accordingly, the comparison estimates are uniform in the
resistance and direction for each fixed $Z$, and the limit
$t\to\infty$ must be taken before $Z\to\infty$.
In the folded problem we next verify these same barriers
against the transformed exchange condition.

\subsection{The transmission condition after the change of variable}
At an equatorial ray, the longitudinal coordinate
$z$ is tangential to the interface. Write the
physical transverse coordinates as
\[
 \zeta=(\zeta_\parallel,y),
 \qquad
 \zeta_\parallel\in\R^{N-2},\quad y\ge0.
\]
These coordinates are not divided by $\sqrt t$;
the scaled variables $(v,Y)$ are reserved for
the leading-edge analysis.
Let $n=-e_y$ be the outward normal on $y=0$.
In particular, $\partial_n z=0$.

Fix either sign and define
\[
 F(z,u)=\exp\bigl[\pm M(z-U_*^{-1}(u)-\xi)\bigr].
\]
Below, $F'(u)$ denotes differentiation with respect to $u$,
with $z$ fixed. Set
\[
 A_F=
 \begin{cases}
 \displaystyle\frac{F(z,p)-F(z,q)}{p-q},&p\ne q,\\[6pt]
 F'(p),&p=q.
 \end{cases}
\]
Since
\[
 F'(u)=\mp\frac{M F(z,u)}{U_*'(U_*^{-1}(u))},
\]
the derivative has a fixed strict sign on $(0,1)$.
The identity
\[
 A_F=\int_0^1 F'\bigl(q+\theta(p-q)\bigr)\dd\theta
\]
shows that $A_F$ has the same sign and never vanishes.

For $\beta>0$, put $\varepsilon=1/\beta$. The folded
transmission conditions and $\partial_n z=0$ give
\begin{align*}
 \partial_nF(z,p)
 &=-\varepsilon\frac{F'(p)}{A_F}
       \bigl(F(z,p)-F(z,q)\bigr),\\
 \partial_nF(z,q)
 &= \varepsilon\frac{F'(q)}{A_F}
       \bigl(F(z,p)-F(z,q)\bigr).
\end{align*}
Thus, with $v_1=F(z,p)$ and $v_2=F(z,q)$,
\begin{equation}\label{eq:transformed-exchange}
 \partial_nv_i+a_i(v_i-v_{3-i})=0,\qquad
 a_1=\varepsilon\frac{F'(p)}{A_F}>0,\quad
 a_2=\varepsilon\frac{F'(q)}{A_F}>0.
\end{equation}
At $p=q$, both coefficients have the continuous value
$\varepsilon$. Positivity holds for both exponential
transformations, including the decreasing one. The following
comparison principle makes precise why no uniform upper
bound on these coefficients is needed.

\begin{lemma}[A common supersolution for boundary exchange]
\label{lem:common-comparison}
Let $I=[s_0,s_1]$ be a finite time interval and set
\[
 \mathcal B=(-Z,\ell)\times(-R,R)^{N-2}\times(0,R).
\]
Let $\Gamma$ denote the relative interior of its face $y=0$.
Suppose two subsolutions are $C^{1,2}$ in the interior,
continuous on the closed cylinder, and have continuous
one-sided normal derivatives on the relative interior of
$\Gamma$ at positive times. Suppose they satisfy
\begin{align*}
 (\partial_s-\Delta-d_i\partial_z-c_i)v_i&\le0
       &&\text{in }I\times\mathcal B,\\
 \partial_nv_i+a_i(v_i-v_{3-i})&\le0
       &&\text{on }I\times\Gamma,
\end{align*}
where $d_i,c_i$ are bounded and $a_i\ge0$ are continuous
on the closed exchange face. All coefficients may depend
on time and position, and $a_1=a_2$ is not required.

Suppose a common scalar function $B$ satisfies
\[
 (\partial_s-\Delta-d_i\partial_z-c_i)B\ge0
 \quad(i=1,2),\qquad
 \partial_nB\ge0\quad\text{on }I\times\Gamma.
\]
If $v_i\le B$ initially and on all other spatial faces,
including their edges by continuity, then $v_i\le B$
throughout the cylinder. The conclusion requires no
quantitative upper bound on $a_i$. In particular, common
barriers give uniform estimates for families whose exchange
coefficients are not uniformly bounded.
\end{lemma}

\begin{proof}
Put $w_i=v_i-B$. Since the same function $B$ is subtracted
from both components,
\[
 (\partial_s-\Delta-d_i\partial_z-c_i)w_i\le0,\qquad
 \partial_nw_i+a_i(w_i-w_{3-i})\le0.
\]
Multiply by $\e^{-C_0(s-s_0)}$, where
$C_0>\max_i\sup c_i$, and relabel the resulting functions
as $w_i$. Their zeroth-order coefficients are now
$\widetilde c_i=c_i-C_0<0$, while their exchange conditions
are unchanged.

To obtain strict inequalities, subtract from both components
$\delta\psi$, where
\[
 \psi(s,y)=\e^{K_0(s-s_0)}[1+(R-y)^2],
 \qquad \delta>0.
\]
Choose $K_0$ large enough that
\[
 (\partial_s-\Delta-d_i\partial_z-\widetilde c_i)\psi>0
 \quad(i=1,2).
\]
This choice uses only the interior coefficient bounds and
the finite box. Moreover,
\[
 \partial_n\psi=2R\e^{K_0(s-s_0)}>0
 \qquad(y=0).
\]
The functions $\widetilde w_i=w_i-\delta\psi$ therefore
satisfy strict interior and exchange inequalities; the common
perturbation cancels from the exchange differences.

A positive maximum over both components and the closed
space-time cylinder cannot occur initially or on a prescribed
spatial face. An interior maximum contradicts the strict
parabolic inequality and $\widetilde c_i<0$. At a maximum
on $\Gamma$, the maximizing component satisfies
\[
 \widetilde w_i\ge\widetilde w_{3-i},
 \qquad \partial_n\widetilde w_i\ge0.
\]
Since $a_i\ge0$, this gives
\[
 \partial_n\widetilde w_i
   +a_i(\widetilde w_i-\widetilde w_{3-i})\ge0,
\]
contrary to the strict exchange inequality. Edges are covered
by the prescribed-face inequalities and continuity.
Letting $\delta\downarrow0$ proves the assertion. No derivative
or quantitative upper bound of $a_i$ enters the argument.
\end{proof}

For each fixed $\beta>0$ and each finite closed cylinder
bounded away from time zero, $p,q$ take values in a compact
subinterval of $(0,1)$. Hence the coefficients in
\eqref{eq:transformed-exchange} are finite and continuous,
including at $p=q$. Their bounds may deteriorate as
$\beta\downarrow0$ or as the boxes grow.
Lemma~\ref{lem:common-comparison} is applied on each such
cylinder separately, using scalar barriers whose constants
are common to the parameter family. The drift coefficients
in \eqref{eq:Phi-pde} have common bounds by
\eqref{eq:wave-ratio}.

On a positive-time interval starting at $s_0$, multiplication
by
\[
 \exp\left(-\int_{s_0}^s\frac{Mk}{r}\dd r\right)
\]
removes the common zeroth-order term $Mk/s$ from
\eqref{eq:Phi-pde}. This factor is independent of position
and the component, so it preserves
\eqref{eq:transformed-exchange}.

We now verify the boundary-influence and initial-data decay
estimates for the transformed pair.
After removing the zeroth-order term, $P_{Z,\ell}$ and
$F_\ell$ have zero normal derivative on $\Gamma$ and satisfy
their scalar interior inequalities for every $d_i\ge b_*$.
They can therefore be used as common barriers.

For the initial-data contribution, use the same function
\[
 S_*(s,z)=\e^{-(s-s_0)/(2L_*)}
             \left(1+\frac{F_\ell(z)}{L_*}\right),
 \qquad L_*=C_*(1+Z+\ell^2)\ge\sup_{[-Z,\ell]}F_\ell.
\]
Its interior supersolution inequality holds for both
$d_i\ge b_*$. It is independent of the transverse variables,
so $\partial_nS_*=0$ on $\Gamma$, and its common boundary
value cancels the exchange terms. It is at least one
initially and nonnegative on all prescribed spatial faces.
Lemma~\ref{lem:common-comparison} therefore gives
\eqref{eq:initial-decay} for each component, with the same
constants and no quantitative bound on the exchange rates.

For the transverse-boundary contribution on a time block
of length $H>0$, use
\[
 B_R(s,\zeta)=
 2\e^{-\lambda R+\lambda^2(s-s_*)}
 \sum_{j=1}^{N-1}\cosh(\lambda\zeta_j),
 \qquad \lambda>0,
\]
where $\zeta_{N-1}=y$ in the folded box.
This function solves the transverse heat equation and is
independent of $z$, so
\[
 (\partial_s-\Delta-d_i\partial_z)B_R=0
 \qquad(i=1,2).
\]
It is at least one on every transverse outer face,
is nonnegative initially and on the longitudinal faces,
and has zero normal derivative at $y=0$.

For $|\zeta_j|\le K<R$, choose $\lambda=(R-K)/(2H)$.
Then
\[
 B_R(s_*+H,\zeta)
 \le2(N-1)\exp\left[-\frac{(R-K)^2}{4H}\right].
\]
For $R\ge2K$, Lemma~\ref{lem:common-comparison} yields
\eqref{eq:lateral-boundary} for the transverse-boundary
contribution. Similarly, $P_{Z,\ell}$ bounds the contribution
from the left face and gives \eqref{eq:left-boundary}.

For a given transformed solution, regard its coefficients
$d_i$ and $a_i$ as fixed functions of time and position.
The resulting interior and boundary operators are linear.
Thus nonnegative multiples of $P_{Z,\ell}$, $S_*$, and $B_R$,
together with a nonnegative constant, form a common global
supersolution. Their sum directly dominates the given data
on every initial or prescribed boundary face, including
edges by continuity; no decomposition into auxiliary
boundary-value problems is required. Every barrier inequality holds
pointwise for these coefficients; no constancy in time or
symmetry between the exchange rates is required.

At $\beta=0$, use the smooth whole-space solution and the
full transverse box. The same scalar barriers apply without
an exchange boundary. Together with the estimates for
$\beta>0$, this gives the required uniformity down to zero
resistance.

\subsection{Closing the comparison}
Fix a finite resistance bound $B>0$ and a receiving cone
$\Gamma_{\delta_0}$, with $0<\delta_0\le1$. The same argument
will be applied at the equator to the folded pair.
Choose
\[
 \begin{gathered}
 0<\gamma<1/3,\qquad
 \gamma<\delta<1/2,\qquad
 2\gamma<a<2\delta,\\
 H=t^a,\qquad
 \ell=(t-H)^\gamma,\qquad
 R=\frac{(t-H)^\delta}{\sqrt{N-1}}.
 \end{gathered}
\]
In particular, $a<1$, so $H=o(t)$ and $t-H>0$ for large $t$.

Fix a left cutoff $Z>0$. On the time interval $[t-H,t]$,
freeze the spatial box
\[
 -Z<z<\ell,\qquad |\zeta_j|<R.
\]
At the equator, replace the last transverse interval by
$0<y<R$ and use both folded components. For every
$s\in[t-H,t]$,
\[
 \ell\le s^\gamma,\qquad
 |\zeta|\le\sqrt{N-1}R\le s^\delta.
\]
Thus Lemma~\ref{lem:coarse} applies throughout the frozen
box once $t$ is sufficiently large, with the starting time
allowed to depend on $Z$. On the fixed receiving cone,
the entire box lies on the receiving side for large $t$.

We first identify the data on the right face. Uniformly
for $s\in[t-H,t]$,
\[
 \frac{\ell}{s^\gamma}\longrightarrow1,\qquad
 \frac{\sqrt{N-1}R}{\sqrt{s}}\longrightarrow0.
\]
The final assertion of Section~\ref{sec:meso} therefore gives
\[
 \frac{\e^\ell}{\ell}
 u_+\bigl(s,(m(s)+\ell)\theta+\zeta\bigr)
 =a_\beta(\theta)+o(1)
\]
uniformly on this face, for $\theta\in\Gamma_{\delta_0}$
and $0\le\beta\le B$. At the equator, the corresponding
statement holds for each folded component with amplitude
$a_\beta(\eta,0)$.

By Lemma~\ref{lem:coarse}, $h$ is uniformly bounded on the
right face. The normalized right-hand wave expansion then
gives
\[
 \frac{\e^\ell}{\ell}U_*(\ell-h)
 =\e^h\left(1+\frac{B_+-h}{\ell}+o(\ell^{-1})\right)
 =\e^h+o(1).
\]
In particular, the amplitudes are bounded above and away
from zero uniformly over the parameter ranges under
consideration. Hence
\[
 \xi=\log a_\beta(\theta)
\]
is uniformly bounded, and comparison of the last two
expansions yields
\[
 |h-\xi|\le\varepsilon_t,\qquad
 \varepsilon_t\longrightarrow0
 \quad\text{on }z=\ell.
\]
At the equator, take $\xi=\log a_\beta(\eta,0)$ for both
components. In each application, $\xi$ is fixed throughout
the time window and spatial box.

Fix $M>\|r_U\|_\infty$ and write
\[
 \Phi_\pm=\exp\bigl(\pm M(h-\xi)\bigr).
\]
The right-face estimate gives
$\Phi_\pm\le\e^{M\varepsilon_t}$ there. On the initial
slice and the remaining prescribed spatial faces,
Lemma~\ref{lem:coarse} and the uniform bound on $\xi$ give,
for every $\epsilon>0$,
\[
 \Phi_\pm\le M_Z,\qquad
 M_Z\le C_\epsilon\e^{M\epsilon Z}.
\]
Here $C_\epsilon$ is independent of $Z,\ell,R,t$ and of
the direction and resistance in the fixed parameter ranges.
The exchange face in the folded box carries its transformed
boundary condition rather than prescribed boundary data.

Remove the term $Mk/s$ in \eqref{eq:Phi-pde} by multiplying
by
\[
 \exp\left(-\int_{t-H}^s\frac{Mk}{r}\dd r\right).
\]
Restoring the original variables at time $t$ costs the factor
\[
 \exp\left(\int_{t-H}^t\frac{Mk}{r}\dd r\right)
 =\left(\frac{t}{t-H}\right)^{Mk}
 \le\e^{CH/t}.
\]
Apply Lemma~\ref{lem:common-comparison} to the folded pair,
and scalar comparison on the receiving cone. The constant
barrier controls the right-face data. The left-boundary,
initial-data decay, and transverse-boundary barriers in
\eqref{eq:left-boundary}, \eqref{eq:initial-decay}, and
\eqref{eq:lateral-boundary} control, respectively,
the left-face data, the initial data, and the transverse
outer-face data.
More explicitly, with $s_0=t-H$ and
$\lambda=(R-K)/(2H)$, use the single common supersolution
\[
 \mathcal B(s,z,\zeta)=\e^{M\varepsilon_t}
       +M_Z\bigl[P_{Z,\ell}(z)+S_*(s,z)+B_R(s,\zeta)\bigr].
\]
It dominates the transformed data initially because $S_*\ge1$
there, on the left face because $P_{Z,\ell}=1$, on the right
face through the constant term, and on each transverse outer
face because $B_R\ge1$. Each summand has zero normal derivative
on the exchange face and is common to the two components.
The integrating factor used above is at most one, so the
same bounds on the original data also bound the transformed
data. This verifies all boundary comparisons simultaneously.

Fix $K>0$ and take $Z>K+1$. For sufficiently large $t$,
we also have $\ell>K+1$ and $R\ge2K$. At
$|z|\le K$ and $|\zeta_j|\le K$, the resulting estimate is
\begin{equation}\label{eq:final-match}
 \Phi_\pm(t,z,\zeta)
 \le\e^{CH/t}\left\{
 \e^{M\varepsilon_t}
 +C_KM_Z\left[
 \e^{-b_0Z}
 +\e^{-cH/(1+Z+\ell^2)}
 +\e^{-cR^2/H}
 \right]\right\}.
\end{equation}
The constants in this estimate are common to both signs
and, at the equator, to both folded components.

Choose $\epsilon>0$ so that $M\epsilon<b_0/2$, and keep
this choice fixed. First fix $Z$ and let $t\to\infty$.
The parameter restrictions imply
\[
 \frac{H}{t}=t^{a-1}\longrightarrow0,\qquad
 \frac{H}{1+Z+\ell^2}\asymp t^{a-2\gamma}
 \longrightarrow\infty,\qquad
 \frac{R^2}{H}\asymp t^{2\delta-a}
 \longrightarrow\infty.
\]
Together with $\varepsilon_t\to0$, these limits reduce
\eqref{eq:final-match} to
\[
 \limsup_{t\to\infty}
 \sup_{\substack{|z|\le K\\|\zeta_j|\le K}}\Phi_\pm
 \le1+C_KC_\epsilon
          \e^{-(b_0-M\epsilon)Z}.
\]
The supremum here also ranges over the fixed direction and
resistance parameters, with both components included in
the equatorial case. Letting $Z\to\infty$ gives
\[
 \limsup_{t\to\infty}\sup\Phi_+\le1,\qquad
 \limsup_{t\to\infty}\sup\Phi_-\le1.
\]
Since
\[
 \max\{\Phi_+,\Phi_-\}=\e^{M|h-\xi|},
\]
we conclude that $h\to\xi$ locally uniformly in the physical
coordinates $(z,\zeta)$.

Consequently,
\[
 u_+\bigl(t,(m(t)+z)\theta+\zeta\bigr)
 \longrightarrow
 U_*\bigl(z-\log a_\beta(\theta)\bigr)
\]
locally uniformly in $(z,\zeta)$ on each fixed receiving
cone. At the equator, both folded components converge to
the profile with displacement $\log a_\beta(\eta,0)$,
including in bounded physical transverse neighborhoods.

Every estimate used above is uniform for $0\le\beta\le B$
and for directions in the stated ranges. The order of
limits, first $t\to\infty$ with $Z$ fixed and then
$Z\to\infty$, is consistent with the cutoff-dependent
starting time in Lemma~\ref{lem:coarse}. This completes
the phase part of Theorem~\ref{thm:phase}.

\section{Parameter dependence and arrival times}\label{sec:parameters}
We first establish a persistence estimate that converts a positive seed into a positive leading-edge amplitude. Under concavity, this estimate upgrades resistance comparison to a strict amplitude gap. We then prove parameter continuity and recover first arrival times from the wave asymptotics.

\subsection{A persistent positive leading edge}
Strict pointwise comparison does not by itself imply a strict gap between asymptotic phases. The following lemma provides the additional leading-edge lower bound.

\begin{lemma}[Persistence of a positive seed]\label{lem:persistence}
Let $V\ge0$ be a whole-space supersolution satisfying
\[
 V_t\ge\Delta V+(1-Q)V,
 \qquad 0\le Q\le C.
\]
Assume that, for some $0<\sigma<1/4$, the potential $Q$ is bounded by a fixed multiple of the right-hand side of \eqref{eq:envelope} in $|x|\ge m(t)$. If $V$ dominates a nonzero, nonnegative compactly supported seed at some positive time, then there exists $c>0$ such that, for every compact $K\subset(0,\infty)$,
\begin{equation}\label{eq:persistence}
 \liminf_{t\to\infty}
 \inf_{\theta\in\Sph^{N-1},\ \rho\in K}
 \left[
 t^{-1/2}\e^{\rho\sqrt t}
 V\bigl(t,(m(t)+\rho\sqrt t)\theta\bigr)
 -c\rho\e^{-\rho^2/4}
 \right]\ge0.
\end{equation}
The constant $c$ can be chosen independently of $K$.
\end{lemma}

\begin{proof}
Let $t_0>0$ be a time at which the seed is dominated. Since $Q$ is bounded, heat comparison gives, at every fixed $T>t_0$,
\[
 V(T,x)\ge c_T
 \exp\left(-\frac{(|x|+M_0)^2}{4(T-t_0)}\right),
\]
where $M_0$ bounds the support radius of the seed and $c_T>0$.

Use logarithmic time $\tau=\log t$ and, in the exterior region, set
\[
 \rho=\frac{|x|-m(t)}{\sqrt t},\qquad
 \eta_t=t^{\sigma-1/2},\qquad
 X=\rho-\eta_t>0,
\]
\[
 c_t=\frac14+\eta_t,\qquad
 F(\tau,\rho)=X\e^{-c_tX^2}.
\]
For the coefficients in \eqref{eq:coefficients}, direct differentiation yields
\begin{equation}\label{eq:gaussian-residual}
 \frac{(\partial_\tau-L_0-p_{\rm rad}\partial_\rho)F}{F}
 =6\eta_t
 -\bigl[(1/2+\sigma)\eta_t+4\eta_t^2\bigr]X^2
 -(\sigma\eta_t+p_{\rm rad})(X^{-1}-2c_tX).
\end{equation}
Here
\[
 r=m(t)+\rho\sqrt t>m(t)+t^\sigma>2t
\]
for sufficiently large $t$. Consequently,
\[
 q_c\ge0,\qquad
 -\frac{k}{\sqrt t}\le p_{\rm rad}<0,\qquad
 \sigma\eta_t+p_{\rm rad}>0.
\]
The singular term $-(\sigma\eta_t+p_{\rm rad})X^{-1}$ is nonpositive. The remaining positive linear term in $X$ is bounded by $C\eta_tX$ and can be absorbed into the negative quadratic term, leaving an upper bound $C\eta_t$ for the residual in \eqref{eq:gaussian-residual}.

The assumed envelope bound also gives
\[
 tQ(t,x)\le Ct^{1+\sigma}\e^{-t^\sigma}
 \qquad\bigl(|x|\ge m(t)+t^\sigma\bigr).
\]
Both error bounds are integrable in logarithmic time. Define
\[
 \chi(\tau)=\chi(\log T)
 \exp\left[
 -\int_{\log T}^{\tau}
 \left(
 C\e^{(\sigma-1/2)s}
 +C\e^{(1+\sigma)s}\e^{-\e^{\sigma s}}
 \right)\dd s
 \right].
\]
Then $\chi(\infty)>0$ whenever $\chi(\log T)>0$. The residual estimate and $q_c\ge0$ show that
\[
 \underline V(t,x)
 =\sqrt t\,\e^{-\rho\sqrt t}\chi(\log t)F(\log t,\rho)
\]
is a radial subsolution of $\partial_t-\Delta-(1-Q)$ in $|x|>m(t)+t^\sigma$. This comparison is valid even when $Q$ is not radial, since its bound is pointwise. The subsolution vanishes on the moving inner boundary.

Choose $T$ sufficiently large that
\[
 \frac14+\eta_T>\frac{T}{4(T-t_0)}.
\]
At time $T$, the subsolution therefore has faster Gaussian decay than the heat lower bound. A sufficiently small $\chi(\log T)>0$ makes $\underline V(T,\cdot)\le V(T,\cdot)$ on the entire exterior region. Exterior comparison gives $\underline V\le V$ for all later times.

Finally, $\eta_t\to0$, $c_t\to1/4$, and $\chi(\log t)\to\chi(\infty)>0$. Hence
\[
 \chi(\log t)F(\log t,\rho)
 \longrightarrow
 \chi(\infty)\rho\e^{-\rho^2/4}
\]
uniformly on each compact subset of $(0,\infty)$. This proves \eqref{eq:persistence}.
\end{proof}

\subsection{Strict resistance ordering under concavity}
\label{sec:order}
Assume \eqref{ass:C} throughout this subsection. Write
\[
 F(S,D)=\frac12[f(S+D)+f(S-D)],\qquad
 H(S,D)=\frac12[f(S+D)-f(S-D)].
\]
On the convex invariant set $0\le S-D\le S+D\le1$,
\begin{equation}\label{eq:derivatives}
 F_S=H_D=\frac{f'(S+D)+f'(S-D)}2,\qquad
 F_D=H_S=\frac{f'(S+D)-f'(S-D)}2\le0.
\end{equation}

Fix $0\le\beta_1<\beta_2<\infty$, and use subscripts $1,2$ for the corresponding folded solutions. Set
\[
 R=S_1-S_2,\qquad Z=D_2-D_1,
\]
and interpolate along
\[
 (S_s,D_s)=(1-s)(S_2,D_2)+s(S_1,D_1),
 \qquad 0\le s\le1.
\]
Define
\[
 A=\int_0^1F_S(S_s,D_s)\dd s,\qquad
 B=-\int_0^1F_D(S_s,D_s)\dd s\ge0.
\]
Using \eqref{eq:derivatives}, the difference equations become
\begin{equation}\label{eq:coop}
 R_t=\Delta R+AR+BZ,\qquad
 Z_t=\Delta Z+BR+AZ.
\end{equation}

If $\beta_1>0$, their boundary conditions are
\[
 \partial_nR=0,\qquad
 \partial_nZ+\kappa_2Z
 =(\kappa_1-\kappa_2)D_1>0,
 \qquad \kappa_i=\frac2{\beta_i}.
\]
If $\beta_1=0$, then $D_0=0$ on the interface and
\[
 \partial_nZ+\kappa_2Z=-\partial_nD_0>0.
\]
The last inequality follows from positivity of $D_0$ in the interior and the boundary point lemma.

Initially $R=Z=0$. Negative-part comparison for the cooperative system gives $R,Z\ge0$, and the positive boundary input makes $Z$ strictly positive at positive times. Since $q_i=S_i-D_i$, it follows that
\[
 q_1-q_2=R+Z>0,
\]
as in the one-dimensional argument of \cite[Proposition 5.1]{Wang}.

To obtain a strict amplitude gap, including at the equator, we also need a positive average-mode seed. The next lemma shows that concavity suffices even when $f''$ vanishes on an interval.

\begin{lemma}[Creation of an average-mode seed]\label{lem:Rseed}
Under \eqref{ass:K} and \eqref{ass:C}, the function $R$ in \eqref{eq:coop} is nonzero at some positive time and is strictly positive throughout $\overline H$ at every sufficiently large time.
\end{lemma}

\begin{proof}
The function $f'$ is nonincreasing and cannot be constant: otherwise $f'(0)=1$ would contradict $f(1)=f(0)=0$. Since $f\in C^2([0,1])$, there is an $s_*\in(0,1)$ and a neighborhood of $s_*$ on which $f''<0$.

By local convergence to $1$, for some sufficiently large $t_*$ the donor density $p_1(t_*,\cdot)$ exceeds $s_*$ at an interior point. At spatial infinity it tends to zero by the finite-time Gaussian bound. Continuity along a path in $H$ therefore gives an interior point $x_*$ such that
\[
 p_1(t_*,x_*)=s_*.
\]
Since $D_1>0$ in $H$, we have $q_1(t_*,x_*)<s_*$. Concavity and strict decrease of $f'$ immediately below $s_*$ imply
\[
 f'(p_1(t_*,x_*))<f'(q_1(t_*,x_*)).
\]
Thus the integrand defining $B$ is strictly positive at the interpolation endpoint $s=1$. By continuity, $B>0$ on a space-time neighborhood of $(t_*,x_*)$.

Since $Z>0$, the source $BZ$ in the first equation of \eqref{eq:coop} is positive on that neighborhood. Choose $C_0$ such that $A+C_0\ge0$. Then
\[
 R_t-\Delta R+C_0R\ge BZ,
 \qquad \partial_nR=0.
\]
Positivity of the Neumann heat kernel makes $R$ strictly positive throughout $\overline H$ at every time after this positive source has acted. The same argument applies when $\beta_1=0$, since $D_0>0$ in the interior.
\end{proof}

Concavity gives $A\le1$, while \eqref{eq:basic-g} implies
\begin{equation}\label{eq:A-potential}
 0\le1-A\le C_f(p_1+q_1+p_2+q_2).
\end{equation}
Extend $R$ evenly across $\Sigma$. Its homogeneous Neumann condition, together with \eqref{eq:coop}, makes the extension a whole-space supersolution of
\[
 R_t\ge\Delta R+(1-Q)R,\qquad Q=1-A.
\]
By \eqref{eq:A-potential}, $Q$ is bounded and satisfies the required exterior envelope estimate. Lemma~\ref{lem:Rseed} provides a positive compact seed. Lemma~\ref{lem:persistence} therefore gives a constant $c>0$ such that the normalized leading edge of $R$ dominates $c\rho\e^{-\rho^2/4}$ in every direction.

Since $q_1-q_2\ge R$, passage to the receiving leading-edge limits yields
\[
 a_{\beta_1}(\theta)-a_{\beta_2}(\theta)\ge c
 \qquad(\theta\in\Sph_+).
\]
Continuity in direction extends this inequality to the closed receiving hemisphere. The constant $c$ may depend on the fixed pair $(\beta_1,\beta_2)$, but not on direction.

By \eqref{eq:arrival},
\[
 \tau_{\beta_2}(\theta)-\tau_{\beta_1}(\theta)
 =\frac12\log\frac{a_{\beta_1}(\theta)}
                         {a_{\beta_2}(\theta)}.
\]
The amplitudes are positive and continuous on the compact closed hemisphere. The uniform amplitude gap therefore gives a uniform positive arrival-time gap, proving both inequalities in \eqref{eq:strict-gap}.

\subsection{Continuity and bounded-parameter uniformity}
Fix $0<B<\infty$. Equation~\eqref{eq:amplitude-moment}, Proposition~\ref{prop:moment}, and finite-time continuity give joint continuity of $a_\beta(\theta)$ on $[0,B]\times\overline{\Sph_+}$. More explicitly, the finite-time moments are jointly continuous, and their uniform convergence allows this continuity to pass to the limiting moment and hence to the amplitude.

The common radial lower seed gives a uniform positive lower bound for the amplitudes on this parameter set. Thus taking logarithms preserves joint continuity:
\[
 \xi_\beta(\theta)=\log a_\beta(\theta),\qquad
 \tau_\beta(\theta)=\frac12
       \bigl(\log a_0(\theta)-\log a_\beta(\theta)\bigr)
\]
are jointly continuous as well.

The Cartesian compactness argument, relative
mesoscopic estimates, and displacement comparison
use bounds uniform for $0\le\beta\le B$.
They therefore yield uniform wave convergence
on each fixed receiving cone $\Gamma_\delta$,
$\delta>0$, and separately on the equator,
without using resistance ordering.
Together with the joint continuity proved above
and the strict ordering under \eqref{ass:C},
this proves Theorem~\ref{thm:parameters}.

\subsection{First arrival times}
Fix a finite resistance bound and let the threshold $\upsilon$ range over a compact subset of $(0,1)$. Write
\[
 z_\upsilon=U_*^{-1}(\upsilon).
\]
These values remain in a bounded interval. Directions are taken in a fixed receiving cone or, separately, on the equator, where the receiving trace is used.

Lemma~\ref{lem:localization} supplies a constant $A>0$ and a sufficiently large time $t_*$ such that, uniformly over these parameters,
\[
 u_+(t,x)<\upsilon
 \quad\text{if }|x|\ge m(t)+A,
\]
and
\[
 u_+(t,x)>\upsilon
 \quad\text{if }|x|\le m(t)-A,
 \qquad t\ge t_*.
\]
Increase $t_*$ if necessary so that $m$ is strictly increasing thereafter. For large $R$, let $t_-(R)$ and $t_+(R)$ satisfy
\[
 m(t_-(R))=R-A,\qquad
 m(t_+(R))=R+A.
\]
The exterior bound excludes a crossing during $[t_*,t_-(R)]$, while the interior bound ensures a crossing by $t_+(R)$. On the remaining bounded interval $[0,t_*]$, the common compact-data Gaussian bound excludes a crossing when $R$ is sufficiently large. Consequently,
\[
 t_-(R)\le T_\upsilon^\beta(R,\theta)\le t_+(R).
\]
In particular,
\[
 T_\upsilon^\beta(R,\theta)<\infty,\qquad
 T_\upsilon^\beta(R,\theta)\longrightarrow\infty,\qquad
 R-m(T_\upsilon^\beta(R,\theta))=O(1),
\]
uniformly in the stated parameter and direction sets. This argument uses spatial localization and does not require monotonicity of the solution in time.

Continuity in time implies that the solution equals $\upsilon$ at its first crossing. We may therefore apply the uniform wave limit along these arrival times. Since the moving coordinate remains bounded and $U_*$ is strictly decreasing, inversion gives
\begin{equation}\label{eq:arrival-inversion}
 R=m(T_\upsilon^\beta)
   +\xi_\beta(\theta)+z_\upsilon+o(1).
\end{equation}
The error is uniform because the thresholds lie in a compact subset of $(0,1)$ and the phases are uniformly bounded.

Subtracting the corresponding formula for $\beta=0$ yields
\[
 m(T_\upsilon^\beta)-m(T_\upsilon^0)
 =\xi_0(\theta)-\xi_\beta(\theta)+o(1).
\]
Both arrival times tend to infinity uniformly. Since $m'(t)=2-k/t\to2$, the mean value theorem gives
\[
 T_\upsilon^\beta-T_\upsilon^0
 =\frac12\bigl(\xi_0(\theta)-\xi_\beta(\theta)\bigr)+o(1),
\]
which is \eqref{eq:arrival}. The threshold-dependent term $z_\upsilon$ cancels, so the limiting difference is independent of the threshold. All errors are uniform in the stated ranges. This completes the proof of Theorem~\ref{thm:phase}.

\section{Strong resistance: input and absorption balances}
\label{sec:input}
Throughout Sections~\ref{sec:input}--\ref{sec:sharp},
fix $0<\delta\le1$ and restrict the direction to
$\theta\in\Gamma_\delta$. We study the limit
$\beta\to\infty$ using
\[
 \varepsilon=\beta^{-1},\qquad
 L=\log(\varepsilon^{-1}).
\]
Write
\[
 s_\varepsilon=p_\beta-q_\beta\ge0,\qquad
 Q_\varepsilon(t,x',x_N)=q_\beta(t,x',|x_N|).
\]
Thus $Q_\varepsilon$ is the even extension of the receiving
density. When $s_\varepsilon$ is evaluated on $\Sigma$,
we write $s_\varepsilon(t,x')$ for its trace.

Since
$\partial_yq_\beta(t,x',0)=-\varepsilon s_\varepsilon(t,x')$,
the jump of $\partial_NQ_\varepsilon$ across $\Sigma$ is
$-2\varepsilon s_\varepsilon$. Consequently, the even
extension satisfies, in distributions,
\begin{equation}\label{eq:source}
 (Q_\varepsilon)_t
 =\Delta Q_\varepsilon+Q_\varepsilon-g(Q_\varepsilon)
   +2\varepsilon s_\varepsilon(t,x')\delta_\Sigma,
 \qquad Q_\varepsilon(0)=0.
\end{equation}
Here $\delta_\Sigma$ denotes surface measure on the
interface, regarded as a distribution. The factor $2$
accounts for the two reflected half-spaces.

\subsection{Finite-time input linearization}
Let $p_\infty$ be the Neumann donor solution.
Reflection comparison and donor comparison give
\[
 0\le s_\varepsilon\le p_\beta\le p_\infty.
\]
Define the linear receiving response $v$ by
\begin{equation}\label{eq:linear-input}
 v_t=\Delta v+v\quad\text{in }H,\qquad
 \partial_nv=p_\infty|_\Sigma,\qquad
 v(0)=0.
\end{equation}

\begin{proposition}\label{prop:finite}
For every finite $T>0$, there is a constant $C_T$ such
that, uniformly for $0\le t\le T$ as
$\varepsilon\downarrow0$,
\begin{equation}\label{eq:finite-linear}
 \|p_\infty-p_\beta\|_\infty+\|q_\beta\|_\infty
 \le C_T\varepsilon,\qquad
 \|q_\beta/\varepsilon-v\|_\infty\le C_T\varepsilon.
\end{equation}
The norms are taken over $H$, with the boundary traces
included at positive times.

Moreover, on each fixed receiving cone, the exponentially
weighted input moments of $s_\varepsilon$ converge to
those of $p_\infty$. This convergence remains valid after
multiplication by any fixed polynomial weight in time
and tangential position.
\end{proposition}

\begin{proof}
The boundary flux
\[
 j=\varepsilon s_\varepsilon
\]
satisfies $0\le j\le\varepsilon$. The Neumann heat
boundary-response operator has $L^\infty$ norm bounded
by $C(t-s)^{-1/2}$: the tangential heat kernel has mass
one, and the normal boundary kernel contributes this
factor.

Set $d=p_\infty-p_\beta$. The mild equations for $d$
and $q_\beta$, together with the Lipschitz continuity of
$f$ on $[0,1]$, yield
\[
 \|d(t)\|_\infty+\|q_\beta(t)\|_\infty
 \le C\varepsilon\sqrt t
 +C\int_0^t
   \bigl(\|d(s)\|_\infty+\|q_\beta(s)\|_\infty\bigr)\dd s.
\]
Gronwall's inequality proves the first estimate in
\eqref{eq:finite-linear}.

For the second estimate, put
$w=q_\beta/\varepsilon-v$. Then
\[
 w_t=\Delta w+w-\frac{g(q_\beta)}{\varepsilon},
 \qquad
 \partial_nw=s_\varepsilon-p_\infty,
 \qquad w(0)=0.
\]
The first estimate gives
\[
 \|s_\varepsilon-p_\infty\|_\infty=O_T(\varepsilon).
\]
Since $g(q)\le C_fq^2$, it also gives
\[
 \left\|\frac{g(q_\beta)}{\varepsilon}\right\|_\infty
 =O_T(\varepsilon).
\]
The same heat-kernel estimates therefore imply
$\|w\|_\infty\le C_T\varepsilon$.

For the input moments, use
$0\le s_\varepsilon\le p_\infty$ and the compact-data
linear Neumann upper bound for $p_\infty$.
Tangential exponential tilting produces the factor
\[
 \e^{-2t}\e^t\e^{|\theta'|^2t}
 =\e^{-\theta_N^2t}.
\]
Thus, for each fixed integer $j\ge0$, the weighted
boundary integral is bounded by a constant times
\[
 t^{-1/2}(1+t)^j
 \e^{-\delta^2t}\e^{-d_0^2/(4t)},
\]
where $d_0>0$ is the distance of the initial support
from the interface. This bound is integrable at both
zero and infinity.

Finite-time convergence, followed by dominated convergence,
proves convergence of the weighted input moments.
The domination is uniform for $\theta\in\Gamma_\delta$,
so the convergence is uniform on that cone. This argument
does not assert an $O(\varepsilon)$ rate after integration
over the entire time interval.
\end{proof}

Replacing $p_\infty$ in \eqref{eq:linear-input} by its
linear Neumann upper bound and evaluating the resulting
boundary convolution gives
\begin{equation}\label{eq:v-bound}
 v(t,x',y)\le
 2\e^t\int_HG_t^{N-1}(x'-a')
 \left(\int_0^\infty G_t^1(y+z+\ell)\dd\ell\right)
 p_0(a',z)\dd a'\dd z.
\end{equation}
At $y=0$, the inner integral is at most $1/2$.
Completing the square in the tangential variables therefore
gives, for each fixed integer $j\ge0$,
\[
 \e^{-2t}\int_\Sigma
 (1+t+|x'|)^j\e^{\theta'\cdot x'}v(t,x',0)\dd x'
 \le C_j(1+t)^j\e^{-\theta_N^2t}.
\]
The corresponding Gaussian bounds for $p_\infty$ prove
finiteness and continuity of $I_N$ on each fixed cone.
Its positivity follows from the positive boundary trace
of the nontrivial Neumann donor solution at positive times.

Define
\begin{align*}
 j_\varepsilon(t,\theta)
 &=\e^{-2t}\int_\Sigma
      \e^{\theta'\cdot x'}s_\varepsilon(t,x')\dd x',\\
 j_{1,\varepsilon}(t,\theta)
 &=\e^{-2t}\int_\Sigma
      (\theta'\cdot x'-2t)\e^{\theta'\cdot x'}
      s_\varepsilon(t,x')\dd x',\\
 J_\varepsilon(\theta)
 &=\int_0^\infty j_\varepsilon(t,\theta)\dd t,\qquad
 K_\varepsilon(\theta)
 =\int_0^\infty j_{1,\varepsilon}(t,\theta)\dd t.
\end{align*}
The weighted input convergence gives
\begin{equation}\label{eq:JK}
 J_\varepsilon=I_N+o_\delta(1),\qquad
 K_\varepsilon=I_{1,N}+o_\delta(1),
\end{equation}
uniformly on $\Gamma_\delta$, where
\[
 I_{1,N}(\theta)=
 \int_0^\infty\e^{-2t}\int_\Sigma
 (\theta'\cdot x'-2t)\e^{\theta'\cdot x'}
 p_\infty(t,x',0)\dd x'\dd t.
\]
Both integrals are absolutely convergent.

\subsection{Exact critical balances}
For the even receiving extension, define
\begin{align*}
 M_\varepsilon(t,\theta)
 &=\e^{-2t}\int_{\R^N}
      \e^{\theta\cdot x}Q_\varepsilon(t,x)\dd x,\\
 B_\varepsilon(t,\theta)
 &=\e^{-2t}\int_{\R^N}
      \e^{\theta\cdot x}g(Q_\varepsilon(t,x))\dd x,\\
 \mathcal E_\varepsilon(t,\theta)
 &=\e^{-2t}\int_{\R^N}
      (\theta\cdot x-2t)\e^{\theta\cdot x}
      Q_\varepsilon(t,x)\dd x.
\end{align*}
Since $|\theta|=1$, the two weights
\[
 \e^{-2t}\e^{\theta\cdot x},
 \qquad
 \e^{-2t}(\theta\cdot x-2t)\e^{\theta\cdot x}
\]
solve the backward adjoint equation
$(\partial_t+\Delta+1)\psi=0$.
Testing \eqref{eq:source} against these weights gives
\begin{equation}\label{eq:MB}
 M_\varepsilon'=-B_\varepsilon
                  +2\varepsilon j_\varepsilon,
\end{equation}
and
\begin{equation}\label{eq:source-moment}
 \mathcal E_\varepsilon'
 =\e^{-2t}\int_{\R^N}
      (2t-\theta\cdot x)\e^{\theta\cdot x}
      g(Q_\varepsilon)\dd x
   +2\varepsilon j_{1,\varepsilon}.
\end{equation}
These calculations can first be performed with spatial
cutoffs and then passed to the limit using the Gaussian
bounds. The zero initial receiving density gives
$M_\varepsilon(0,\theta)=\mathcal E_\varepsilon(0,\theta)=0$.

The common radial envelope implies
$M_\varepsilon(t,\theta)\to0$ as $t\to\infty$.
For $\theta_N\ge\delta$, the contribution of the negative
half-space to the critical derivative moment decays
exponentially. Indeed, exponential tilting centers the
linear Gaussian upper bound at $2t\theta$, whose distance
from that half-space is at least $2\delta t$.
The same observation applies to the donor contribution
in the critical derivative moment of the original solution.

On the receiving half-space, $Q_\varepsilon=q_\beta$.
The amplitude identification for the original solution
therefore yields
\[
 \lim_{t\to\infty}\mathcal E_\varepsilon(t,\theta)
 =(4\pi)^{N/2}a_\beta(\theta)
 =\kap^{-1}a_\beta(\theta).
\]
Integrating \eqref{eq:MB} and
\eqref{eq:source-moment} now gives
\begin{align}
 \int_0^\infty B_\varepsilon(t,\theta)\dd t
 &=2\varepsilon J_\varepsilon(\theta),
 \label{eq:total-absorption}\\
 \kap^{-1}a_\beta(\theta)
 &=
 \int_0^\infty\e^{-2t}\int_{\R^N}
 (2t-\theta\cdot x)\e^{\theta\cdot x}
 g(Q_\varepsilon)\dd x\dd t
 +2\varepsilon K_\varepsilon(\theta).
 \label{eq:amplitude-absorption}
\end{align}

The centered absorption integral in
\eqref{eq:amplitude-absorption} is absolutely convergent.
Indeed, the estimate in \eqref{eq:absorption-tail} applies
to $Q_\varepsilon$ through the same radial envelope,
while the finite-time Gaussian bounds control the initial
time interval. We retain the centered factor
$2t-\theta\cdot x$ throughout: the argument controls this
combined expression and does not require separate
absolute-integrability estimates for its two terms.

\section{The sharp amplitude coefficient}\label{sec:sharp}
We prove matching lower and upper bounds for the amplitude on each fixed receiving cone. The lower bound follows from a sharp small-seed estimate applied to the accumulated finite-time input. For the upper bound, we first show that almost all critical mass is absorbed before a small power of $\varepsilon^{-1}$, and then estimate the centered absorption moment.

\subsection{A sharp small-seed lower bound}
\begin{lemma}[Small-seed lower estimate]\label{lem:logistic-seed}
Consider the auxiliary reaction $f_{\log}(s)=s-s^2$. Let $0\le\varphi\in C_c^\infty(\R^N)$ be nonzero, and let $a_{\varepsilon,\varphi}$ denote the homogeneous directional amplitude for initial density $\varepsilon\varphi$. Then
\begin{equation}\label{eq:small-seed}
 \liminf_{\varepsilon\downarrow0}
 \inf_{\theta\in\Sph^{N-1}}
 \left[
 \frac{a_{\varepsilon,\varphi}(\theta)}
      {\varepsilon\log(1/\varepsilon)}
 -\kap\int_{\R^N}\e^{\theta\cdot y}\varphi(y)\dd y
 \right]\ge0.
\end{equation}
\end{lemma}

\begin{proof}
Let $U_\varepsilon$ be the corresponding solution. Fix $0<\zeta<1$, put $L=\log(1/\varepsilon)$, and choose
\[
 A=(1-\zeta)L,\qquad T=L^4.
\]
Suppose $\supp\varphi\subset B_{R_0}$ and take $\varepsilon$
sufficiently small that $\varepsilon\|\varphi\|_\infty\le1$.
For each $\theta\in\Sph^{N-1}$, introduce the moving half-space
\[
 D_t=\{x:\theta\cdot x>2t-A\},\qquad 0\le t\le T.
\]
We construct a linear Dirichlet lower solution in $D_t$ and
compute its kernel explicitly. All constants below may depend
on $\varphi,N,\zeta$, but are independent of $\theta$ and
sufficiently small $\varepsilon$.

Linear comparison gives $U_\varepsilon(t,x)\le C\varepsilon\e^t$.
For $t\ge A/2\ge1$ and $x\in D_t$, the compact-data heat bound gives
\[
 U_\varepsilon(t,x)
 \le C\varepsilon t^{-N/2}
 \exp\left(t-\frac{[(2t-A-R_0)_+]^2}{4t}\right)
 \le C\varepsilon\e^{A+R_0}
 \le C\varepsilon^\zeta.
\]
Indeed, with $B=A+R_0$, the exponent is at most $B$:
if $2t\le B$ it equals $t\le B/2$, and otherwise it equals
$B-B^2/(4t)$. Choose a common constant $C$ so that
\[
 q_\varepsilon(t)=
 \begin{cases}
 C\varepsilon\e^t,&0\le t\le A/2,\\
 C\varepsilon^\zeta,&A/2<t\le T
 \end{cases}
\]
bounds $U_\varepsilon$ throughout $D_t$, and set
$\Lambda_\varepsilon(t)=\int_0^t q_\varepsilon(s)\dd s$.
Then
\begin{equation}\label{eq:seed-absorption-loss}
 0\le \Lambda_\varepsilon(T)
 \le C\left[\varepsilon^{(1+\zeta)/2}
                 +T\varepsilon^\zeta\right]=o(1).
\end{equation}

Let $v$ solve the linear problem
\[
 \begin{cases}
 v_t=\Delta v+(1-q_\varepsilon(t))v,&x\in D_t,\\
 v=0,&x\in\partial D_t,\\
 v(0,x)=\varepsilon\varphi(x),&x\in D_0.
 \end{cases}
\]
For large $L$, the initial support lies strictly inside $D_0$.
Moreover,
\[
 (\partial_t-\Delta-1+q_\varepsilon)U_\varepsilon
 =(q_\varepsilon-U_\varepsilon)U_\varepsilon\ge0
 \quad\hbox{in }D_t.
\]
Thus $U_\varepsilon\ge v$ by the parabolic comparison principle.
One can apply comparison separately on the two time intervals
on which $q_\varepsilon$ is continuous. On each finite time
interval the solutions are bounded. To justify comparison on
the unbounded domain, add
$\delta\e^{(2N+2)t}(1+|x|^2)$ to $U_\varepsilon$, apply
comparison on $D_t\cap B_R$ for sufficiently large $R$, and
then let $R\to\infty$ and $\delta\downarrow0$. The added
function has nonnegative residual under
$\partial_t-\Delta-1+q_\varepsilon$ and dominates the bounded
solutions on the artificial outer boundary.

To compute $v$, write
\[
 x=x_\perp+(h+2t-A)\theta,\qquad
 x_\perp\in\theta^\perp,\quad h>0,
\]
and let $\widetilde v(t,x_\perp,h)=v(t,x)$. Then
\[
 \widetilde v_t=\Delta_{x_\perp,h}\widetilde v
       +2\partial_h\widetilde v+(1-q_\varepsilon)\widetilde v.
\]
The substitution
$\widetilde v=\e^{-h-\Lambda_\varepsilon(t)}Z$
reduces this equation to $Z_t=\Delta Z$ on $h>0$, with zero
Dirichlet boundary value and initial datum
\[
 Z(0,y_\perp,h_0)
 =\varepsilon\e^{h_0}\varphi(y_\perp+(h_0-A)\theta).
\]
The method of images gives the Dirichlet kernel
\[
 G_t^{N-1}(x_\perp-y_\perp)
 \bigl[G_t^1(h-h_0)-G_t^1(h+h_0)\bigr],
 \qquad
 G_t^d(a)=(4\pi t)^{-d/2}\e^{-|a|^2/(4t)}.
\]
Using
\[
 \frac{G_t^1(h+h_0)}{G_t^1(h-h_0)}=\e^{-hh_0/t},
 \qquad
 \e^{h_0-h}G_t^1(h-h_0)
 =\e^tG_t^1(h-h_0+2t),
\]
we obtain the exact formula
\begin{equation}\label{eq:seed-dirichlet-kernel}
 v(t,x)=\varepsilon\e^{t-\Lambda_\varepsilon(t)}
 \int_{\R^N}G_t^N(x-y)
 \left(1-\e^{-(A+\theta\cdot y)(A+\theta\cdot x-2t)/t}\right)
 \varphi(y)\dd y,
 \qquad x\in D_t.
\end{equation}
The integral may be taken over $\R^N$ because the support of
$\varphi$ lies in $D_0$. This identity follows entirely from
the fixed-half-space heat equation and its explicit kernel.
It also constructs the bounded solution used in the comparison:
the factor in parentheses belongs to $[0,1]$ on the support,
so $0\le v(t,x)\le\varepsilon\e^t(G_t^N*\varphi)(x)$.
The compact initial support is separated from $\partial D_0$,
which ensures compatibility at the initial boundary; $v$ is
continuous across $t=A/2$ and classical on either side of it.

At $x=(m(T)+\rho\sqrt T)\theta$, $\rho\ge0$, the two heights are
\[
 h_0=A+\theta\cdot y\ge\alpha_LL,\qquad
 h_T=A-k\log T+\rho\sqrt T\ge\rho\sqrt T,
 \qquad
 \alpha_L=1-\zeta-\frac{R_0}{L}>0,
\]
for sufficiently large $L$. Hence the kernel factor satisfies
\[
 1-\e^{-h_0h_T/T}
 \ge1-\e^{-\alpha_LL\rho/\sqrt T}
 \ge\frac{\alpha_LL\rho}{\sqrt T}
          \e^{-\alpha_LL\rho/\sqrt T}.
\]
Here we used the valid inequality
$1-\e^{-v}=\int_0^v\e^{-s}\dd s\ge v\e^{-v}$ for $v\ge0$.

Define the normalized field and the directional coefficient by
\[
 W(\log t,\rho,\theta)
 =t^{-1/2}\e^{\rho\sqrt t}
 U_\varepsilon(t,(m(t)+\rho\sqrt t)\theta),\qquad
 c_\varphi(\theta)=\kap\int_{\R^N}
              \e^{\theta\cdot y}\varphi(y)\dd y.
\]
For $a=\theta\cdot y$ and $y_\perp=y-a\theta$, completing
squares gives the exact identity
\[
 T^{-1/2}\e^{\rho\sqrt T}\e^TG_T^N(x-y)
 =\kap\sqrt T\,\e^a
 \exp\left[-\frac{(\rho\sqrt T-k\log T-a)^2
                         +|y_\perp|^2}{4T}\right].
\]
Since $|y|\le R_0$ on the support and $\rho\ge0$, it follows that
\[
 T^{-1/2}\e^{\rho\sqrt T}\varepsilon\e^T
 \int G_T^N(x-y)\varphi(y)\dd y
 \ge\varepsilon\sqrt T\,c_\varphi(\theta)
 \exp\left[-d_L-\frac{\rho^2}{4}
                     -\frac{R_0\rho}{2\sqrt T}\right],
\]
where
\[
 d_L=\frac{(k\log T+R_0)^2+R_0^2}{4T}=o(1).
\]
Combining comparison, \eqref{eq:seed-dirichlet-kernel}, and
\eqref{eq:seed-absorption-loss} therefore yields
\begin{equation}\label{eq:seed-at-T}
 W(\log T,\rho,\theta)
 \ge\e^{-\Lambda_\varepsilon(T)-d_L}
       \varepsilon L\alpha_Lc_\varphi(\theta)
       \rho\e^{-\rho^2/4-c_L\rho},
 \qquad
 c_L=\frac{\alpha_LL+R_0/2}{\sqrt T}=O(L^{-1}),
\end{equation}
uniformly for $\rho\ge0$ and $\theta\in\Sph^{N-1}$.
The prefactor $\e^{-\Lambda_\varepsilon(T)-d_L}=1-o(1)$ is
independent of $\rho$ and $\theta$.

We next propagate \eqref{eq:seed-at-T} with a multiplicative loss tending to zero. Fix
\[
 0<\sigma<1/4,\qquad 0<\nu<1-2\sigma,
\]
and, for $t\ge T$, set
\[
 \eta_t=t^{\sigma-1/2},\qquad X=\rho-\eta_t>0,
 \qquad b_G(t)=L^{-1}(t/T)^{-\nu},
\]
\[
 F(t,\rho)=X\e^{-(1/4+b_G(t))X^2}.
\]
We record the residual calculation because the
vanishing multiplicative loss is needed for the
sharp coefficient. Put
\[
 c(t)=\frac14+b_G(t),\qquad
 r=m(t)+\rho\sqrt t,
\]
and use logarithmic-time derivatives. Then
\[
 \partial_\tau b_G=-\nu b_G,\qquad
 \partial_\tau\eta_t=(\sigma-\tfrac12)\eta_t.
\]
For
\[
 F=X\e^{-c(t)X^2},\qquad X=\rho-\eta_t,
\]
direct differentiation gives
\[
 \frac{F_\rho}{F}=X^{-1}-2cX,\qquad
 \frac{F_{\rho\rho}}{F}=-6c+4c^2X^2,
\]
and hence
\begin{equation}\label{eq:seed-residual-exact}
 \frac{(\partial_\tau-L_0-p_{\rm rad}\partial_\rho)F}{F}
 =6b_G-\bigl[(1-\nu)b_G+4b_G^2\bigr]X^2
  -(\sigma\eta_t+p_{\rm rad})(X^{-1}-2cX).
\end{equation}
Here $p_{\rm rad}$ and $L_0$ are as in
\eqref{eq:coefficients}.

For sufficiently large $L$, throughout $t\ge T=L^4$
and $X>0$,
\[
 r>m(t)+t^\sigma>2t,\qquad
 -\frac{k}{\sqrt t}\le p_{\rm rad}<0,\qquad
 0<\sigma\eta_t+p_{\rm rad}\le\sigma\eta_t.
\]
The singular term
$-(\sigma\eta_t+p_{\rm rad})X^{-1}$ is therefore nonpositive.
Also $b_G\le L^{-1}$, so $c(t)$ is uniformly bounded.
Since $\nu<1$, Young's inequality gives
\[
 2c(\sigma\eta_t+p_{\rm rad})X
 \le\frac{1-\nu}{2}b_GX^2
       +C_\nu\frac{\eta_t^2}{b_G}.
\]
It follows from \eqref{eq:seed-residual-exact} that
\[
 \frac{(\partial_\tau-L_0-p_{\rm rad}\partial_\rho)F}{F}
 \le6b_G+C_\nu\frac{\eta_t^2}{b_G}.
\]

In radial and angular coordinates, the normalized
logistic equation is
\[
 \partial_\tau W
 =\left(L_0+p_{\rm rad}\partial_\rho+q_c
          +\frac{t}{r^2}\Delta_\theta\right)W
   -tU_\varepsilon W.
\]
Here $W$ is the normalized field defined above.
The positive angular function $c_\varphi$ is smooth, and
\[
 \kap\e^{-R_0}\int\varphi
 \le c_\varphi(\theta)\le\kap\e^{R_0}\int\varphi.
\]
Since
$\Delta_\theta\e^{\theta\cdot y}
=(|y|^2-(\theta\cdot y)^2-(N-1)\theta\cdot y)
\e^{\theta\cdot y}$, we also have
\[
 |\Delta_\theta c_\varphi|
 \le [R_0^2+(N-1)R_0]c_\varphi.
\]
Thus the angular error has a bound independent of both
$\theta$ and $\varepsilon$.
Moreover, $q_c\ge0$ and $t/r^2\le C/t$ in the
exterior region under consideration. A common
radial upper datum for the small-seed solutions,
together with \eqref{eq:envelope}, gives
\[
 tU_\varepsilon(t,r\theta)
 \le Ct^{1+\sigma}\e^{-t^\sigma},
 \qquad r\ge m(t)+t^\sigma,
\]
uniformly in $\varepsilon$.

Consequently, define the trial profile
\[
 \Psi_L(t,\rho,\theta)=c_\varphi(\theta)F(t,\rho).
\]
Its full residual satisfies
\[
 \frac{1}{\Psi_L}
 \left[
 \partial_\tau-L_0-p_{\rm rad}\partial_\rho-q_c
 -\frac{t}{r^2}\Delta_\theta+tU_\varepsilon
 \right]\Psi_L
 \le\mathcal R_L(t),
\]
where
\[
 \mathcal R_L(t)
 =6b_G+C_\nu\frac{\eta_t^2}{b_G}
       +\frac Ct+Ct^{1+\sigma}\e^{-t^\sigma}.
\]
All constants are independent of $\varepsilon$ and
$t\ge T$, with $\sigma,\nu,\varphi$ fixed.

The loss is integrable in logarithmic time.
Indeed,
\[
 \int_T^\infty b_G(t)\frac{\dd t}{t}
 =\frac1{\nu L},
\]
and, since $1-2\sigma-\nu>0$,
\[
 \int_T^\infty
 \frac{\eta_t^2}{b_G(t)}\frac{\dd t}{t}
 =\frac{LT^{2\sigma-1}}{1-2\sigma-\nu}
 =\frac{L^{8\sigma-3}}{1-2\sigma-\nu}.
\]
Therefore
\begin{equation}\label{eq:seed-loss}
 \int_T^\infty\mathcal R_L(t)\frac{\dd t}{t}
 \le
 \frac6{\nu L}
 +\frac{C_\nu L^{8\sigma-3}}{1-2\sigma-\nu}
 +\frac CT
 +C\int_T^\infty t^\sigma\e^{-t^\sigma}\dd t
 =o(1).
\end{equation}
Since $\sigma<1/4$, we have $8\sigma-3<0$.
Thus the correction factor
\[
 \exp\left[
 -\int_T^t\mathcal R_L(s)\frac{\dd s}{s}
 \right]
 =1-o(1)
\]
uniformly for $t\ge T$ as $\varepsilon\downarrow0$.
Multiplying $\Psi_L$ by this factor makes its
full residual nonpositive.

It remains to fit the initial data.
At $t=T$, write $\rho=X+\eta_T$.
The ratio of
$\rho\e^{-\rho^2/4-c_L\rho}$ to $F(T,\rho)$
is bounded below, uniformly for $X>0$, by
\[
 \exp\left[
 -\frac{(\eta_T/2+c_L)^2}{4b_G(T)}
 -\frac{\eta_T^2}{4}-c_L\eta_T
 \right]=1-o(1).
\]
Indeed, $b_G(T)=L^{-1}$, $c_L=O(L^{-1})$, and
$\eta_T=L^{4\sigma-2}$ imply
\[
 \frac{(\eta_T/2+c_L)^2}{b_G(T)}
 =O(L^{8\sigma-3}+L^{-1})=o(1),
 \qquad \eta_T^2+c_L\eta_T=o(1).
\]
Thus the full exponential loss in the Dirichlet-kernel lower
bound is compatible with the same initial fitting and does
not change the limiting coefficient.
To track all losses explicitly, put
\[
 D_L=\frac{(\eta_T/2+c_L)^2}{4b_G(T)}
             +\frac{\eta_T^2}{4}+c_L\eta_T,
 \qquad
 E_L=\Lambda_\varepsilon(T)+d_L+D_L
                 +\int_T^\infty\mathcal R_L(s)\frac{\dd s}{s}.
\]
For fixed $\zeta,\sigma,\nu,\varphi$, every term is nonnegative
and $E_L=o(1)$, independently of the direction. Define
\[
 \underline W(t,\rho,\theta)
 =\varepsilon L\alpha_L
 \exp\left[-\Lambda_\varepsilon(T)-d_L-D_L
           -\int_T^t\mathcal R_L(s)\frac{\dd s}{s}\right]
 c_\varphi(\theta)F(t,\rho).
\]
The initial fitting above gives
$\underline W(T,\rho,\theta)\le W(\log T,\rho,\theta)$
for every $\rho>\eta_T$ and every direction. The residual is
nonpositive, and $\underline W=0$ at $\rho=\eta_t$.

For completeness, comparison on the unbounded moving exterior
can be justified in physical coordinates. Set
\[
 \underline u(t,r\theta)=\sqrt t\,\e^{-\rho\sqrt t}
                  \underline W(t,\rho,\theta),\qquad
 \rho=\frac{r-m(t)}{\sqrt t},\quad r>m(t)+t^\sigma.
\]
The normalized equation implies
\[
 (\partial_t-\Delta-1+U_\varepsilon)\underline u\le0,
 \qquad
 (\partial_t-\Delta-1+U_\varepsilon)U_\varepsilon=0.
\]
Here $U_\varepsilon$ is a fixed bounded coefficient; this is
linear comparison, not a claim that $\underline u$ solves
the nonlinear logistic equation. On each finite interval
$[T,T_1]$, $\underline u$ tends to zero as $r\to\infty$
uniformly in time and direction, by its Gaussian formula.
For $\delta>0$, the function
$U_\varepsilon+\delta\e^{t-T}$ is a supersolution of the same
linear equation, since its added term has residual
$\delta\e^{t-T}U_\varepsilon\ge0$. For a sufficiently large
outer radius it dominates $\underline u$ there; it also
dominates the initial and moving inner-boundary data.
Comparison on the truncated moving annulus, followed by
letting the outer radius tend to infinity and
$\delta\downarrow0$, gives $\underline u\le U_\varepsilon$.
As $T_1$ is arbitrary, this comparison holds for all $t\ge T$.

Now fix $\varepsilon$ and any $\rho>0$, and let $t\to\infty$.
Then $\eta_t,b_G(t)\to0$ and
$F(t,\rho)\to\rho\e^{-\rho^2/4}$. The homogeneous version of
\eqref{eq:cone-profile}, applied to the fixed datum
$\varepsilon\varphi$, identifies the limit of $W$ as
$a_{\varepsilon,\varphi}(\theta)\rho\e^{-\rho^2/4}$.
Dividing the comparison by this positive Gaussian profile gives
\begin{equation}\label{eq:seed-coefficient-loss}
 a_{\varepsilon,\varphi}(\theta)
 \ge\varepsilon L\alpha_L\e^{-E_L}c_\varphi(\theta)
 \qquad(\theta\in\Sph^{N-1}).
\end{equation}
The right-hand error is independent of $\theta$; no convergence
rate uniform in $\varepsilon$ is needed, because the large-time
limit has already been taken for each fixed $\varepsilon$.
Next let $\varepsilon\downarrow0$ with $\zeta$ fixed. Since
$\alpha_L\to1-\zeta$ and $E_L\to0$, the liminf in
\eqref{eq:small-seed} is at least
$-\zeta\sup_\theta c_\varphi(\theta)$. Finally let
$\zeta\downarrow0$. This proves \eqref{eq:small-seed} with the
exact coefficient $c_\varphi$, without requiring any of the
error estimates to be uniform as $\zeta\downarrow0$.
\end{proof}

\begin{lemma}[Small seeds for a general KPP reaction]
\label{lem:small-seed}
The estimate \eqref{eq:small-seed} also holds for the homogeneous equation with reaction $f$ satisfying \eqref{ass:K}.
\end{lemma}

\begin{proof}
The scaling cancellation used in \cite[Section 7.2]{Wang} applies in every dimension. Choose $C_f\ge1$ such that $f(u)\ge u-C_fu^2$ on $[0,1]$, and let
\[
 (w_\varepsilon)_t
 =\Delta w_\varepsilon+w_\varepsilon-C_fw_\varepsilon^2,
 \qquad
 w_\varepsilon(0)=\varepsilon\varphi,
 \qquad
 \varepsilon\|\varphi\|_\infty\le1/C_f.
\]
If $U_{\varepsilon,f}$ is the homogeneous $f$-solution with the same initial density, comparison gives $w_\varepsilon\le U_{\varepsilon,f}$.

The function $C_fw_\varepsilon$ solves the standard logistic equation with initial density $C_f\varepsilon\varphi$. Lemma~\ref{lem:logistic-seed} therefore yields
\begin{align}\label{eq:cancel}
 a[w_\varepsilon](\theta)
 &=C_f^{-1}a_{\log,C_f\varepsilon\varphi}(\theta)
 \notag\\
 &\ge(1-o(1))\varepsilon\log(1/\varepsilon)c_N
       \int_{\R^N}\e^{\theta\cdot x}\varphi(x)\dd x.
\end{align}
Here $a[w_\varepsilon]$ denotes the diffusive leading-edge coefficient. The factor $C_f$ cancels, and replacing $\log(1/(C_f\varepsilon))$ by $\log(1/\varepsilon)$ changes the bound only by a relative $o(1)$.

The homogeneous $f$-amplitude exists by the Cartesian argument of Section~\ref{sec:Cartesian}. Comparison in the diffusive leading edge transfers \eqref{eq:cancel} to that amplitude, uniformly in direction.
\end{proof}

\begin{proposition}[Lower coefficient]\label{prop:lower}
On each fixed receiving cone,
\[
 \liminf_{\varepsilon\downarrow0}
 \inf_{\theta\in\Gamma_\delta}
 \left[
 \frac{a_{1/\varepsilon}(\theta)}{\varepsilon L}
 -2\kap I_N(\theta)
 \right]\ge0.
\]
\end{proposition}

\begin{proof}
Fix $S>0$ and set
\[
 v_S(x)=v(S,x',|x_N|).
\]
For $0<\zeta<1$, choose a nonnegative smooth compactly supported function $\varphi$, even in $x_N$, such that $\varphi\le(1-\zeta)v_S$. The strict margin on its compact support and \eqref{eq:finite-linear} imply that, for sufficiently small $\varepsilon$, both folded components at time $S$ dominate $\varepsilon\varphi$.

The homogeneous solution from this even initial density has equal interface traces and zero normal derivative. It therefore satisfies every interface condition. Comparison gives
\[
 a_{1/\varepsilon}(\theta)
 \ge\e^{-2S}a_{\varepsilon,\varphi}(\theta).
\]
The factor $\e^{-2S}$ follows from $m(t)-m(t-S)\to2S$ and the normalized leading-edge asymptotics.

The even extension of \eqref{eq:linear-input} has source $2p_\infty\delta_\Sigma$. Its critical mass identity gives
\[
 \e^{-2S}\int_{\R^N}\e^{\theta\cdot x}v_S(x)\dd x
 =2\int_0^S\e^{-2t}\int_\Sigma
       \e^{\theta'\cdot x'}p_\infty(t,x',0)\dd x'\dd t.
\]
Apply Lemma~\ref{lem:small-seed}, approximate $(1-\zeta)v_S$ from below by such functions $\varphi$, and then let $\zeta\downarrow0$ and $S\to\infty$. The fixed-time Gaussian tails and the exponential input tail on $\Gamma_\delta$ make these passages uniform in direction.
\end{proof}

\subsection{Exhausting the critical mass}
We now show that almost all critical mass is absorbed before a small power of $\varepsilon^{-1}$.

\begin{lemma}[Early absorption]\label{lem:early}
Fix $\gamma_0>0$ and $0<\eta_0<\gamma_0/[2(N+2)]$. If $T=\varepsilon^{-\eta_0}$, then, uniformly on $\Gamma_\delta$,
\begin{equation}\label{eq:early}
 M_\varepsilon(T)=o(\varepsilon),\qquad
 \int_T^\infty B_\varepsilon(t)\dd t=o(\varepsilon).
\end{equation}
\end{lemma}

\begin{proof}
A common compactly supported radial datum $\psi$ and linear heat comparison give
\begin{equation}\label{eq:tilted-gaussian}
 V_\varepsilon(t,x):=
 \e^{-2t+\theta\cdot x}Q_\varepsilon(t,x)
 \le\int_{\R^N}\e^{\theta\cdot a}
       G_t^N(x-a-2t\theta)\psi(a)\dd a.
\end{equation}
For every $P_0>0$, choose $C_{P_0}$ sufficiently large. Outside
\[
 |x-2t\theta|\le C_{P_0}(t+L),
\]
the integral of $(1+|2t-\theta\cdot x|)V_\varepsilon$ is bounded by $C_{P_0}\e^{-P_0L}\e^{-ct}$. Completing the square on a hyperplane gives the corresponding trace estimate.

Fix $\theta\in\Gamma_\delta$ and project onto its
longitudinal coordinate. Write
$Z_\perp\in\theta^\perp$ for the physical transverse
integration variable, with $\dd Z_\perp$ denoting
Lebesgue measure on this subspace. Define
\begin{align*}
 \mathfrak f(t,z)
 &=\e^z\int_{\theta^\perp}
 Q_\varepsilon\bigl(t,(2t+z)\theta+Z_\perp\bigr)
 \dd Z_\perp,\\
 \mathfrak b(t,z)
 &=\e^z\int_{\theta^\perp}
 g\!\left(
 Q_\varepsilon\bigl(t,(2t+z)\theta+Z_\perp\bigr)
 \right)\dd Z_\perp\ge0.
\end{align*}
Then
\[
 \int_\R\mathfrak f(t,z)\dd z=M_\varepsilon(t,\theta),
 \qquad
 \int_\R\mathfrak b(t,z)\dd z=B_\varepsilon(t,\theta).
\]
Define the positive source measure by
\begin{equation}\label{eq:projected-source-measure}
 \int_\R\chi(z)\,
       \mathcal J_{\varepsilon,\theta}(t,\dd z)
 =2\varepsilon\e^{-2t}\int_\Sigma
 \e^{\theta'\cdot x'}s_\varepsilon(t,x')
 \chi(\theta'\cdot x'-2t)\dd x'
\end{equation}
for every bounded Borel function $\chi$. Projecting \eqref{eq:source} yields
\begin{equation}\label{eq:projected-heat}
 (\partial_t-\partial_{zz})\mathfrak f
 =-\mathfrak b+\mathcal J_{\varepsilon,\theta},
 \qquad \mathfrak f(0,\cdot)=0,
\end{equation}
in space--time distributions, with source $\dd t\,\mathcal J_{\varepsilon,\theta}(t,\dd z)$. The critical weight and the shift $\theta\cdot x=2t+z$ cancel linear growth and longitudinal drift; integration over $\theta^\perp$ removes transverse diffusion. Gaussian truncation justifies these operations.

No density of the source in $z$ is required. In particular, for the normal direction,
\begin{equation}\label{eq:normal-projected-source}
 \mathcal J_{\varepsilon,e_N}(t,\dd z)
 =2\varepsilon j_\varepsilon(t,e_N)\,
       \delta_{-2t}(\dd z).
\end{equation}

We also need the absolute input moment
\[
 j_{\mathrm{abs},\varepsilon}(t,\theta)
 =\e^{-2t}\int_\Sigma
 |\theta'\cdot x'-2t|\e^{\theta'\cdot x'}
 s_\varepsilon(t,x')\dd x'.
\]
Let $d_0>0$ be the distance of the initial support from $\Sigma$. The linear Neumann bound for $p_\infty$ and $s_\varepsilon\le p_\infty$ imply
\begin{equation}\label{eq:absolute-input-bound}
 j_\varepsilon(t,\theta)
 +j_{\mathrm{abs},\varepsilon}(t,\theta)
 \le C(1+t)t^{-1/2}
       \exp[-\delta^2t-d_0^2/(4t)],
 \qquad t>0.
\end{equation}
Indeed, the tilted tangential Gaussian has mean $a'+2t\theta'$ for $a$ in the initial support. The expected absolute centered factor is therefore at most $C(1+t+\sqrt t)\le C'(1+t)$.

It follows that, for every $A\ge0$,
\begin{align}\label{eq:projected-source-bounds}
 \mathcal J_{\varepsilon,\theta}(t,\R)
 &=2\varepsilon j_\varepsilon(t,\theta),\notag\\
 \int_0^\infty\int_\R(z+A)_+\,
       \mathcal J_{\varepsilon,\theta}(t,\dd z)\dd t
 &\le2\varepsilon\left[
 A J_\varepsilon(\theta)
 +\int_0^\infty
       j_{\mathrm{abs},\varepsilon}(t,\theta)\dd t
 \right]\notag\\
 &\le C_\delta\varepsilon(A+1),\notag\\
 \int_S^\infty
       \mathcal J_{\varepsilon,\theta}(t,\R)\dd t
 &\le C_\delta\varepsilon\e^{-c_\delta S},
 \qquad S\ge1.
\end{align}
The total space--time source mass is $2\varepsilon J_\varepsilon=O_\delta(\varepsilon)$. These estimates involve no division by $|\theta'|$ and are therefore uniform up to the normal direction.

Set $A=(1+\gamma_0)L$ and work on $z>-A$. Let $b_A(t)=\mathfrak f(t,-A)$ be the actual wall trace. Using $Q_\varepsilon\le1$ inside the large ball and \eqref{eq:tilted-gaussian} outside it gives
\begin{equation}\label{eq:wall}
 b_A(t)\le\overline b_A(t):=
 C\varepsilon^{1+\gamma_0}(t+L)^{N-1}
 +C\e^{-P_0L}\e^{-ct}.
\end{equation}
The Dirichlet heat kernel and boundary Poisson kernel on this half-line are
\begin{align*}
 D_A(r;z,\zeta)
 &=G_r^1(z-\zeta)-G_r^1(z+\zeta+2A)\ge0,\\
 P_A(r,z)&=\frac{z+A}{r}G_r^1(z+A)\ge0.
\end{align*}
Discarding the nonpositive absorption term in \eqref{eq:projected-heat} gives
\begin{align}\label{eq:measure-halfline-comparison}
 \mathfrak f(t,z)\le{}&
 \int_0^t\int_{(-A,\infty)}
 D_A(t-s;z,\zeta)
 \mathcal J_{\varepsilon,\theta}(s,\dd\zeta)\dd s
 \notag\\
 &+\int_0^tP_A(t-s,z)\overline b_A(s)\dd s,
 \qquad z>-A.
\end{align}

For completeness, this comparison follows by testing against nonnegative backward Dirichlet heat solutions. The boundary term is the Poisson contribution of the actual trace $b_A$, which is bounded by \eqref{eq:wall}. One may first use positive initial times and finite spatial cutoffs. The Gaussian bound removes the outer cutoff, while \eqref{eq:MB} gives
\[
 0\le M_\varepsilon(s,\theta)
 \le2\varepsilon\int_0^s j_\varepsilon(r,\theta)\dd r
 \longrightarrow0
 \qquad(s\downarrow0).
\]
Thus the initial term vanishes. Approximation of nonnegative terminal test functions then gives \eqref{eq:measure-halfline-comparison}, with the kernel integrated directly against the source measure.

Only the source restricted to $(-A,\infty)$ appears in the interior integral. Source at or to the left of the wall affects the half-line through its actual trace, already controlled by \eqref{eq:wall}. For example, the atom in \eqref{eq:normal-projected-source} lies inside the half-line for $s<A/2$ and outside it for $s>A/2$. Its crossing time carries no atom in the time measure.

For $d=\zeta+A>0$, direct integration gives
\begin{align}\label{eq:halfline-kernel-masses}
 \int_{-A}^\infty D_A(r;z,\zeta)\dd z
 &=\frac2{\sqrt\pi}
   \int_0^{d/(2\sqrt r)}\e^{-v^2}\dd v
 \le\min\left\{1,\frac d{\sqrt{\pi r}}\right\},
 \notag\\
 \int_{-A}^\infty P_A(r,z)\dd z
 &=\frac1{\sqrt{\pi r}}.
\end{align}
Tonelli's theorem applies to these nonnegative integrals, including atomic sources.

Split the source times at $T/2$. By \eqref{eq:projected-source-bounds} and \eqref{eq:halfline-kernel-masses}, the source contribution to the mass at time $T$ is at most
\begin{align*}
 &\frac C{\sqrt T}
 \int_0^{T/2}\int_\R(\zeta+A)_+\,
       \mathcal J_{\varepsilon,\theta}(s,\dd\zeta)\dd s
 +\int_{T/2}^T
       \mathcal J_{\varepsilon,\theta}(s,\R)\dd s\\
 &\qquad\le C_\delta\left[
 \frac{\varepsilon(A+1)}{\sqrt T}
 +\varepsilon\e^{-c_\delta T}\right].
\end{align*}
The wall contribution is bounded by
\[
 \int_0^T
 \frac{\overline b_A(s)}{\sqrt{\pi(T-s)}}\dd s
 \le C\left[
 \varepsilon^{1+\gamma_0}\sqrt T(T+L)^{N-1}
 +\e^{-P_0L}\sqrt T\right].
\]

For the mass with $z\le-A$, use $\e^zQ_\varepsilon\le\e^{-A}$ inside the large ball and \eqref{eq:tilted-gaussian} outside it. This gives
\[
 C\varepsilon^{1+\gamma_0}(T+L)^N+C\e^{-P_0L}.
\]
Adding the three contributions and using $A+1=O(L)$, we obtain
\begin{align}\label{eq:M-T}
 M_\varepsilon(T)\le C\bigl[
 &\varepsilon L/\sqrt T+\varepsilon\e^{-cT}
 +\varepsilon^{1+\gamma_0}(T+L)^N\notag\\
 &+\varepsilon^{1+\gamma_0}\sqrt T(T+L)^{N-1}
 +\e^{-P_0L}(1+\sqrt T)\bigr].
\end{align}
With $T=\varepsilon^{-\eta_0}$, the restriction on $\eta_0$ makes the two volume terms $o(\varepsilon)$. The first two terms also have this property, and choosing $P_0>1+\eta_0/2$ controls the last term.

Finally, integrating \eqref{eq:MB} from $T$ to infinity and using $M_\varepsilon(\infty)=0$ gives
\[
 \int_T^\infty B_\varepsilon(t)\dd t
 =M_\varepsilon(T)
  +2\varepsilon\int_T^\infty j_\varepsilon(t)\dd t
 =o(\varepsilon).
\]
This proves both assertions in \eqref{eq:early}.
\end{proof}

\subsection{The centered absorption upper bound}
\begin{proposition}[Upper coefficient]\label{prop:upper}
On each fixed receiving cone,
\[
 \limsup_{\varepsilon\downarrow0}
 \sup_{\theta\in\Gamma_\delta}
 \left[
 \frac{a_{1/\varepsilon}(\theta)}{\varepsilon L}
 -2\kap I_N(\theta)
 \right]\le0.
\]
\end{proposition}

\begin{proof}
Write $D_x=2t-\theta\cdot x$ and fix $\gamma_0>0$. Choose
\[
 0<\eta_0<
 \min\{1,\gamma_0/[2(N+2)]\},
 \qquad
 T=\varepsilon^{-\eta_0},\qquad
 T_2=\varepsilon^{-4}.
\]
Then $T<T_2$, and Lemma~\ref{lem:early} applies at time $T$.

On $0\le t\le T$, use
\[
 D_x\le A+(D_x-A)_+,\qquad A=(1+\gamma_0)L.
\]
The constant part contributes at most
\[
 A\int_0^TB_\varepsilon(t)\dd t
 \le2(1+\gamma_0)\varepsilon L J_\varepsilon.
\]
For the excess, inside $|x-2t\theta|\le C_{P_0}(t+L)$ use
\[
 \e^{-D_x}g(Q_\varepsilon)\le C_f\e^{-D_x},
 \qquad
 (D_x-A)_+\e^{-D_x}\le C\e^{-A}.
\]
Spatial and time integration gives
\begin{equation}\label{eq:early-excess}
 C\varepsilon^{1+\gamma_0}(T+L)^{N+2}
 =o(\varepsilon L).
\end{equation}
Outside the ball, $g(Q_\varepsilon)\le C_fQ_\varepsilon$ and \eqref{eq:tilted-gaussian} give an arbitrarily large exponential power of $\varepsilon$.

On the intermediate interval $T\le t\le T_2$, split at $D_x=C_0L$. The constant part is bounded by
\[
 C_0L\int_T^\infty B_\varepsilon(t)\dd t
 =o(\varepsilon L).
\]
The excess is controlled by the same volume argument:
\[
 C\varepsilon^{C_0}(T_2+L)^{N+2}
 +O(\varepsilon^{P_0}).
\]
Choose $C_0>4(N+2)+2$ and then $P_0$ sufficiently large. This contribution is also $o(\varepsilon L)$.

For $t\ge T_2$, choose $\sigma=1/16$ in \eqref{eq:absorption-tail}. The absolute centered absorption is bounded by
\[
 C\int_{T_2}^\infty t^{-11/8}(1+\log t)\dd t
 =O(\varepsilon^{3/2}L)
 =o(\varepsilon L).
\]
Finally, \eqref{eq:JK} makes the source first moment in \eqref{eq:amplitude-absorption} $O_\delta(\varepsilon)$. Combining the three time intervals yields
\[
 \kap^{-1}a_{1/\varepsilon}(\theta)
 \le2(1+\gamma_0)\varepsilon L I_N(\theta)
       +o_\delta(\varepsilon L).
\]
For each fixed $\gamma_0$, all error estimates are uniform on $\Gamma_\delta$. Let $\varepsilon\downarrow0$ and then $\gamma_0\downarrow0$ to conclude.
\end{proof}

Propositions~\ref{prop:lower} and \ref{prop:upper} prove \eqref{eq:strong-a}. Since $I_N$ is positive and continuous, it has a positive minimum on each fixed receiving cone. Taking logarithms therefore gives \eqref{eq:strong-tau} with the stated uniformity, completing the proof of Theorem~\ref{thm:strong}.

\section{Grazing directions and equatorial saturation}
\label{sec:grazing}
We first determine the singular behavior of the interface input as a receiving direction approaches the equator. We then treat exactly equatorial directions through the critical moment, whose interface contribution vanishes there. These two arguments give different strong-resistance limits and explain their failure to commute.

\subsection{The tangential input tail}
Write
\[
 \theta_s=(b\eta,s),\qquad
 b=\sqrt{1-s^2},\qquad \lambda=1-b,
 \qquad \eta\in\Sph^{N-2}.
\]
Set
\[
 F_b(t,\eta)=\e^{-2t}\int_{\R^{N-1}}
       \e^{b\eta\cdot x'}P(t,x',0)\dd x',
 \qquad \alpha=\frac{N-2}{2}.
\]
Since $P=p_\infty$ on the interface,
\[
 I_N(\theta_s)=\int_0^\infty F_b(t,\eta)\dd t.
\]
The logarithmic singularity of this integral comes from the large-time behavior of $F_b$ when $\lambda t$ remains bounded.

\begin{lemma}[Tangential Laplace asymptotics]
\label{lem:tangential}
For each fixed $M<\infty$, as $t\to\infty$,
\begin{equation}\label{eq:tangent-tail}
 t\e^{2\lambda t}F_b(t,\eta)\longrightarrow
 A_N(\eta):=2(4\pi)^{(N-2)/2}a_\Neu(\eta,0)
\end{equation}
uniformly in $\eta$ and $0\le\lambda t\le M$.
Moreover, for $b$ sufficiently close to $1$ and $t\ge2$,
\begin{equation}\label{eq:tangent-bound}
 F_b(t,\eta)\le
 C\e^{-2\lambda t}t^{-1+k\lambda},
\end{equation}
with $C$ independent of $b,t,\eta$.
\end{lemma}

\begin{proof}
Use polar coordinates on the $(N-1)$-dimensional interface. With $d=N-1$, the angular Laplace formula is
\[
 \int_{\Sph^{d-1}}\e^{br\eta\cdot\omega'}
       g(\omega')\dd\omega'
 \sim
 \e^{br}\left(\frac{2\pi}{br}\right)^{(d-1)/2}g(\eta).
\]
For $d=1$, the angular integral is a two-point sum: the forward term gives the displayed expression, while the backward term is exponentially smaller.

For the profiles considered here, the formula can be used uniformly in $\eta$. On compact positive longitudinal intervals, homogeneous leading-edge convergence for $P$ is uniform in direction, and its limiting angular amplitude is continuous. The radial Gaussian envelope controls the remaining longitudinal ranges and supplies an integrable majorant.

First consider $r\le m(t)$. Using $P\le1$ and the angular upper bound gives a contribution
\[
 O\bigl(\e^{-2\lambda t}t^{-2+k\lambda}\bigr).
\]
For $r>m(t)$, write
\[
 r=m(t)+\rho\sqrt t,\qquad
 P(t,r\omega)=\sqrt t\,\e^{-\rho\sqrt t}
             w_P(\log t,\rho,\omega).
\]
Combining the angular factor, the radial Jacobian, and $\dd r=\sqrt t\,\dd\rho$ yields, when $\lambda t\le M$,
\begin{equation}\label{eq:trace-main}
 \begin{aligned}
 F_b(t,\eta)
 ={}&\left(\frac{4\pi}{b}\right)^\alpha
       \e^{-2\lambda t}t^{-1+k\lambda}
       \int_0^\infty\e^{-\lambda\sqrt t\rho}
       w_P(\log t,\rho,(\eta,0))\dd\rho\\
 &+o\bigl(\e^{-2\lambda t}t^{-1}\bigr).
 \end{aligned}
\end{equation}
The error is uniform in $\eta$ and $\lambda t\le M$. To obtain this expression, first restrict $\rho$ to a compact subset of $(0,\infty)$, where $r/(2t)\to1$ uniformly, and then remove the restriction using the Gaussian envelope. The contribution near $\rho=0$ is controlled by the same envelope.

In this regime,
\[
 \lambda\sqrt t\to0,\qquad
 t^{k\lambda}\to1,\qquad b\to1,
\]
uniformly, while
\[
 w_P(\log t,\rho,(\eta,0))
 \longrightarrow
 a_\Neu(\eta,0)\rho\e^{-\rho^2/4}.
\]
Since
\[
 \int_0^\infty\rho\e^{-\rho^2/4}\dd\rho=2,
\]
dominated convergence in \eqref{eq:trace-main} proves \eqref{eq:tangent-tail}.

For the upper bound, use the angular Gaussian estimate and the radial envelope directly, without imposing $\lambda t\le M$. The factor $\e^{-\lambda\sqrt t\rho}$ is at most one, and the remaining longitudinal integral is uniformly bounded. Together with the interior estimate, this gives \eqref{eq:tangent-bound}.
\end{proof}

\subsection{Proof of Theorem~\ref{thm:grazing}}
Fix a sufficiently large $R$. For $\lambda>0$ small enough that $R<1/\lambda$, \eqref{eq:tangent-tail} gives
\[
 F_b(t,\eta)
 =\frac{\e^{-2\lambda t}}t
       \bigl[A_N(\eta)+o_R(1)\bigr],
 \qquad R\le t\le1/\lambda,
\]
where the error tends to zero uniformly in $\eta$ and the indicated range as $R\to\infty$. Moreover,
\[
 \int_R^{1/\lambda}\frac{\e^{-2\lambda t}}t\dd t
 =\log(1/\lambda)-\log R+O(1).
\]
Hence
\[
 \int_R^{1/\lambda}F_b(t,\eta)\dd t
 =
 \bigl[A_N(\eta)+o_R(1)\bigr]
 \bigl[\log(1/\lambda)-\log R+O(1)\bigr].
\]

For fixed $R$, the integral over $[0,R]$ is uniformly bounded by the finite-time Gaussian estimates. For the remaining tail, \eqref{eq:tangent-bound} gives
\[
 \int_{1/\lambda}^\infty F_b(t,\eta)\dd t
 \le C\lambda^{-k\lambda}
       \int_1^\infty\e^{-2v}v^{-1+k\lambda}\dd v
 \le C.
\]
Here $\lambda^{-k\lambda}$ remains bounded as $\lambda\downarrow0$.

Divide by $\log(1/\lambda)$, first let $\lambda\downarrow0$ with $R$ fixed, and then let $R\to\infty$. We obtain
\[
 \frac{I_N(\theta_s)}{\log(1/\lambda)}
 \longrightarrow A_N(\eta)
\]
uniformly in $\eta$. Since
\[
 \lambda=\frac{s^2}{2}+O(s^4),\qquad
 \log(1/\lambda)=2\log(1/s)+O(1),
\]
this proves \eqref{eq:grazing-I}. In particular,
\begin{equation}\label{eq:grazing-coefficient}
 2\kap I_N(\theta_s)
 =\frac2\pi a_\Neu(\eta,0)\log(1/s)(1+o(1)).
\end{equation}
The amplitude $a_\Neu$ has a positive minimum on the equator, so the multiplicative error is uniform. Continuity and positivity of $a_0$, together with the definition of $C_N$, now give \eqref{eq:grazing-C}.

\begin{remark}[Scope of the expansion]\label{rem:scope}
For each fixed $\beta>1$, continuity of $a_\beta$ on the closed receiving hemisphere makes $\beta a_\beta(\theta)/\log\beta$ bounded. By contrast, the coefficient in \eqref{eq:grazing-coefficient} diverges as the direction approaches the equator. Thus the strong-resistance amplitude expansion cannot hold uniformly on the entire open receiving hemisphere.

The delay expansion is likewise nonuniform there. For fixed $\beta>1$, its proposed right-hand side tends to $-\infty$ as $s\downarrow0$, whereas $\tau_\beta$ has a finite continuous equatorial limit. Theorem~\ref{thm:grazing} describes the grazing behavior of the coefficient obtained after the strong-resistance limit. It does not provide a joint angle--resistance transition law.
\end{remark}

\subsection{The equatorial strong-resistance limit}
Let $\vartheta=(\eta,0)$. At this direction, the interface contribution in \eqref{eq:moment-balance} vanishes. Use the resistance-independent equatorial tail in Proposition~\ref{prop:moment}, with $\sigma=1/16$ and $\mu=3/8$.

Here $E_\infty(t,\vartheta)$ denotes the critical moment of the closed problem $\beta=\infty$, and $E_\infty^\infty(\vartheta)$ its large-time limit. For every sufficiently large fixed $T$,
\begin{align*}
 |E_\beta^\infty(\vartheta)
       -E_\infty^\infty(\vartheta)|
 &\le
 |E_\beta(T,\vartheta)-E_\infty(T,\vartheta)|\\
 &\quad+2CT^{-\mu}(1+\log T).
\end{align*}
Lemma~\ref{lem:finite-boundary-continuity}, together with the common fixed-time Gaussian bound, makes the first term tend to zero as $\beta\to\infty$, uniformly in $\eta$. Taking $\beta\to\infty$ first and then $T\to\infty$ therefore proves uniform convergence of the limiting moments.

In the closed problem, the receiving density is zero, while the donor density is the lower-half restriction of $P$. The equatorial weight is independent of $x_N$, and $P$ is even in that variable. Thus, for every $t$,
\begin{equation}\label{eq:half-moment}
 E_\infty(t,(\eta,0))
 =\frac12E_P(t,(\eta,0)).
\end{equation}
Combining this exact identity with the homogeneous amplitude formula gives
\[
 a_\beta(\eta,0)\longrightarrow
 \frac12a_\Neu(\eta,0)
\]
uniformly on the equator, proving the first limit in \eqref{eq:saturation}. Positivity and continuity of $a_\Neu$ allow passage through the logarithm, giving the delay limit in the same equation. Under concavity, Theorem~\ref{thm:parameters} also gives strict increase of the delay with resistance.

Assume now \eqref{ass:C}. To prove the strict range of the limiting delay, compare $P$ with the transparent solution $U=u^0$. Their initial densities differ by the nonzero reflected seed, so $P\ge U$. Their difference satisfies
\[
 (P-U)_t
 =\Delta(P-U)+[1-c_g(P,U)](P-U),
\]
where
\[
 c_g(P,U)=\int_0^1g'\bigl(U+r(P-U)\bigr)\dd r,
 \qquad
 0\le c_g(P,U)\le C_f(P+U).
\]
The potential is bounded and has the required exterior envelope. The reflected initial seed gives a positive compact seed at a positive time. Lemma~\ref{lem:persistence} therefore yields a strictly positive gap between the directional amplitudes:
\[
 a_\Neu(\eta,0)>a_0(\eta,0).
\]

For the opposite bound, fix any $\beta_1>0$. Strict resistance ordering and the limit already proved give
\[
 \frac12a_\Neu(\eta,0)
 \le a_{\beta_1}(\eta,0)<a_0(\eta,0).
\]
Consequently,
\[
 a_0(\eta,0)<a_\Neu(\eta,0)<2a_0(\eta,0),
\]
which is equivalent to \eqref{eq:saturation-range}.

Finally, for each fixed $s>0$, Theorem~\ref{thm:strong} gives
\[
 a_\beta(\theta_s)\to0,\qquad
 \tau_\beta(\theta_s)\to+\infty
 \qquad(\beta\to\infty).
\]
For each fixed finite $\beta$, continuity in direction instead gives
\[
 a_\beta(\theta_s)\to a_\beta(\eta,0),\qquad
 \tau_\beta(\theta_s)\to\tau_\beta(\eta,0)
 \qquad(s\downarrow0).
\]
Combining these statements with \eqref{eq:saturation} proves the asserted noncommuting iterated limits. This completes the proof of Theorem~\ref{thm:saturation}.

\section{General linear growth rate}\label{sec:scaling}
The normalization $f'(0)=1$ can be removed by rescaling
time and space. Assume that $f$ satisfies the corresponding
smooth KPP hypotheses with
\[
 f'(0)=r>0,\qquad 0<f(u)\le ru\quad(0<u<1).
\]
Set
\[
 \widehat t=rt,\qquad
 \widehat x=\sqrt r\,x,\qquad
 \widehat f(u)=\frac{f(u)}r,\qquad
 \widehat\beta=\sqrt r\,\beta,
\]
and define
\[
 \widehat u(\widehat t,\widehat x)
 =u\left(\frac{\widehat t}{r},
          \frac{\widehat x}{\sqrt r}\right).
\]
Then $\widehat u$ solves the normalized problem with
reaction $\widehat f$, resistance $\widehat\beta$, and
initial density
\[
 \widehat u_0(\widehat x)
 =u_0(\widehat x/\sqrt r).
\]
In particular, $\widehat f'(0)=1$. The interface condition
retains its form because
\[
 u_+-u_-=\beta\,\partial_Nu_\pm
 =\widehat\beta\,\partial_{\widehat x_N}\widehat u_\pm.
\]
Concavity is also preserved by this rescaling.

Let $\widehat a_{\widehat\beta}(\theta)$ denote the
directional amplitude of this normalized problem.
The physical front center is
\[
 m_r(t)=\frac{m(rt)}{\sqrt r}
 =2\sqrt r\,t-\frac{N+2}{2\sqrt r}\log(rt),
\]
and its directional displacement is
\[
 \xi_\beta^{(r)}(\theta)
 =\frac1{\sqrt r}\log\widehat a_{\widehat\beta}(\theta).
\]
If $\widehat U_*$ is the normalized traveling wave for
$\widehat f$, the corresponding physical profile is
\[
 u_+\bigl(t,(m_r(t)+z)\theta\bigr)
 \longrightarrow
 \widehat U_*\bigl(\sqrt r\,z
                  -\log\widehat a_{\widehat\beta}(\theta)\bigr),
\]
with the same receiving-cone and equatorial interpretations
as in the normalized results.

Arrival times scale by $r^{-1}$. Thus the physical
arrival-time difference is
\[
 \tau_\beta(\theta)
 =\frac1{2r}\log
   \frac{\widehat a_0(\theta)}
        {\widehat a_{\widehat\beta}(\theta)}.
\]
For fixed $r>0$, the strong-resistance expansion becomes
\[
 \tau_\beta(\theta)
 =\frac1{2r}\log\widehat\beta
  -\frac1{2r}\log\log\widehat\beta
  +\frac1r\widehat C_N(\theta)+o(1)
 \qquad(\beta\to\infty),
\]
uniformly on each fixed receiving cone. Here
$\widehat C_N$ is computed from the normalized reaction
and the rescaled initial density. In particular, both
the constant term and the normalized input moment refer
to this rescaled problem.

\section{Further questions}\label{sec:discussion}
The results suggest several directions for further study.
The first is the joint limit in which $\beta\to\infty$
and $\theta_N\downarrow0$. The divergence of the delay
at each fixed receiving angle, together with its finite
equatorial limit, indicates a transition between these
regimes. The iterated limits established here determine
neither the angular scale of this transition nor its
limiting profile. A joint analysis would need to control
the Robin difference mode over the time range that
produces the logarithmic input tail.

A second question concerns differentiability at zero
resistance. Under concavity, joint continuity and strict
ordering determine how the phase and arrival-time
difference vary with resistance, but they do not yield
a first-order expansion of $\xi_\beta$ or $\tau_\beta$
as $\beta\downarrow0$. Such an expansion would require
quantitative control of the transmission perturbation
throughout the critical-moment time integral, including
a tail estimate strong enough to justify differentiation
with respect to resistance.

The angular regularity of the amplitude also remains
to be clarified beyond the H\"older estimate
\eqref{eq:holder-main} under \eqref{ass:K}.
The finite-time moment estimate
$|\partial_\zeta E_\beta(t,\theta)|\le C\sqrt t$
is uniform in direction and resistance, but grows
with time. It therefore does not by itself provide
a uniformly bounded angular derivative as
$t\to\infty$, or justify differentiation of the
limiting moment.
Improving the regularity, particularly near the
equator, would require sharper control of the
angular dependence of the interface contribution.

For nonconcave KPP reactions, the present argument does
not establish resistance ordering. The off-diagonal
coefficient in \eqref{eq:derivatives} may change sign,
so the cooperative comparison used here is unavailable.
Whether ordering persists under weaker assumptions,
or can fail within the smooth KPP class, is a separate
question.

Finally, curved interfaces and unequal diffusion or
reaction rates would extend the model beyond the
symmetric planar setting. Curvature changes the geometry
of the interface interaction, while unequal coefficients
remove the simple average--difference decomposition.
Both extensions would require new estimates for the
coupling between propagation and transmission.

\appendix
\section{A positive linear leading-edge supersolution}
\label{app:linear}
This appendix constructs a radial supersolution for the linear leading-edge equation. The construction applies both to a moving inner boundary and to the fixed physical boundary used in the sharp upper estimate.

Put $r=m(t)+\rho\sqrt t$ and define
\begin{equation}\label{eq:coefficients}
 L_0=\partial_{\rho\rho}+\tfrac\rho2\partial_\rho+1,
 \qquad
 p_{\rm rad}=\frac{(N-1)\sqrt t}{r}-\frac{k}{\sqrt t},
 \quad
 q_c=\frac{N-1}{2}-\frac{(N-1)t}{r},
 \quad b=\frac{t}{r^2}.
\end{equation}
The radial linear operator is $\mathcal P=L_0+p_{\rm rad}\partial_\rho+q_c$, with logarithmic time $\tau=\log t$. For a radial solution $U(t,r)$, the normalized field
\[
 W(\tau,\rho)
 =t^{-1/2}\e^{\rho\sqrt t}
       U(t,m(t)+\rho\sqrt t)
\]
satisfies
\[
 W_\tau
 =\mathcal PW-t^{1/2}\e^{\rho\sqrt t}g(U)
 \le\mathcal PW.
\]
We adapt the three-term construction of \cite[Section 3.1.1, equations (28)--(35) of the preprint version]{RRR}, where the nonlinear absorption is likewise discarded. The argument below verifies the Gaussian exponents, the coefficient errors in \eqref{eq:coefficients}, and the fixed-inner-boundary case.

\begin{lemma}[Moving and fixed inner boundaries]
\label{lem:linear-envelope}
Fix $0\le\sigma<1/2$ and $0<\mu<1/8$, and set
\[
 \eta_t=t^{\sigma-1/2},\qquad
 X=\rho+\eta_t,\qquad a_\mu=1/8+\mu.
\]
There is $T=T(N,\sigma,\mu)\ge2$ and a positive supersolution $\Phi_\tau\ge\mathcal P\Phi$ on $t\ge T$, $\rho\ge-\eta_t$, such that
\begin{equation}\label{eq:Phi-bounds}
 c\eta_t\e^{-a_\mu X^2}
 \le\Phi(\tau,\rho)
 \le C(X+\eta_t)\e^{-a_\mu X^2}.
\end{equation}
For every radial solution with compactly supported initial density in $[0,1]$ and reaction satisfying \eqref{ass:K}, a constant multiple of $\Phi$ dominates $W$ throughout this domain. The comparison constants may depend on the radial datum and on $N,\sigma,\mu$, but not on the interface resistance.

In physical coordinates, the inner boundary is $z=r-m(t)=-t^\sigma$. In particular, $\sigma=0$ gives $z=-1$. The choices $\mu=1/32$ and $\mu=3/32$ give the Gaussian exponents $5/32$ and $7/32$, respectively.
\end{lemma}

\begin{proof}
All coefficients below are evaluated at $\rho=X-\eta_t$. After increasing $T$, we have
\[
 r\ge2t-k\log t-t^\sigma\ge t.
\]
Define
\[
 \ell_1=p_{\rm rad}-\sigma\eta_t,\qquad
 \ell_2=q_c-\tfrac X4\ell_1,\qquad
 \mathcal M=-\partial_{XX}+\tfrac{X^2}{16}-\tfrac34.
\]
For $\Phi=\e^{-X^2/8}F(\tau,X)$, differentiation at fixed $\rho$, using $\partial_\tau\eta_t=(\sigma-1/2)\eta_t$, gives
\begin{equation}\label{eq:linear-conjugation}
 \e^{X^2/8}(\partial_\tau-\mathcal P)\Phi
 =\mathcal RF
 :=F_\tau+\mathcal MF-\ell_1F_X-\ell_2F.
\end{equation}
Here $F_\tau$ on the right is taken at fixed $X$.

The coefficient bounds follow from
\begin{align*}
 q_c&=\frac{(N-1)(r-2t)}{2r},
 &|p_{\rm rad}|&\le C_Nt^{-1/2},\\
 |q_c|&\le C_N\left(
 \frac{\log t}{t}+\frac{X+\eta_t}{\sqrt t}\right),
 &|q_c|&\le C_N.
\end{align*}
Consequently, for all $X\ge0$,
\begin{equation}\label{eq:linear-errors}
 |\ell_1|\le C\eta_t,\qquad
 |\ell_2|\le C\eta_t(1+X),\qquad
 |\ell_2|\le C(1+X).
\end{equation}
These estimates include $\sigma=0$, for which $\ell_1=p_{\rm rad}$ and $\log t/t=O(t^{-1/2})$. The coefficients become small on every fixed $X$-interval. We do not require $\ell_2$ to be uniformly small on the entire half-line.

We now choose the three spatial profiles. Let $\chi\in C^\infty(\R;[0,1])$ be nonincreasing, equal to one on $(-\infty,0]$ and zero on $[1,\infty)$. For example, take
\[
 E(s)=
 \begin{cases}
 \e^{-1/s},&s>0,\\
 0,&s\le0,
 \end{cases}
 \qquad
 \chi(s)=\frac{E(1-s)}{E(s)+E(1-s)}.
\]
Set
\begin{gather*}
 \phi_1(X)=\cos(\pi X),\qquad
 \gamma_1(X)=\chi(12X-3),\qquad
 \gamma_2(X)=1-\chi(X-1),\\
 \varphi_0(X)=X\e^{-X^2/8},\qquad
 G_\mu(X)=\e^{-\mu X^2}.
\end{gather*}
Thus $\gamma_1=1$ on $[0,1/4]$ and vanishes on $[1/3,\infty)$, while $\gamma_2=0$ on $[0,1]$ and equals one on $[2,\infty)$. Moreover,
\[
 \phi_1\gamma_1\ge0,\qquad
 -\phi_1''=\pi^2\phi_1,\qquad
 \mathcal M\varphi_0=0.
\]
The positive mode $\varphi_0$ bridges the two cutoff regions.

Define
\begin{equation}\label{eq:linear-Phi}
 \Phi=\e^{-X^2/8}\left[
 \frac{Q(\tau)}2\phi_1(X)\gamma_1(X)
 +\frac{Q(\tau)}2\gamma_2(X)G_\mu(X)
 +\zeta_{\rm lin}(\tau)X\e^{-X^2/8}
 \right].
\end{equation}
We first estimate its residual under the assumptions
\[
 Q,\zeta_{\rm lin}>0,\qquad Q'\ge-Q,\qquad \zeta_{\rm lin}'\ge0,
\]
where primes denote logarithmic-time derivatives. The coefficient equations chosen below will ensure all three conditions.

For $0\le X\le1/4$, we have $\phi_1\ge2^{-1/2}$ and $|\phi_1'|\le\pi\phi_1$. By \eqref{eq:linear-errors}, increasing $T$ gives
\[
 \pi^2+X^2/16-3/4-|\ell_1|\pi-|\ell_2|\ge1.
\]
Since $\varphi_0$ and $\varphi_0'$ are bounded on this interval, and $\zeta_{\rm lin}'\varphi_0\ge0$,
\begin{equation}\label{eq:linear-near}
 \mathcal RF\ge
 \left[\tfrac12(Q'+Q)-C_1\zeta_{\rm lin}\eta_t\right]\phi_1.
\end{equation}

For the far region, put $d_\mu=1/16-4\mu^2>0$. Direct calculation gives
\[
 (\mathcal M-\ell_1\partial_X-\ell_2)G_\mu
 =
 \left[d_\mu X^2+2\mu-3/4
       +2\mu\ell_1X-\ell_2\right]G_\mu.
\]
The bracket is bounded below by $d_\mu X^2-C(1+X)$, uniformly for $t\ge T$. Choose $R\ge2$ so that this lower bound is at least one for $X\ge R$. Also, \eqref{eq:linear-errors} and $\mu<1/8$ give
\[
 |\ell_1\varphi_0'+\ell_2\varphi_0|
 \le C\eta_t(1+X^2)\e^{-X^2/8}
 \le C_2\eta_tG_\mu.
\]
Since $\gamma_1=0$ and $\gamma_2=1$ in this region,
\begin{equation}\label{eq:linear-far}
 \mathcal RF\ge
 \left[\tfrac12(Q'+Q)-C_2\zeta_{\rm lin}\eta_t\right]G_\mu,
 \qquad X\ge R.
\end{equation}

On the remaining compact interval $1/4\le X\le R$, all cutoff profiles and their first two derivatives are bounded, and $\min\varphi_0>0$. Using $Q'\ge-Q$ and \eqref{eq:linear-errors}, we obtain
\begin{equation}\label{eq:linear-middle}
 \mathcal RF\ge
 \left[\zeta_{\rm lin}'-C_3(Q+\zeta_{\rm lin}\eta_t)\right]\varphi_0.
\end{equation}
Choose $K\ge\max\{2C_1,2C_2,C_3,1\}$ and solve
\begin{equation}\label{eq:linear-odes}
 Q'+Q=K\zeta_{\rm lin}\eta_t,\qquad
 \zeta_{\rm lin}'=K(Q+\zeta_{\rm lin}\eta_t)
\end{equation}
from positive initial values at $\tau_0=\log T$. Positivity is preserved, $Q'\ge-Q$, and $\zeta_{\rm lin}'\ge0$. Thus \eqref{eq:linear-near}--\eqref{eq:linear-middle}, together with \eqref{eq:linear-conjugation}, prove the supersolution inequality. The cutoffs and constants chosen up to this point depend only on $N,\sigma,\mu$.

We next control the time-dependent coefficients. Set $\mathcal V=Q+\zeta_{\rm lin}/K$. Then
\[
 \mathcal V'=(K+1)\zeta_{\rm lin}\eta_t
 \le K(K+1)\eta_t\mathcal V,
 \qquad
 \int_{\tau_0}^\infty\eta_{\e^s}\dd s<\infty.
\]
Gronwall's inequality and monotonicity of $\zeta_{\rm lin}$ imply
\[
 0<\zeta_{\rm lin}(\tau_0)\le\zeta_{\rm lin}(\tau)\le C.
\]
With $\alpha=1/2-\sigma\in(0,1/2]$, variation of constants gives
\[
 Q(\tau)=\e^{-(\tau-\tau_0)}Q(\tau_0)
 +K\int_{\tau_0}^{\tau}
       \e^{-(\tau-s)}\zeta_{\rm lin}(s)\e^{-\alpha s}\dd s.
\]
Dividing by $\eta_t=\e^{-\alpha\tau}$ yields
\[
 c\eta_t\le Q(\tau)\le C\eta_t
 \qquad(\tau\ge\tau_0).
\]
These constants may also depend on the positive initial coefficients.

The near-origin term, the positive mode on compact subintervals of $X>0$, and the far Gaussian term now give \eqref{eq:Phi-bounds}. In particular,
\[
 \Phi(\tau,-\eta_t)=Q(\tau)/2\ge c\eta_t.
\]
The upper bound uses $a_\mu<1/4$, so that the mode $X\e^{-X^2/4}$ decays faster than the far Gaussian.

It remains to fit the comparison data. Take a compact radial upper datum $U_0\in[0,1]$ as in Section~\ref{sec:prelim}, supported in $\{r\le R_0\}$. At the fixed starting time $T$, linear heat comparison gives
\[
 U(T,r)\le C_T
 \exp\left(-\frac{(r-R_0)_+^2}{4T}\right).
\]
Fix $a_\mu<\nu<1/4$, substitute
\[
 r=m(T)+\sqrt T(X-\eta_T),
\]
and multiply by $T^{-1/2}\e^{\sqrt T(X-\eta_T)}$. Absorbing the resulting linear term in $X$ into the Gaussian yields
\begin{equation}\label{eq:linear-initial}
 W(\tau_0,X-\eta_T)
 \le C_{T,\nu}\e^{-\nu X^2},
 \qquad X\ge0.
\end{equation}
Because $\nu>a_\mu$, a constant multiple of the lower bound in \eqref{eq:Phi-bounds} dominates this initial field on the entire half-line.

At the moving inner boundary, $U\le1$ gives
\begin{equation}\label{eq:linear-wall}
 W(\tau,-\eta_t)
 \le t^{-1/2}\e^{-t^\sigma}
 =\eta_t t^{-\sigma}\e^{-t^\sigma}
 \le\eta_t.
\end{equation}
For $\sigma=0$, the bound is exactly $\e^{-1}\eta_t$ before the last inequality. Since \eqref{eq:linear-odes} is linear and homogeneous in $(Q,\zeta_{\rm lin})$, multiplying both initial coefficients by one sufficiently large constant fits \eqref{eq:linear-initial} and \eqref{eq:linear-wall}. We continue to denote the rescaled supersolution by $\Phi$.

We finally justify comparison at spatial infinity. On each finite logarithmic-time interval $[\tau_0,\tau_1]$, the heat bound implies that $W\to0$ uniformly as $X\to\infty$. Add
\[
 \varepsilon\e^{A(\tau-\tau_0)}
\]
to $\Phi$, where $A\ge\sup(1+q_c)$ on the domain. This added function is a supersolution in the original $\rho$ coordinates. For fixed $\varepsilon>0$ and sufficiently large $L$, it dominates $W$ on the outer boundary $X=L$.

Apply comparison on the moving strip $0\le X\le L$, then let $L\to\infty$ and $\varepsilon\downarrow0$. This proves $W\le\Phi$ on $[\tau_0,\tau_1]$. Since $\tau_1$ is arbitrary, the comparison holds for all later times. The interface solutions are first bounded by the common radial solution $U$, so the resulting constants are uniform for $\beta\in[0,\infty]$, including both interface traces.

For $z\ge0$, multiply the upper bound in \eqref{eq:Phi-bounds} by $\sqrt t\,\e^{-z}$ to obtain
\[
 U(t,m(t)+z)
 \le C(z+2t^\sigma)\e^{-z}
 \exp\left[
 -a_\mu\left(\frac z{\sqrt t}+\eta_t\right)^2
 \right].
\]
This gives \eqref{eq:envelope}; the choice $\sigma=0$ also gives \eqref{eq:sharp-upper}. The bounded interval $2\le t\le T$ is covered by the compact-data heat bound after increasing the constant. The construction uses neither reaction-specific wave convergence nor concavity.
\end{proof}

\section{Tightness of the Cartesian critical moment}
\label{app:tightness}
We justify the passage from local Cartesian convergence
to the limiting critical moment in \eqref{eq:cart-moment}.
Use the coordinates \eqref{eq:cartesian}, write
$X=(\rho,v,Y)$, and define the positive measure
\[
 \dd\nu_t(X)=W_S(\log t,X)\dd X
 \qquad\text{on }\R\times\R^{N-2}\times(0,\infty).
\]
We suppress its dependence on the resistance and the
equatorial direction. The estimates below are uniform
in both parameters, since they use only the common
radial upper bound.

For sufficiently large $T$, we claim that
\begin{align}
 \sup_{t\ge T}\int(1+|\rho|)\dd\nu_t
 &<\infty,\label{eq:tight-one}\\
 \lim_{M\to\infty}\limsup_{t\to\infty}
 \int_{|\rho|+|v|+Y>M}(1+|\rho|)\dd\nu_t
 &=0,\label{eq:tight-two}\\
 \lim_{t\to\infty}
 \int_{\rho<0}|\rho|\dd\nu_t
 &=0.\label{eq:tight-three}
\end{align}
All unqualified integrals in this appendix are over
the folded Cartesian domain $Y>0$.

For $\vartheta=(\eta,0)$, the coordinate change gives
\begin{equation}\label{eq:nu-change}
 \int F(X)\dd\nu_t
 =t^{1/2}\e^{-2t}\int_H
 F(X_t(x))\e^{\vartheta\cdot x}S_\beta(t,x)\dd x,
\end{equation}
where $X_t$ is the inverse Cartesian coordinate map.
Indeed,
$\rho\sqrt t=\vartheta\cdot x-m(t)$ and
$\dd x=t^{N/2}\dd X$; the identity follows from
$m(t)=2t-k\log t$ and $k=(N+2)/2$.

Fix $0<\sigma<1/4$. First consider
$m(t)\le r=|x|\le4t$, where $r\asymp t$.
Set
\[
 q=\frac{r-m(t)}{\sqrt t}\ge0,
 \qquad b=\sqrt r\,\alpha,
\]
with $\alpha$ the angle between $x$ and $\vartheta$.
The elementary angular bounds give
\[
 \e^{-r(1-\cos\alpha)}\le\e^{-cb^2},
 \qquad
 \dd\omega\le
 Cr^{-(N-1)/2}b^{N-2}\dd b\,\dd\upsilon,
\]
where $\upsilon$ denotes the remaining angular variables,
whose total measure is bounded. In dimension $N=2$,
these variables form a finite set, and the same estimate
holds.

Insert the radial envelope into \eqref{eq:nu-change}.
After integration in $\upsilon$, the resulting measure
is bounded by
\begin{equation}\label{eq:tight-dominant}
 C(q+t^{\sigma-1/2})
 \e^{-cq^2}\e^{-cb^2}b^{N-2}\dd q\,\dd b.
\end{equation}
There is no residual power of time. The relevant factors
are
\[
 t^{1/2}t^{-k}t^{(N-1)/2}\sqrt t\sqrt t
 =t^{(N+2)/2-k}=1.
\]
Here the last two factors come from $\dd r=\sqrt t\,\dd q$
and the prefactor $r-m(t)+t^\sigma$ in the radial envelope.

The Cartesian coordinates satisfy
\[
 \rho=q-\frac{r(1-\cos\alpha)}{\sqrt t},
 \qquad
 |\rho|\le q+\frac{b^2}{2\sqrt t},
\]
and
\[
 |(v,Y)|=\frac{r\sin\alpha}{\sqrt t}\le2b.
\]
Since $t^{\sigma-1/2}\le1$ for large $t$, multiplying
\eqref{eq:tight-dominant} by $1+|\rho|$ gives a bound
by the fixed integrable function
\[
 C(1+q)(1+q+b^2)
 \e^{-cq^2-cb^2}b^{N-2}.
\]
Moreover,
\[
 |\rho|+|v|+Y
 \le q+\frac{b^2}{2\sqrt t}+2\sqrt2\,b.
\]
Thus bounded sets of $(q,b)$ map into uniformly bounded
Cartesian sets. The Gaussian majorant controls both
the total weighted mass and all large-$X$ tails in
this region.

We next consider the interior region $r<m(t)$.
For $r\le t$, the bound $S_\beta\le1$ and
$\vartheta\cdot x\le r$ make the contribution to
\eqref{eq:nu-change} exponentially small, even after
multiplication by $1+|\rho|$.

On $t<r<m(t)$, set $z=r-m(t)<0$ and again use the
angular Gaussian estimate. The unweighted measure is
bounded, after angular integration, by
\[
 Ct^{-1}\e^z\e^{-cb^2}b^{N-2}\dd z\,\dd b.
\]
In this region,
\[
 |\rho|
 =\frac{|z|+r(1-\cos\alpha)}{\sqrt t}
 \le\frac{|z|+b^2/2}{\sqrt t}.
\]
Integrability of $\e^z(1+|z|)$ on $(-\infty,0)$ therefore
gives mass $O(t^{-1})$ and first absolute
$\rho$-moment $O(t^{-3/2})$. In particular, the entire
interior contribution to the weighted mass tends to zero.

For $r>4t$, use the exterior radial envelope in
\eqref{eq:nu-change}. Since $m(t)<2t$ for large $t$,
we have $r-m(t)>r/2$, and hence
\[
 \exp[-c(r-m(t))^2/t]\le\exp[-cr^2/(4t)].
\]
All remaining radial, angular, and moment factors are
polynomial and can be absorbed into this Gaussian.
Thus the mass and first absolute $\rho$-moment in
$r>4t$ are exponentially small.

These estimates prove \eqref{eq:tight-one} and
\eqref{eq:tight-two}. To prove \eqref{eq:tight-three},
it remains only to examine the main region
$m(t)\le r\le4t$. There, $\rho<0$ implies
\[
 0\le q<
 \frac{r(1-\cos\alpha)}{\sqrt t}
 \le\frac{b^2}{2\sqrt t}.
\]
On bounded $(q,b)$ sets, this also gives
$|\rho|\le b^2/(2\sqrt t)$, so the negative first-moment
contribution tends to zero. The complement is uniformly
controlled by the integrable Gaussian majorant.
Together with the interior and far-region estimates,
this proves \eqref{eq:tight-three}.

Finally, rewrite \eqref{eq:cart-moment} as
\[
 E_\beta(t,(\eta,0))
 =2\int\rho\,\dd\nu_t
 -\frac{2k\log t}{\sqrt t}\int\dd\nu_t.
\]
The second term tends to zero by \eqref{eq:tight-one},
and the negative part of the first term tends to zero
by \eqref{eq:tight-three}.

For the positive part, local convergence applies on
bounded Cartesian sets separated from $\rho=0$.
The omitted strip satisfies
\[
 \int_{0<\rho<h}\rho\,\dd\nu_t
 \le h\int\dd\nu_t\le Ch,
\]
uniformly in $t$. Large Cartesian tails are controlled
by \eqref{eq:tight-two}. We may therefore first pass
to the limit on bounded sets with $\rho\ge h$, then
let the spatial cutoff tend to infinity and $h\downarrow0$.
This justifies passage to the limiting integral in
\eqref{eq:cart-moment}, including along the parameter
and direction sequences used in the compactness argument.

\section{Scalar expanding-interval matching}\label{app:energy}
We adapt the energy argument of \cite[Lemma 4.1]{Wang}
to the variable radial drift used in
Lemma~\ref{lem:coarse}. The drift is treated directly
in the energy identity, using only its size and without
differentiating it.

\begin{lemma}[Expanding-interval matching]\label{lem:energy}
Assume \eqref{ass:K}. Fix $\alpha>0$ and
$0<\gamma<1/3$, and put $R(t)=t^\gamma$.
Let $q$ solve
\[
 q_t=q_{zz}+(2+d(t,z))q_z+f(q),
 \qquad -R(t)<z<R(t),
\]
where $d$ is smooth and $|d(t,z)|\le C/t$.
Suppose $0\le q\le1$ and
\[
 q(t,-R(t))=h_b(t),\qquad
 q(t,R(t))=\alpha R(t)\e^{-R(t)},
\]
with $h_b\in C^1$, $0\le h_b\le1$, and
$\sup|h_b'|<\infty$.

For any bounded initial profile in $[0,1]$, posed at a
sufficiently large initial time so that the right boundary
value lies in $[0,1]$, the solution satisfies
\[
 q(t,z)\longrightarrow U_*(z-\log\alpha)
 \qquad(t\to\infty)
\]
locally uniformly in $z$.
\end{lemma}

\begin{proof}
Set
\[
 V=\e^zq,\qquad
 G(u)=\int_0^u g(s)\dd s,\qquad
 P(z,V)=\e^{2z}G(\e^{-z}V).
\]
The transformed equation is
\[
 V_t-V_{zz}-d(t,z)(V_z-V)+P_V(z,V)=0,
 \qquad
 P_V(z,V)=\e^zg(\e^{-z}V)\ge0.
\]
The boundary values become
\[
 V(t,-R(t))=b(t):=\e^{-R(t)}h_b(t),
 \qquad V(t,R(t))=\alpha R(t).
\]

Define the renormalized energy
\[
 \mathcal F(t)=
 \int_{-R(t)}^{R(t)}
 \left(\frac12V_z^2+P(z,V)\right)\dd z
 -\frac12\alpha^2R(t).
\]
Completing the square on the positive half-interval gives
the exact identity
\begin{equation}\label{eq:energy-positive}
 \begin{aligned}
 \mathcal F+\alpha
 ={}&\frac12\int_{-R}^0V_z^2\dd z
 +\frac12\int_0^R(V_z-\alpha)^2\dd z\\
 &+\int_{-R}^RP(z,V)\dd z
 +\alpha(1-V(t,0))
 \ge0.
 \end{aligned}
\end{equation}
Here $V(t,0)=q(t,0)\le1$ and $P\ge0$.

This identity also controls the drift error. On $[0,R]$,
write
\[
 V(z)=V(0)+\alpha z+\int_0^z(V_z-\alpha)\dd s.
\]
On $[-R,0]$, use the corresponding expression in terms
of $V_z$. Cauchy--Schwarz and
\eqref{eq:energy-positive} then yield, for $R\ge1$,
\[
 \int_{-R}^R(V_z-V)^2\dd z
 \le CR^2(\mathcal F+\alpha+1)+CR^3.
\]

We next differentiate the energy, including the moving
endpoints. At the right boundary,
\[
 V_t(t,R(t))=R'(t)(\alpha-V_z(t,R(t))).
\]
After including the derivative of
$-\alpha^2R/2$, the right endpoint contributes
\[
 -\frac{R'}2(V_z-\alpha)^2
 +R'P(R,\alpha R).
\]
At the left boundary,
\[
 V_t(t,-R(t))=b'(t)+R'(t)V_z(t,-R(t)),
\]
so its contribution is
\[
 -V_zb'-\frac{R'}2V_z^2+R'P(-R,b).
\]
Since $R'>0$, this is bounded above by
\[
 \frac{|b'|^2}{2R'}+R'\e^{-2R}G(h_b).
\]

All positive boundary remainders are integrable in time.
Indeed, $G(u)\le Cu^3$ on $[0,1]$ gives
\[
 P(R,\alpha R)\le C R^3\e^{-R},
\]
and
\[
 b'=\e^{-R}(h_b'-R'h_b).
\]
The boundedness of $h_b'$ and the exponential factors
therefore control the left boundary terms as well.

The interior contribution is
\[
 -\int_{-R}^RV_t^2\dd z
 +\int_{-R}^RV_t\,d(t,z)(V_z-V)\dd z.
\]
Young's inequality and $|d|\le C/t$ give
\[
 \left|\int_{-R}^RV_t\,d(V_z-V)\dd z\right|
 \le\frac12\int_{-R}^RV_t^2\dd z
 +Ct^{-2}\int_{-R}^R(V_z-V)^2\dd z.
\]
Combining the preceding estimates yields
\begin{equation}\label{eq:energy-dissipation}
 \mathcal F'
 \le-\frac12\int_{-R}^RV_t^2\dd z
 +Ct^{2\gamma-2}(\mathcal F+\alpha+1)
 +Ct^{3\gamma-2}
 +C(1+R^3)\e^{-R}.
\end{equation}
No derivative of $d$ occurs.

The two power errors are integrable because $\gamma<1/3$.
Applying Gronwall's inequality to
$\mathcal F+\alpha+1$, which is positive by
\eqref{eq:energy-positive}, gives a uniform energy bound.
Integrating \eqref{eq:energy-dissipation} and using the
lower bound $\mathcal F\ge-\alpha$ then gives
\[
 \int_{t_1}^\infty\int_{-R(t)}^{R(t)}
 V_t^2\dd z\dd t<\infty
\]
for any fixed starting time $t_1$ after parabolic
regularization.

Consider a sequence $t_j\to\infty$. On each fixed
space--time cylinder, the translated functions
$V(t_j+s,z)$ are uniformly bounded because
$0\le V\le\e^z$. Local parabolic compactness gives a
subsequential limit. The total dissipation on every
such translated cylinder tends to zero, so the limit
is independent of $s$. Since $d(t_j+s,z)\to0$ uniformly
on compact sets, the limiting equation is
\[
 V_\infty''=\e^zg(\e^{-z}V_\infty)\ge0.
\]
The passage to this equation uses local energy bounds
and does not require uniform derivative bounds for $d$.

Lower semicontinuity in \eqref{eq:energy-positive} gives
\[
 \int_0^\infty(V_\infty'-\alpha)^2\dd z<\infty.
\]
Convexity makes $V_\infty'$ nondecreasing. Its square
integrability relative to $\alpha$ therefore forces
\[
 V_\infty'(z)\longrightarrow\alpha,\qquad
 V_\infty(z)=\alpha z+o(z)
 \quad(z\to+\infty).
\]
In particular, the limiting profile is nonzero.

Set $Q=\e^{-z}V_\infty$. Then
\[
 Q''+2Q'+f(Q)=0,\qquad
 0\le Q\le1,\qquad Q(+\infty)=0.
\]
Its mechanical energy
\[
 \mathcal H(z)=
 \frac12Q'(z)^2+\int_0^{Q(z)}f(s)\dd s
\]
satisfies $\mathcal H'=-2Q'^2$.
Boundedness of the entire solution and the ODE give
bounded derivatives. The limiting sets at either end
therefore consist of equilibria, as follows from the
energy identity; since the only zeros of $f$ in $[0,1]$
are $0$ and $1$, the end limits exist and belong to
$\{0,1\}$.

The left limit cannot be zero. Otherwise the mechanical
energy would have the same value zero at both ends,
forcing $Q'\equiv0$, contrary to the nonzero
right-hand tail. Hence $Q(-\infty)=1$.
Uniqueness for the ODE gives $0<Q<1$ at finite $z$.
At any critical point,
\[
 Q''=-f(Q)<0.
\]
Thus every critical point would be a strict local
maximum. The end limits $1$ and $0$, together with
$0<Q<1$, exclude such a point, so $Q$ is strictly
decreasing.

Wave uniqueness now identifies $Q$ as a translate of
$U_*$. Its leading-edge coefficient is $\alpha$, because
$\e^zQ(z)=V_\infty(z)=\alpha z+o(z)$. The normalization
of $U_*$ therefore gives
\[
 Q(z)=U_*(z-\log\alpha).
\]
Every subsequential limit is the same, proving local
uniform convergence of the full solution.

The energy calculations may be started at a positive
time after the initial boundary corners. Thus no
compatibility between the bounded initial profile
and the boundary data is needed.
\end{proof}

\section{An auxiliary logistic radial estimate}\label{app:logistic}
This appendix supplies the scalar lower estimate used
in \eqref{eq:sharp-lower}. The auxiliary logistic reaction
is used only for comparison and does not impose an
additional assumption on the original reaction.

Let $w$ solve
\[
 w_t=\Delta w+w-C_fw^2
\]
from a nonzero compactly supported radial seed with
values in $[0,1/C_f]$. In this appendix only, $u$ denotes
the auxiliary rescaled solution $u=C_fw$, rather than the
interface solution. Then $u$ solves
the standard logistic equation and satisfies $0\le u\le1$.
Its scalar leading-edge amplitude $a>0$ follows from
the radial envelope, the scalar moment identity, and
Lemma~\ref{lem:persistence}. The relative barriers of
Section~\ref{sec:meso} give the corresponding mesoscopic
boundary asymptotics.

Fix $0<\gamma<1/3$ and a constant $\kappa\ge1$, and put
\[
 \ell(t)=\kappa t^\gamma,\qquad \xi=\log a.
\]
In the radial moving frame, define
\[
 v(t,z)=\e^z u(t,m(t)+z).
\]
For sufficiently large $t$ and $|z|\le\ell(t)$,
\[
 v_t=v_{zz}+d(t,z)(v_z-v)-\e^{-z}v^2,
 \qquad
 d(t,z)=-\frac{k}{t}+\frac{N-1}{m(t)+z},
 \qquad |d(t,z)|\le C/t.
\]
The relative mesoscopic estimate applies at
$z=\ell(t)$ because $\ell(t)/t^\gamma=\kappa$ is fixed.
Thus
\[
 \frac{v(t,\ell(t))}{\ell(t)}\longrightarrow a.
\]

Let $U_{\log}$ be the normalized logistic wave.
For a fixed $\epsilon>0$, introduce the stationary
weighted profiles
\[
 V_-(z)=\e^zU_{\log}(z-\xi+\epsilon),\qquad
 V_+(z)=\e^zU_{\log}(z-\xi-\epsilon).
\]
They satisfy
\[
 V_\pm''-\e^{-z}V_\pm^2=0,
 \qquad
 \frac{V_\pm(z)}z\longrightarrow a\e^{\pm\epsilon}
 \quad(z\to+\infty).
\]
Consequently, for sufficiently large $t$,
\[
 V_-(\ell(t))\le v(t,\ell(t))\le V_+(\ell(t)).
\]
The fixed shifts provide the strict amplitude margin
needed at the right boundary.

For either profile, define its residual in the
time-dependent radial equation by
\[
 \mathcal R_\pm
 =-V_\pm''-d(t,z)(V_\pm'-V_\pm)+\e^{-z}V_\pm^2
 =-d(t,z)(V_\pm'-V_\pm).
\]
The wave tails give
\[
 |V_\pm'(z)-V_\pm(z)|\le C(1+z_+),
\]
and hence
\[
 |\mathcal R_\pm|\le Ct^{\gamma-1}
 \qquad(|z|\le\ell(t)).
\]

We estimate the two positive errors
\[
 H_+=(v-V_+)_+,\qquad H_-=(V_--v)_+.
\]
The quadratic divided difference is
$\e^{-z}(v+V_\pm)\ge0$. The positive-part inequality
therefore yields, in the weak sense,
\[
 \mathcal L H_\pm\le Ct^{\gamma-1},
 \qquad
 \mathcal L=\partial_t-\partial_{zz}
             -d(t,z)\partial_z+d(t,z).
\]
Here the nonnegative quadratic term has been discarded;
the zeroth-order term $d(t,z)$ is retained.

We now construct a common error barrier. Set
\[
 \phi(t,z)=2-\frac{z^2}{\ell(t)^2},
 \qquad 1\le\phi\le2,
\]
and, for a sufficiently large starting time $T$, let
\[
 B(t,z)=
 \left[
 C_1t^{3\gamma-1}
 +C_2\exp\{-c(t^{1-2\gamma}-T^{1-2\gamma})\}
 \right]\phi(t,z).
\]
The constants may depend on the seed, $\kappa$, and
$\epsilon$.

To verify the barrier inequality, note that
\[
 -\phi_{zz}=\frac2{\ell(t)^2},\qquad
 \phi_t=\frac{2\gamma z^2}{t\ell(t)^2}\ge0,
 \qquad |\phi_z|\le\frac2{\ell(t)}.
\]
Since $|d|\le C/t$ and $\ell(t)^{-2}\asymp t^{-2\gamma}$,
the positive second-derivative contribution dominates
the drift and zeroth-order errors for large $T$.
The derivative of $t^{3\gamma-1}$ costs only a multiple
of $t^{-1}$ relative to that factor. The derivative of
the exponential factor costs
$c(1-2\gamma)t^{-2\gamma}$.
Choosing $c>0$ sufficiently small and then $T$
sufficiently large gives
\[
 \mathcal L B
 \ge c_0\ell(t)^{-2}
 \left[
 C_1t^{3\gamma-1}
 +C_2\exp\{-c(t^{1-2\gamma}-T^{1-2\gamma})\}
 \right].
\]
Increasing $C_1$ therefore ensures
$\mathcal L B\ge Ct^{\gamma-1}$.

At $z=\ell(t)$, both positive errors vanish.
At $z=-\ell(t)$, the bounds $0\le u,U_{\log}\le1$
give
\[
 H_\pm(t,-\ell(t))\le\e^{-\ell(t)},
\]
which is dominated by the polynomial part of $B$ for
large $T$. Finally, choose $C_2$ large enough to dominate
both errors at time $T$. Comparison on each finite
moving cylinder, followed by extension in time, yields
\[
 V_-(z)-B(t,z)\le v(t,z)\le V_+(z)+B(t,z),
 \qquad |z|\le\ell(t).
\]

Because $\gamma<1/3$, both terms in the coefficient
of $B$ tend to zero. Thus the errors relative to the
two shifted profiles vanish uniformly on the expanding
interval. In particular, the polynomial error is
$O(t^{3\gamma-1})$.

For fixed $\epsilon>0$, positivity and the right-hand
tail of $U_{\log}$ imply
\[
 c_\epsilon:=
 \inf_{z\ge0}\frac{V_-(z)}{1+z}>0.
\]
For sufficiently large $t$, the error bound therefore
gives
\[
 v(t,z)\ge\frac{c_\epsilon}{2}(1+z),
 \qquad 0\le z\le\ell(t).
\]
Returning to $w=u/C_f$, we obtain
\[
 w(t,m(t)+z)\ge c(1+z)\e^{-z},
 \qquad 0\le z\le\kappa t^\gamma.
\]
Taking $\kappa=2$ supplies the range required in
\eqref{eq:sharp-lower}. This argument uses only scalar
logistic leading-edge estimates and does not invoke
inner matching for a general reaction.

\end{document}